\documentclass[12pt,a4paper]{amsart}
\usepackage{amsthm,amsmath,amssymb,amscd,graphics,enumerate,xspace,verbatim,url,stmaryrd,epic, eepic}
\usepackage{mathtools}
\usepackage{fullpage}

\usepackage[usenames,dvipsnames]{color}

\usepackage[colorlinks=true]{hyperref}

\usepackage{MnSymbol}
\usepackage[active]{srcltx}
\usepackage[all]{xypic}
\SelectTips{cm}{}

{
   \newtheorem{theorem}[subsubsection]{Theorem}
      \newtheorem*{theorem*}{Theorem}
   \newtheorem{proposition}[subsubsection]{Proposition}

   \newtheorem{lemma}[subsubsection]{Lemma}
    \newtheorem{Inductiveclaim}[subsubsection]{Inductive Claim}
   \newtheorem*{claim}{Claim}

   \newtheorem{corollary}[subsubsection]{Corollary}
   
   \newtheorem*{conjecture*}{Conjecture}

}
{\theoremstyle{definition}
          \newtheorem*{exercise*}{Exercise}
   
   \newtheorem{example}[subsubsection]{Example}
   \newtheorem*{example*}{Example}
   \newtheorem{definition}[subsubsection]{Definition}
   
   \newtheorem*{definition*}{Definition}
   
   \newtheorem{remark}[subsubsection]{Remark}
   
   \newtheorem{notation}[subsubsection]{Notation}

\newtheorem{warning}[subsubsection]{Warning}
}
\newcommand{\CC}{{\mathbb{C}}}
\newcommand{\QQ}{{\mathbb{Q}}}
\newcommand{\NN}{{\mathbb{N}}}

\newcommand{\PP}{{\mathbb{P}}}
\newcommand{\ZZ}{{\mathbb{Z}}}

\newcommand{\GG}{{\mathbb{G}}}

\renewcommand{\AA}{{\mathbb{A}}}

\newcommand{\bmu}{{\boldsymbol{\mu}}}

\newcommand{\cA}{{\mathcal A}}

\newcommand{\cC}{{\mathcal C}}
\renewcommand{\cD}{{\mathcal D}}

\newcommand{\cF}{{\mathcal F}}
\newcommand{\cG}{{\mathcal G}}
\renewcommand{\cH}{{\mathcal H}}
\newcommand{\cI}{{\mathcal I}}

\newcommand{\cK}{{\mathcal K}}

\newcommand{\cO}{{\mathcal O}}

\renewcommand{\cR}{{\mathcal R}}

\def\<{\langle}
\def\>{\rangle}

\newcommand{\Int}{{\operatorname{Int}}}

\newcommand{\Spec}{\operatorname{Spec}}
\newcommand{\Span}{\operatorname{Span}}

\newcommand{\Proj}{\operatorname{Proj}}

\newcommand{\cHom}{{{\cH}om}}

\newcommand{\codim}{\operatorname{codim}}

\newcommand{\lcm}{{\operatorname{lcm}}}

\newcommand{\satpull}[1]{#1^*_{\sat}}
\newcommand{\dispull}[1]{#1^*_{\operatorname{dis}}}
\newcommand{\folpull}[1]{#1^*_{\operatorname{fol}}}

\def\:{{\colon}}
\def\.{{,\dots,}}
\def\dim{{\rm dim}}

\def\inv{{\rm inv}}

\def\-inv{{\operatorname{-inv}}}

\newcommand\lb{\llbracket}
\newcommand\rb{\rrbracket}

\newcommand{\double}{\genfrac..{0pt}1
{\raise -1pt\hbox{$\scriptstyle\longrightarrow$}}{\raise 3pt\hbox
{$\scriptstyle\longrightarrow$}}}

\renewcommand{\setminus}{\smallsetminus}

\def\sat{{\rm sat}}

\def\tototi{\mathbin{\mathop{\otimes}\limits^{\raise-1pt\hbox
{$\scriptscriptstyle {\rm L}$}}}}

\def\indlim{\mathop{\vrule width0pt height7pt depth
4pt\smash{\lim\limits_{\raise 1pt\hbox to 14.5pt
{\rightarrowfill}}}}}
\def\projlim{\mathop{\vrule width0pt height7pt depth
4pt\smash{\lim\limits_{\raise 1pt\hbox to 14.5pt
{\leftarrowfill}}}}}

\newcommand\displaceamount{3pt}

\newcommand{\doubledown}{\ar@<\displaceamount>[d]\ar@<-\displaceamount>[d]}

\newcommand{\doubleup}{\ar@<\displaceamount>[u]\ar@<-\displaceamount>[u]}

\newcommand{\doubleright}{\ar@<\displaceamount>[r]\ar@<-\displaceamount>[r]}

\newcommand{\spa}{{\operatorname{span}}}

\newcommand{\nsing}{{\operatorname{ns}}}

\newcommand{\inte}{{\operatorname{int}}}
\newcommand{\ext}{{\operatorname{ext}}}
\newcommand{\ord}{{\operatorname{ord}}}

\def\rank{{\operatorname{rank}}}

\def\supp{{\operatorname{supp}}}

\newcommand\pa{{\mathfrak a}}

\newcommand\pb{{\mathfrak b}}

\def\nsrank{{\operatorname{sm-rank}}}
\newcommand\Sm{{smooth}}

\begin{document}

\title{Principalization on logarithmically foliated orbifolds}

\author[D. Abramovich]{Dan Abramovich}
\address{Department of Mathematics, Box 1917, Brown University,
Providence, RI, 02912, U.S.A}
\email{dan\_abramovich@brown.edu}

\author[A. Belotto da Silva]{Andr\'e Belotto da Silva}
\address{Universit\'e Paris Cit\'e and Sorbonne Universit\'e, UFR de Math\'ematiques, Institut de Math\'ematiques de Jussieu-Paris Rive Gauche, UMR7586,
F-75013 Paris, France.
Institut universitaire de France (IUF).}
\email{belotto@imj-prg.fr}

\author[M. Temkin]{Michael Temkin}
\address{Einstein Institute of Mathematics\\
               The Hebrew University of Jerusalem\\
                Edmond J. Safra Campus, Giv'at Ram, Jerusalem, 91904, Israel}
\email{temkin@math.huji.ac.il}

\author[J. W{\l}odarczyk] {Jaros{\l}aw W{\l}odarczyk}
\address{Department of Mathematics, Purdue University\\
150 N. University Street \\ West Lafayette, IN 47907-2067}
\email{wlodar@math.purdue.edu}

\thanks{This research is supported by BSF grants 2018193, 2022230, ERC Consolidator Grant 770922 - BirNonArchGeom, NSF grants
DMS-2100548, DMS-2401358, the Plan d'Investissements France 2030, IDEX UP ANR-18-IDEX-0001, and Simon foundation grants MPS-SFM-00006274, MPS-TSM-00008103. }


\date{\today}

\begin{abstract}
In characteristic zero, we  construct principalization of ideals on  smooth orbifolds endowed with a normal crossings divisor and a foliation. We then illustrate how the method can be used in the general study of foliations via two applications. First, we provide a resolution of singularities of Darboux totally integrable foliations in arbitrary dimensions --- including  rational and meromorphic Darboux foliations. Second, we show how to transform a generically transverse section into a transverse section. \end{abstract}

\maketitle

\setcounter{tocdepth}{1}
\tableofcontents 

\section{Introduction}
\emph{Resolution of singularities} encompasses some of the most vexing problems in algebraic geometry. Apart from the problem of resolution of singularities of schemes in positive  and mixed characteristics, major open problems, even in characteristic 0,  have included resolution of singularities of morphisms and of differential equations. The present work addresses first-order linear differential equations, classically described using singular \emph{foliations} or \emph{involutive distributions}. We have been involved in resolution of morphisms in the past, see \cite{ATW-relative,BelBmon}, and hope the present work will help address remaining question in that direction as well.

Much of the work on resolution of singularities of varieties proceeds through a classical process of embedding a singular variety $X$ in a smooth variety $Y$, and then \emph{principalizing} the ideal of $X$. We accordingly address here, as a first step, the problem of principalization of an ideal on an ambient space carrying a mildly singular foliation. The work \cite[Page 5]{Belthesis}, \cite{BelRACSAM,BelJA} of Belotto da Silva introduced this problem and solved it in a number of cases:

\begin{enumerate}
\item The case of an $\cF$-invariant ideal: $\cF(\cI)= \cI$;
\item The case of a foliation by curves; and
\item The local uniformization principalization problem, that is, the general case, but the resolution is only given by local blow-ups.
\end{enumerate}

Thus our present work generalizes the work \cite{BelRACSAM,BelJA} to the general case, see Theorem \ref{thm:PrincipalizationFoliated} --- always in characteristic $0$. To this end, we combine the methods developed by Abramovich, Temkin and W{\l}odarczyk \cite{ATW-principalization,ATW-relative,ATW-weighted} concerning logarithmic and weighted resolution of varieties and morphisms, of Belotto \cite{BelRACSAM,BelJA,BelRI} concerning resolution in the presence of a foliation, and of W{\l}odarczyk \cite{Wlodarczyk-cobordant} concerning resolution of singularities via Rees algebras and cobordant blow-ups. We then illustrate how our result can be useful in at least two different situations. First, we provide an application to the study of resolution of singularities of foliations, by settling it in the case of Darboux (including rational or  meromorphic) totally integrable foliations, see Theorem \ref{thm:ResDarbouxTotallyIntegrable}. Second, we prove resolution of singularities under the constraint of preserving transverse sections, see Theorem \ref{thm:ResTransverseSections}, see \cite{BelPan}.

\subsection{Principalization in the presence of a foliation}\label{ssec:IntroResolutionFoliated}

Our main result relies on the notions of foliated logarithmic varieties, $\cI$-admissible centers, $\cK$-monomial foliations, $\cF$-aligned centers and strict and controlled transforms foliations. These notions are described in detail in the paper. Here we only explain their meaning, discussing these notions mostly in the algebraic setup, the analytic analogue being similar.

\smallskip
\noindent
\emph{Foliated logarithmic varieties.} By a foliated logarithmic variety, we mean a triple $(X,\cF,E)$ where $X$ is an algebraic variety --- always in characteristic 0 --- $E$ is a simple normal crossing divisors and $\cF$ is an involutive (not necessarily saturated, see Remark \ref{rk:LiteratureDefFoliations}) coherent sub-sheaf $\cF\subset \cD^{\log}_X$ of the sheaf of logarithmic derivations $\cD^{\log}_X$ with respect to $E$. See details in Section~\ref{ssec:2FoliatedLog}. We also consider the induced notion of foliated logarithmic \emph{orbifolds} $(X,\cF,E)$, discussed in Section \ref{sssec:Orbifolds}. 

\smallskip
\noindent
\emph{$\cK$-monomial foliations.} Let $\cK$ be a sub-field of $\CC$. The notion of a foliation $\cF$ being a $\cK$-monomial foliation was defined in \cite[Section~3]{BelRACSAM}. This notion guarantees that the foliation has everywhere formal first integrals which are monomial, with exponents in $\cK$, in terms of local parameters, including the parameters of the logarithmic structure, see Definition \ref{def:KMonomialFoliation}. These are the  final singularities for rational or Darboux totally integrable foliations, see details in Section~\ref{ssec:ResTotallyIntegrable} and Theorem~\ref{thm:ResDarbouxTotallyIntegrable}, which answers \cite[Problem 1.2]{BelRI}.

\smallskip
\noindent
\emph{$\cI$-admissible centers.} Let $\cI \subset \cO_X$ be an ideal sheaf. In \cite{ATW-weighted} one discusses the property of a center of the form ${J}=(x_1^{a_1},\dots, x_k^{a_k})$ being \emph{$\cI$-admissible}. A generalization accounting for logarithmic parameters was introduced in Quek's \cite{Quek} and an equivalent treatment in the language of rationally graded Rees algebras is given  in \cite{Wlodarczyk-cobordant},  see Sections \ref{ssec:CanonicalInvariant}, \ref{Sec:transforms-Rees}. The definition of admissibility remains fundamentally the same.
 
\smallskip
\noindent
\emph{$\cF$-aligned centers.} In \cite[Section~4]{BelRACSAM} one discusses regular centers of blowing up which are \emph{admissible} with respect to a foliation $\cF$. To avoid confusion with admissibility of a center for an ideal sheaf, we prefer to call these \emph{$\cF$-aligned} instead, see Definition \ref{def:Faligned}. Heuristically, these are centers which are ``as close as possible" to having normal crossings with the foliation --- they can be ``decomposed" into $\cF$-transverse and $\cF$-invariant parts.

\smallskip
\noindent
\emph{Weighted blow-ups and cobordant blow-ups.}  To a center ${J}$ one associates its \emph{stack theoretic weighted blow-up} $ Bl_{J}(X) \to X$, the central operation of our paper, see Section \ref{sec:BlowUp}. It is best decribed via its Cox construction: the spectrum of the extended Rees algebra of ${J}$ is its \emph{weighted blow-up cobordism} $B_{J} \to X$. It admits a \emph{vertex} $V_{J}$, and the complement $B_{{J}+} = B_{J} \setminus V_{J}$ is called \emph{the cobordant blow-up} of ${J}$. Grading provides a $\GG_m$-action, and one sets $Bl_{J}(X) = [ B_{{J}+}/\GG_m]$, the stack-theoretic quotient. 

As $B_{{J}+}$ provides a presentation of the blow-up, we may use it to study the properties of the blow-up, staying in the category of schemes with group actions rather than stacks. Quite remarkably, transforms of ideals and of foliations are best described on $B_{{J}+}$.

\smallskip
\noindent
\emph{Strict and controlled transform of $\cF$.} Given a center ${J}$ which is $\cI$-admissible and $\cF$-aligned, we consider weighted (or cobordant) blow-ups between triples $\sigma : (X',\cF',E') \to (X,\cF,E)$, where $X' = Bl_{J}(X)$ (or $X' =B_{{J}+}$), $E'$ is the union of the total transform of $E$ with the exceptional divisor of $\sigma$, and $\cF$ stands for either the \emph{strict} or \emph{controlled} transform of $\cF$, see Definition \ref{def:TransformFoliation} (or  \ref{Def:split-transforms}) --- the choice will be made precise each time. In birational geometry, it is common to use the strict transform of foliations; but, as observed in \cite{Belthesis}, it does not behave well when working with ideals.

It is shown in \cite[Section~4]{BelRACSAM} that the blowing up of a regular $\cF$-aligned center preserves the property of a foliation being $\cK$-monomial. 
  Here we have to generalize $\cF$-aligned centers and their blow-up properties to centers of weighted blow-ups, see Theorem \ref{Th:aligned center}. More is true: the transform of a logarithmically smooth foliation remains logarithmically smooth. This is manifested exceptionally  beautifully on the cobordant blow-up, where the split transform of a \emph{smooth} foliation remains \emph{smooth}.

\subsection{Main results: principalization and embedded resolution} Let $X$ be either a smooth irreducible variety over a field $K$ of characteristic $0$, or a coherent complex analytic space. With the notions above in place, our main results are:

\begin{theorem}[Principalization over foliated manifolds]\label{thm:PrincipalizationFoliated}
Let $(X,\cF,E)$ be a smooth foliated logarithmic variety and $\cI$ be a coherent ideal sheaf over $X$.
If $X$ is algebraic, let $X_0=X$; otherwise, let $X_0 \subset X$ be a relatively compact open set, and let $(X_0,\cF_0,E_0)$ be the restricted triple. There exists a sequence of weighted (or cobordant) blow-ups
\[
(X_0,\cF_0,E_0)\stackrel{\sigma_0}\leftarrow (X_1,\cF_1,E_1)\stackrel{\sigma_{1}}\leftarrow\cdots\stackrel{\sigma_{k-1}}\leftarrow (X_k,\cF_k,E_k)=(X',\cF',E'),
\]
where $\cF_{i+1}$ is the controlled (resp., strict) transform of $\cF_i$ by $\sigma_i$ such that:
\begin{enumerate}
\item\label{It:principalization} The total transform $\cI' = \cI \cdot \cO_{X'}$ (resp., strict transform, see Definition \ref{def:StrictTransformReesAlgebra}) is locally principal and monomial (resp., $\cI'=\cO_{X'}$).
\item\label{It:is-aligned} The center ${J}_i$ of the blow-up $\sigma_i$ is $\cF_i$-aligned and $\cI_i$-admissible, where $\cI_i$ is the controlled (resp., strict) transform of $\cI_0$, see Definitions \ref{def:FadmissibleCenter},
\ref{def:ControlledTransformReesAlgebra} and \ref{def:StrictTransformReesAlgebra}.

 \item\label{It:non-singular} If $\cF$ is logarithmically smooth then each $\cF_i$ is logarithmically smooth. If $\cF$ is $\cK$-monomial  then each $\cF_i$ is $\cK$-monomial. If $\cF$ is \Sm, then for the cobordant principalization only, each $\cF_{i}$ is also \Sm.

\item\label{It:functorial} The process is functorial for field extensions   and for smooth morphisms up to trivial redundant blow-ups (or trivial cobordant blow-ups), and equivariant for
   group actions and derivations for which $\cF$ and $\cI$ are invariant, 
see Definitions \ref{def:fieldextension}, \ref{def:MorphismLogFoliatedVariety}, \ref{def:GsemiinvariantFol} and \ref{def:deltainvariantFol}.

\item\label{It:is-F-invariant} If $\cI$ is $\cF$-invariant, that is $\cF(\cI)\subset \cI$, then all centers ${J}_i$ are $\cF_i$-invariant.

 \item\label{It:non-compact} If $(X,\cF,E)$ is a  coherent analytic space then there is a proper bimeromorphic map $(X,\cF,E)\leftarrow (X',\cF',E')$ providing a principalization of $\cI$ such that, in a neighborhood $X_0$ of any compact set $Z\subset X$, it factors  into finitely many weighted blowings up satisfying properties (1)-(5) above. 
\end{enumerate}
\end{theorem}

For the next result, we denote by $\cI := \cI^{Y}$ the reduced coherent ideal sheaf whose zero locus is a subvariety $Y \subset X$.

\begin{theorem}[Embedded desingularization of a subvariety] \label{thm:EmbeddedDesingularization} Let $(X,\cF,E)$ be a smooth foliated logarithmic variety and $Y \subset X$ be a subvariety. Denote by $(X_0,\cF_0,E_0)$ the restricted triple,  as in Theorem \ref{thm:PrincipalizationFoliated}, and $Y_0 := Y \cap X_0$. There exists a sequence of weighted (or cobordant) blow-ups
\[
(X_0,\cF_0,E_0)\stackrel{\sigma_0}\leftarrow (X_1,\cF_1,E_1)\stackrel{\sigma_{1}}\leftarrow\cdots\stackrel{\sigma_{k-1}}\leftarrow (X_k,\cF_k,E_k)=(X',\cF',E'),
\]
where $\cF_{i+1}$ is the strict transform of $\cF_i$, such that:
\begin{itemize}
\item[(1')] The strict transform $Y'$ of $Y_0 $ is smooth, has normal crossings with $E'$, and it is itself an $\cF'$-aligned center.  
\item[(2')] The center ${J}_i$ of the blow-up $\sigma_i$ is $\cF_i$-aligned and $\cI_i:=\cI^{Y_i}$-admissible, where $Y_i$ is the strict transform of $Y_0$.

\item[(\ref{It:non-singular}'-\ref{It:non-compact}')] Properties \eqref{It:non-singular}-\eqref{It:non-compact} from Theorem \ref{thm:PrincipalizationFoliated} hold true. \end{itemize}
\end{theorem}

These results are foliated versions of the main results from \cite{ABQTW,ATW-weighted,Wlodarczyk-cobordant}. They also include the construction of an invariant $\inv_{\cF}$ for ideal sheaves, extending the invariant $\inv$ from \cite{ATW-weighted}, see Theorems \ref{Th:semicontinuity}, \ref{Th:the center} and \ref{Th:principalization}. Our proofs follow several of the ideas of these papers --- but they remain self-contained and no knowledge of the previously cited papers is necessary. In fact, by taking $\cF = \cD_X$ or $\cD_{X}^{\log}$, we recover some of their main results as well.

\begin{remark}[On more general foliations]\label{rk:LogCanonical}
We expect that Theorem \ref{thm:PrincipalizationFoliated}\eqref{It:non-singular} can be strengthened to include much more general classes of foliations; the key is to prove that $\cF$-aligned centers preserve the desired class of foliations, see footnote \ref{Foot:thick} on page \pageref{Foot:thick}.
For instance, in \cite{BelJA} it is proved that non-weighted, analytic $\cF$-aligned blow-ups preserve the property of a foliation being \emph{elementary}, that is, a saturated rank-$1$ log-canonical foliation \cite[I.ii.4]{McQPan}. In particular, a version of Theorem \ref{thm:PrincipalizationFoliated}\eqref{It:non-singular} where we replace  ``$\mathcal{K}$-monomial" by ``elementary'' can be deduced, see \cite{BelJA}. Such arguments are however orthogonal to the present paper.
\end{remark}

\subsubsection{Relation to resolution of morphisms.}
Taking a different perspective, the problem of resolution of foliated varieties is closely related to the problem of resolution or monomialization of logarithmic morphisms \cite{ATW-principalization,ATW-relative,BelBmon}. In such a case the foliation $\cF$ associated with logarithmically smooth morphism $X\to B$ is simply given by the sheaf  $\cF=\cD_{X/B}^{\log}$ of relative logarithmic derivations. Consequently the algorithm of princpalization of ideals $\cI$ on foliated varieties should be thought of as the natural extension of the relative principalization of ideals on $\cI$ on $X$ \emph{over $B$}.

One should stress however two key differences between the two problems. First, in the relative logarithmic resolution one needs to keep the morphism $X\to B$ logarithmically smooth during the process. This has consequences even to the arguably simple situation of an $\cF$-invariant ideal: one is led to consider modifications on the base $B$, which induce modifications on the space $X$. In the foliated case the whole resolution process is done on $X$, the base $B$ implicitly replaced by the highly non-separated object $[X/\cF]$ classifying leaves. Second, foliations on logarithmic varieties are more general than those defined by logarithmically smooth morphisms, and they open the prospect of applications to differential geometry and differential equations.

\subsection{Resolution of classes of foliations}\label{sec:IntroResTotallyIntegrable}

Resolution of singularities of foliations is a problem which goes back to the beginning of the twentieth century \cite{Ben} and remains quite mysterious in high dimensions. In fact, the best results to date are either limited to foliations in threefolds \cite{Ben,Sei,Pan,McQPan,Can,CanCer,CRS,RReis,posva}; or are local in nature \cite{BelRI,CCS,CanoDuque}. 

As is well-known in birational geometry of foliations, it is usually impossible to find a birational model of a foliation which is non-singular. This is analogous to the case of morphisms: it is usually impossible to find a birational model of a morphism which is a smooth morphism. In the case of \emph{morphisms} of varieties in characteristic 0, it is shown in \cite{AK} and \cite{ATW-relative} that one can find a \emph{logarithmically smooth} birational model. As the vector field 
\begin{equation}\label{eq:Euler}
(x+y) \frac{\partial}{\partial x} + y^2 \frac{\partial}{\partial y}
\end{equation}
from \cite[Example 7.10]{IY} shows, even this is impossible for general foliations, already for a rank-1 foliation on a surface: it can  be brought to a log smooth form only on formal completion.\footnote{\label{foot:formal-cov}As shown in  \cite[Example 7.10]{IY}, the change of variables $(x_1,y_1) = (x+\sum_{k=1}^\infty (k-1)! y^k, y)$ brings it to the form $x_1\partial_{x_1} + y_1^2\partial_{y_1}$, which is log smooth for the formal divisor $x_1=0$.}

Thus  the goal of resolution of singularities of foliations is to find birational models having a form  which  is ``as simple as possible". In fact, some references prefer the term ``reduction of singularities" instead of resolution of singularities, eg. \cite{CRS}.

In analogy to the case of varieties and of morphisms, where principalization implies resolution through embedding \cite{ATW-relative}, one naturally asks
\begin{quote}
can one use principaliation to ``resolve" a singular foliation?
\end{quote}
We argue in this paper that the answer should be yes in many cases. 
The immediate cases are those where the foliation admits rational or meromorphic first integrals, since the problem can be viewed as a version of resolution of morphisms, see Definition \ref{Def:totally integrable} below. More examples, including with transcendental leaves, are given below. But first, we provide a framework that allows to treat a range of cases at the same time.

\begin{definition}[Thick Classes]\label{def:thick}
We say that a class $\cC$ of foliated logarithmic orbifolds is \emph{thick} if for every element $(Y,\cF_Y,E_Y)$ of $\cC$, the foliation $\cF_Y$ is saturated and:
\begin{enumerate} 
\item\label{defItem:thickProduct} If $(X,  E_X)$ is a log smooth manifold, then $(Z,\cF_Z,E_Z) \in \cC$ where
\[
Z= X \times Y, \quad \cF_Z =   \pi_X^*\cD_X^{\log}\oplus \pi_Y^* \cF_Y, \quad E_Z = \pi^{-1}_X(E_X) \cup \pi^{-1}_{Y}(E_Y).
\]
\item\label{defItem:thickFaligned} If $\phi: (Y',\cF',E') \to (Y,\cF_Y,E_Y)$ is an $\cF$-aligned blow-up, then $(Y',\cF',E') \in \cC$; and
\item\label{defItem:thickRestriction} If $X\subset Y$ is an $\cF$-aligned center such that $X \not\subset E$, then $(X, \cF_X, E_X) \in \cC$, where $\cF_X = \cF_{ \boldsymbol | \boldsymbol | X}$ is the saturated restriction, see Definition \ref{Def:saturated-pullback}\eqref{DefItem:saturated-pullback-Restriction}, and $E_X = E_{|X}$.
\end{enumerate}
\end{definition}

A few examples of thick classes are as follows:
\begin{theorem}\label{thm:ExamplesThickClasses}
The classes of  foliated logarithmic orbifolds $(Y,\cF,E)$ where $\cF$ is 
\begin{enumerate}
\item log-smooth, Definition \ref{Def:foliations}(\ref{ItDef:LogSmooth})

 or 
\item
 $\cK$-monomial, Definition \ref{def:KMonomialFoliation},

  \end{enumerate}
are thick.

\end{theorem}

The above result is proved as Theorems \ref{Th:aligned center log-smooth} and \ref{Th:aligned center new}.\footnote{Several other classes were suggested by P. Cascini, J. V. Pereira, and S. Spicer.\label{Foot:thick}}

By combining Theorem \ref{thm:EmbeddedDesingularization} with an idea developed in \cite[Theorem 1.2.12]{ATW-relative}, we get:

\begin{theorem}[Thick Reduction of singularities]\label{thm:ResThickClasses}
Let $\cC$ be a thick class. Let $(B,\cG,E_B) \in \cC$ and consider a rational (or meromorphic) morphism $\phi: (X,E) \dashedrightarrow (B,E_B)$ whose graph projects properly over $X$ and $\phi^{-1}(U)$ is dense in $X$, where $U$ is the regular locus of $\cG$\footnote{The latter property is always satisfied for dominant morphisms; c.f. Lemma \ref{lem:SatPullBack}.}. 

Consider the triple $(X,\cF,E)$, where $\mathcal{F} = \satpull{\phi}(\cG)$ is the saturated pullback of $\cG$ by $\phi$, see Definition \ref{Def:saturated-pullback}. There exists a proper birational or bimeromorphic morphism $\tau: (X',\cF',E') \to (X,\cF,E)$, where $\cF'$ is the strict transform of $\cF$, such that $(X',\cF',E') \in \cC$.
 
The process is functorial for field extensions and smooth morphisms with respect to the pair $(\cF,\phi)$.
\end{theorem}

The above result is proved in Section \ref{ssec:Proofthm:ResTotallyIntegrable}. We use the generality of meromorphic maps  because in foliation theory this is sometimes important. For instance the dicritical case is often harder, see  \cite{Cano-Cerveau, Cano}. In our situation we use elimination of indeterminacies and reduce to the holomorphic case.

We illustrate Theorem \ref{thm:ResThickClasses} with some familiar cases, starting with the following:

\begin{definition}[Globally totally integrable foliations]\label{Def:totally integrable} We say that a foliation $\cF$ on a logarithmic variety $X$ is  a \emph{globally rationally (or meromorphically) totally integrable foliation} if there exists a rational (or meromorphic) morphism $\varphi: X \dashedrightarrow B$ whose graph projects properly over $X$, such that 
\[
\cF = \cD_{X/B}^{\log}, \text{ that is } \cF \cdot \cO_{X,\pa} = \spa_{\cO_{X,\pa}}\{ \partial \in \cD_X^{\log};\, \partial( f\circ \varphi) \equiv 0,\, \forall f\in \cO_{B,\varphi(\pa)}\},
\]
for all $\pa\in X$ where $\varphi$ is well-defined. 
\end{definition}

In other words, a totally integrable foliation is the pullback, in the sense appearing in Theorem \ref{thm:ResThickClasses}, of the zero foliation, which is logarithmically smooth and monomial. Thanks to Corollary \ref{cor:LogSmoothnessSmooth}, in this case we can say a bit more than Theorem \ref{thm:ResThickClasses}:

\begin{theorem}[Resolution of rationally totally integrable foliations]\label{thm:ResTotallyIntegrable}
Let $(X,\cF,E)$ be a smooth foliated logarithmic variety. Suppose that $\cF$ is a globally rationally (or meromorphically) totally integrable foliation.

There exists a proper birational or  bimeromorphic morphism $\tau: (X',\cF',E') \to (X,\cF,E)$, where $X'$ is an orbifold, $\cF'$ is the strict transform of $\cF$ and $E'$ is a SNC divisor, such that $\cF'$ is a $\mathbb{Q}$-monomial and logarithmically smooth foliation.

Moreover, the split transform $\cF''$ under the corresponding cobordant resolution $\sigma:(X'', \cF'', E'') \to (X, \cF, E)$,  is in fact \emph{\Sm}.

 The process is functorial for field extensions and smooth morphisms with respect to the pair $(\cF,\varphi)$.

\end{theorem}

We simultaneously provide a similar result for \emph{globally Darboux-totally integrable foliations}:

\begin{definition}[Globally $\mathcal{K}$-Darboux totally integrable foliations]\label{def:DarbouxTI}
Let $(X,\cF,E)$ be a foliated logarithmic variety and let $\mathbb{Q} \subset \mathcal{K} \subset K$ be a field. We say that the foliation $\cF$ is \emph{globally $\mathcal{K}$-Darboux totally integrable} if there exists a smooth variety $B$ and a  rational (or meromorphic) morphism $\phi: X \dashedrightarrow B$ whose graph projects properly over $X$, and a $\mathcal{K}$-monomial foliation $\mathcal{G}$  on $B$ such that $\mathcal{F} = \satpull{\phi}(\cG)$ is the saturated pullback of $\cG$ by $\phi$, see Definition \ref{Def:saturated-pullback}, and $\phi^{-1}(U)$ is dense in $X$, where $U$ is the regular locus of $\cG$.
\end{definition}

In other words, a Darboux totally integrable foliation is the pullback, in the sense appearing in Theorem \ref{thm:ResThickClasses}, of a monomial foliation. The usual definition of complex analytic Darboux totally integrable foliation is explicitly given in terms of locally defined multivalued functions, see e.g. \cite[Page 968]{BelRI} and references therein; we show that these definitions are consistent in Section \ref{rk:BasicKmonomialDarboux} below. Thanks to Theorem \ref{thm:LogSmoothnessMonomial}, in this case we can once again say a bit more than Theorem \ref{thm:ResThickClasses}:

\begin{theorem}[Resolution of Darboux  totally integral foliations]\label{thm:ResDarbouxTotallyIntegrable}
Let $(X,\cF,E)$ be a smooth foliated logarithmic variety.  
Suppose that $\cF$ is a globally $\mathcal{K}$-Darboux totally integrable foliation.

There exists a proper  birational or bimeromorphic morphism $\tau: (X',\cF',E') \to (X,\cF,E)$, where $X'$ is an orbifold, $\cF'$ is the strict transform of $\cF$ and $E'$ is a SNC divisor, such that $\cF'$ is a $\cK$-monomial and logarithmically smooth foliation.

Moreover, in the corresponding cobordant resolution $\sigma:(X'', \cF'', E'') \to (X, \cF, E)$, the foliation $\cF''$ is in fact \emph{\Sm}.

 The process is functorial for field extensions and smooth morphisms with respect to the pair $(\cF,\varphi)$.
\end{theorem}

Thus, if one accepts non-birational operations, such as the cobordant blow-up, a more beautiful picture is revealed: the \emph{split} transform under a sequence of \emph{cobordant} blow-ups of a Darboux foliation is in fact \emph{\Sm.}\footnote{This is a manifestation of the fact that the sheaf $\cD_{B_+/\AA^1}$ of \emph{relative} derivations is an equivariant model of the sheaf $\cD^{\log}_{Bl(X)}$ of \emph{logarithmic} derivation, when the logarithmic structure is given by the exceptional divisor. See also Corollary \ref{cor:LogSmoothnessSmooth}.}

\subsection{Resolution preserving transverse locus}\label{sec:IntroResTransverseLocus}
We now illustrate how our result can be used to study geometric properties of subvarieties. Let us recall the notion of \emph{transverse section} to a foliation:

\begin{definition}[Transverse section]\label{def:TransverseSection}
A subvariety $Y \subset X$ is said to be {\it transverse} to a 
 foliation $\cF$ at a point $\pa \in Y$ if there exists a partial system of parameters $(x_1,\ldots,x_p)$ centered at $\pa$ and derivations $\partial_1,\ldots,\partial_p \in \cF \cdot \cO_{X,\pa}$ such that: (i) the ideal $\cI^{Y}$ of $Y$ is locally equal to $(x_1,\ldots,x_p)$; and (ii) $\det(\partial_i(x_j))_{i,j}$ is a unit. We say that $Y$ is  transverse to  $\cF$ if it is everywhere transverse to $\cF$. If additionally the rank of $\cF$ is equal to $p$ so that $\cF_p=\spa_{{\cO}_{X,\pa}}(\partial_1,\ldots,\partial_p)$ then we shall call $Y \subset X$ the {\it transverse section} of $\cF$ at $\pa$.
\end{definition}

It is impossible to overstate the importance of transverse sections in foliation theory. In many applications, natural choices of sub-varieties $Y \subset X$ are only \emph{generically} transverse sections, see eg. \cite{Mattei91,BelPan}. We may use Theorem \ref{thm:PrincipalizationFoliated} to transform them into transverse sections without modifying the points where $Y$ is already transverse to $\cF$:

\begin{theorem}[Resolution preserving transverse locus]\label{thm:ResTransverseSections}
Let $(X,\cF,E)$ be a smooth foliated logarithmic variety. Let $Y\subset X$ be a closed subvariety, and denote by $Y^{tr} \subset Y$ the set of points where $Y$ is transverse to $\cF$. Denote by $(X_0,\cF_0,E_0)$ the restricted triple, just as in Theorem \ref{thm:PrincipalizationFoliated}, and $Y_0 := Y \cap X_0$. Assume that $Y^{tr}\cap Y_0\subset Y_0$ is dense. There exists a sequence of weighted (or cobordant) blow-ups
\[
(X_0,\cF_0,E_0)\stackrel{\sigma_0}\leftarrow \cdots\stackrel{\sigma_{k-1}}\leftarrow (X_k,\cF_k,E_k)=(X',\cF',E'),
\]
where $\cF_{i+1}$ is the strict transform of $\cF_i$ such that:
\begin{itemize}
\item[(1')] The strict transform $Y'$ of the subvariety $Y$ is  transverse to $\cF'$. 
If additionally $dim(Y)=rank(\cF)$ then  $Y'$  is a transverse section of $\cF'$. Moreover, the composite $\sigma= \sigma_0 \circ \cdots \sigma_{k-1}$ is an isomorphism over $Y^{tr} \cap X_0$.
\item[(\ref{It:is-aligned}',\ref{It:non-singular}',\ref{It:functorial}')] Properties \emph{(\ref{It:is-aligned}'), (\ref{It:non-singular}'),} and \emph{(\ref{It:functorial}')} from Theorem \ref{thm:EmbeddedDesingularization} hold true. 
\end{itemize}
\end{theorem}

We prove Theorem \ref{thm:ResTransverseSections} in Section~\ref{ssec:Proofthm:ResTransverseSections}. A complementary statement where  $\cF$ is transverse to $Y$ at a point, suggested to us by J. V. Pereira, is provided in Proposition \ref{Prop:sub-transverse}.

\subsection{Main methods of the proof}\label{sec:IntroMethods}

\subsubsection{Orbifolds and weighted (or cobordant) blow-ups}\label{sssec:Orbifolds}

The presence of weighted blow-ups in resolution of singularities of foliations goes back to the late 80's \cite{BrunellaMiari,DenRou,Pelletier,Rou,PanFamilies}. An example of Sanchez-Sanz highlighted by Panazzolo \cite[Section~1.4]{Pan} shows that, already in dimension three, line foliations cannot be resolved without weighted blow-ups, or at the very least roots. See \cite{CRS,RReis} for a different way forward. In the same paper, Panazzolo proves the existence of resolution of singularities for real-analytic line foliations in threefolds. All of these references concern the real-analytic case and use \emph{manifolds with corners} in order to formalize weighted blow-ups.

When considering the complex analytic or the algebraic category, it becomes natural to work with \emph{orbifolds} instead. Formally, an orbifold is a smooth algebraic or analytic \emph{Deligne Mumford stack} where the stabilizers are generically trivial, see eg. \cite{Fantechi, AbramovichStacksEveryone} for  introductions. We will rarely explicitly use this definition, and a reader that is familiar with the classical definition of orbifolds from differential geometry, see eg. \cite[$\S$ 2.4]{LieGrupoids}, may easily read the paper as well (but they might miss on a beautiful feature of working with stacks: morphisms are very natural and it is, therefore, easy to encode functorial statements). The use of stacks is at the core of the approach of Abramovich, Temkin and W\l odarczyk for varieties and morphisms \cite{ATW-weighted,ATW-relative}, of McQuillan for varieties \cite{MMcQ} and of McQuillan and Panazzolo for line foliations in three dimensions \cite{McQPan}. As a drawback of this choice, the paper may seem technical for a non-algebraic geometer. We sweeten the pill by relying in two observations.

First, our orbifolds are everywhere locally a stack-theoretic quotients $[W/G]$ of a smooth variety or  an analytic manifold $W$ by a group $G$ acting on $W$, where $G$ is either a torus $\GG_m^d$ (or $\CC^*$) or a finite abelian group. In particular, in order to define an object (such as a divisor, foliation, center of blow-up, etc) over an orbifold, it is sufficient to define it over the variety or  analytic manifold $W$ and demand it to be \emph{$G$-invariant}. Since our methods are functorial with respect to group actions and open embedding, this observation allows us to work almost always locally over a smooth variety or  analytic manifold --- we only need to explain, for each object, what being $G$-invariant means.

Second, we present stack-theoretic blow-ups via a two step process, following an idea that W\l odarczyk developed in \cite{Wlodarczyk-cobordant}. In fact, we start by studying \emph{weighted blow-up cobordisms and cobordant blow-ups}: a classical algebraic geometry operation, known also as the \emph{degeneration to the weighted normal cone}  (see \cite[Section 1]{Rees-valuation}, \cite[Chapter 5]{Fulton}) over varieties or  analytic manifolds, see Section~\ref{ssec:CobordantBU}. As it has been observed in \cite{Wlodarczyk-cobordant} and independently studied in \cite{Quek-Rydh}, one can represent the stack-theoretic weighted blow-up as a stack-theoretic quotient of a cobordant blow-up. 
In particular, we could restrict this paper to the category of schemes and analytic spaces provided that we only use cobordant blow-ups. Additionally, some of our key proofs (eg. that the invariant $\inv_{\cF}$ decreases after blow-up of the $\inv_{\cF}$-maximal center, see Theorem \ref{prop:DecreaseInvariant}) remain elementary.

\begin{remark}[Passage from local to global]\label{war:Overaching}
We briefly expand on our use of the first observation above. Throughout sections Sections~\ref{sec:MainObjects} through \ref{sec:BlowUp}, we work with smooth varieties or  analytic manifolds $X$ admitting global coordinates $x=(x_1,\ldots,x_n)$. This hypothesis simplifies notation. In the general case, it is necessary to work on neighborhoods $U$ of every point $\pa \in X$ which admits such coordinates. The passage from local to global follows by canonicity --- in other words, functoriality with respect to local isomorphisms ---  of all the local constructions. There is, therefore, no loss of generality in assuming that $X$ is actually equal to one of these neighborhoods $U$, as long as we prove (as we do in those sections) that all constructions are canonical.
\end{remark}

\begin{remark}[Manifold with corners and Lie groupoids]\label{rk:OnManifoldCorners}
Bearing in mind future applications to non-commutative and differential geometry, it is possible to re-state our results in the language of manifold with corners \cite{KotteMelrose, Gillam-Molcho} or Lie groupoids \cite{LieGrupoids}.
\end{remark}

\subsubsection{Rees algebras, lifting and the foliated Replacement Lemma}\label{sssec:Rees} Orthodox proofs of resolution of singularity of ideal sheaves and varieties in characteristic zero are based on the construction of auxiliary objects related to ideal sheaves: marked ideals, maximal contacts, coefficient ideals, companion ideals, homogenized ideals, MC-ideals, etc, \cite{Hironaka,Kollar,Bierstone-Milman,ATW-weighted,Encinas-Hauser,Encinas-Villamayor,Wlodarczyk}. In this work, we follow a variation of this approach developed by W{\l}odarczyk in \cite{Wlodarczyk-cobordant}, where all auxiliary objects are defined in terms of \emph{rationally graded Rees algebras}, see Definition \ref{def:ReesAlgebra}, instead of ideal sheaves. We note that Rees algebras appear as a natural replacement for coefficient ideals in \cite{Encinas-Villamayor-Rees, Bravo-Villamayor,  BV13, Bravo-Villamayor-Book}, and a natural way to describe centers of weighted blow-ups in \cite{Quek-Rydh}; the paper \cite{Wlodarczyk-cobordant} puts them as a centerpiece of the entire resolution algorithm.

There are several reasons for this choice, having to do with underlying simplicity, elegance, and a unifying philosophy explaining why resolution in characteristic 0 actually works. This is explained in detail in \cite{Wlodarczyk-cobordant} and will hopefully become clear through the paper. But importantly here, it is technically key for our proofs in the foliated case, as we now explain.

Recall that in Hironaka's classical inductive approach, one assumes existence of functorial resolution in a lower dimension and pushes it to a resolution on a higher dimensional variety. Such a resolution \emph{lifts} from a (maximal contact) subvariety to a variety using the concept of the coefficient ideal. On the other hand the gluing process is usually controlled by some sort of homogenization technique, see eg. \cite{Wlodarczyk,Kollar},  which links the resolution of different choices of (maximal contact) subvarieties by \emph{connecting isomorphism}.\footnote{This discussion applies equally to the earlier approach using ``Hironaka's trick", see \cite{BMT}.} In the presence of a foliation, both the lifting and the connecting isomorphisms now need to keep track of an ambient foliation \emph{which is not encoded in the subvariety}. Consequently the resolution on a maximal contact subvariety does not lift automatically to the resolution on a variety, and the gluing process is therefore affected by the foliation. By using Rees algebras, following \cite{Wlodarczyk-cobordant}, we can address both of these points. Note especially our notion of \emph{lifting associated to derivations}, see Definition \ref{def:LiftingAssociatedToD} 
and Proposition \ref{prop:CoefficientSplitting}, 
and our \emph{foliated replacement Lemma}, see Lemma \ref{lem:ReplacementFoliated}.

\begin{warning}
Following this choice, we systematically consider Rees algebras throughout Section~\ref{sec:MainObjects} to \ref{sec:Final} --- ideal sheaves are implicitly covered by Remark \ref{rk:Dictionary}.
\end{warning}

\subsubsection{The invariant $\inv_{\cF}$ and the desingularization algorithm.} Since the original proof of Hironaka's embedded desingularization of varieties in characteristic zero \cite{Hironaka}, an important new feature in the proof is the inclusion of an \emph{invariant} that controls the desingularization algorithm \cite{Bierstone-Milman,Encinas-Villamayor}. The invariant does not only provide a better understanding of the proof, but also helps in understanding the structure of the singularities themselves.

In \cite{ATW-weighted}, Abramovich, Temkin and W{\l}odarczyk introduced a new invariant $\inv$, in fact a natural recursive extension of \emph{order} of an ideal, which provides the necessary control of desingularization via weighted blow-ups. The invariant $\inv$ shares many of the features from \cite{Bierstone-Milman,Encinas-Villamayor}, but has the advantage of being independent of the \emph{history} of the resolution, making it simpler and arguably more geometrical. An extension of the invariant to a context involving exceptional divisors is proposed by W\l odarczyk \cite{Wlodarczyk-cobordant} and in our joint recent work with Quek \cite{ABQTW}. In this paper we further extend the invariant to the context of foliations $\inv_{\cF}$, taking our inspiration from \cite{ATW-relative}. Our paper is, nevertheless, completely self-contained --- no knowledge of previous papers is necessary --- and recovers the previous invariants setting $\cF = \cD_X$ or $\cF=\cF_X^{\log}$.

In order to explain the main ideas of the construction, let us briefly recall the idea behind the invariant $\inv$ of an ideal sheaf $\cI$ in terms of Rees algebras, see Section~\ref{sssec:Rees}, broadly following \cite{Wlodarczyk-cobordant} (compare \cite[Section 9.8]{Bravo-Villamayor-Book}, which differs by a factor). Given a point $\pa \in X$, recall that
\[
\ord_{\cF}(\cI):=\min  \{n \in \mathbb{N} \mid  \cD^n(\cI)=\cO_{X,\pa}\},
\]
which can be re-written in terms of Rees algebras as
\[
\ord_{\pa}(\cI)=\max\{k\in \QQ_{>0} \mid\cI t\subset \cO_X[m_{\pa}t^{1/k}]\},
\]
where $\cR=\cO_X[m_{\pa}t^{1/k}]$ is the rationally graded Rees algebra generated by its degree-$1/k$ term $m_{\pa} t^{1/k}$, where $t$ is the place-holder variable. The definition of $\inv_{\pa}(\cI)$ from  \cite{Wlodarczyk-cobordant} is a natural extension of the above presentation of the order of the ideal, which can be defined recursively via \emph{coefficient ideals} and \emph{hypersurfaces of maximal contact}, see Lemma \ref{lem:FcanonicalInvForFiniteRees}. If the divisors are not present then it boils down to:
\[
\inv_{\pa}(\cI):=\max\{(a_1,\ldots,a_k)\mid \cI t\subset \cO_X[x_1t^{1/a_1},\ldots,x_lt^{1/a_l}]^{\inte}\},
\]
where $a_1\leq\ldots\leq a_l$ are  rational numbers ordered lexicographically and the maximum is considered over all the partial coordinate systems $x_1,\ldots,x_k$ at $ \pa \in X$. Here ``$\inte$" stands for the corresponding integral closure.

In this paper we adapt this invariant to take into consideration $\cF$.  We start by establishing a notion of $\cF$-order to replace the usual order, see Definition \ref{def:Forder},
\[
\ord_{\cF,\pa}(\cI):=\min  \{\{n \in \mathbb{N} \mid  \cF^n(\cI)=\cO_{X,\pa}\} \cup \{\infty\}\},
\]
which is a variation of previous notions of orders \cite{ATW-relative,BelJA,BelRACSAM,BelBmon}. We then extend the invariant $\inv$ to its foliated version, see Definition \ref{def:FCanonicalInvariantGeneral}:\footnote{As explained in Section \ref{sssec:Foliation} below, the inherent transcendental nature of foliation requires working with completions.}
\begin{equation}\label{eq:FinvPresentationSimple}
\begin{aligned} \inv_{\cF,\pa}(\cI):=\max \{&(a_1,\ldots,a_k,\infty+b_1,\ldots,\infty+b_k)\mid \\ &  \cI t \subset \widehat{\cO}_{X,\pa}[x_1t^{1/a_1},\ldots,x_lt^{1/a_l},y_1t^{1/b_1},\ldots,y_rt^{1/b_r}]^\inte\}
\end{aligned}
\end{equation}
where $a_1\leq\ldots\leq a_l$ and $b_1\leq\ldots\leq b_r$ are rational numbers and the coordinate systems $(x_1,\ldots,x_k,y_1,\ldots,y_r)$ ``respect" the foliation $\cF$, that is, they form a formal \emph{$\cF$-aligned center}, see Definition \ref{def:Faligned}.\footnote{The complete definition includes a third layer of invariants accounting for monomial variables, which we suppress in the discussion here.} More precisely, the coordinates $(y_1,\ldots,y_r)$ are $\cF$-invariant, that is:
\[
\cF \cdot \widehat{\cO}_{X,\pa}[y_1t^{1/b_1},\ldots,y_rt^{1/b_r}]^\inte \subset \widehat{\cO}_{X,\pa}[y_1t^{1/b_1},\ldots,y_rt^{1/b_r}]^\inte,
\]
and we can find a formal presentation
\begin{equation}\label{eq:FallignedPresentationSimple}
\widehat{\cF}=\widehat{\cO}_{X,\pa}\cdot \cF=\spa(\partial_{x_1},\ldots,\partial_{x_k}, \nabla_1,\ldots,\nabla_r),
\end{equation}
where $\nabla_j \in \cD_{X}$ are derivations independent of $(x_i,\partial_{x_i})$, see Definition \ref{def:DerIndependentx}. The center provided by these formal coordinates is in fact algebraic. Moreover the presentation $\spa(\partial_{x_1},\ldots,\partial_{x_k}, \nabla_1,\ldots,\nabla_r)$ is  \emph{nested-regular}
as discussed in Section \ref{sssec:Foliation}. Following this construction, we collect some of the most important properties of $\inv_{\cF}$ and $\cF$-aligned centers in the following results:

\begin{theorem}[Semicontinuity and functoriality of invariant]\label{Th:semicontinuity}
The invariant $\inv_{\cF}(\cI)$ is upper semicontinuous, functorial for smooth morphisms and field extensions, and takes values in a well-ordered set.

The invariant  $\inv_{\cF}(\cI)$ attains a maximum (1) when $X$ is algrebraic, (2) in the preimage  $\phi^{-1}(U) \subset X$ 
of a neighborhood of a compact subset of $Y$, when $X\to Y$ is smooth, and $\cI$ and $\cF$ are pulled back from $Y$. Here for  
 $\cF$ we use either Section \ref{Sec:nonsat-smooth}  or \ref{Sec:relative-pullback}. 
\end{theorem}

We note that functoriality implies that $\inv_{\cF}(\cI)$ is invariant under arbitrary group actions.

\begin{theorem}[The center associated to an ideal]\label{Th:the center}

Suppose the invariant $\inv_{\cF}(\cI)$ attains its maximum along a locus $Z$. Then $Z$ is closed, and is
the support of a unique $\cF$-aligned $\cI$-admissible center ${J}$ having invariant $\max_{\pa}\inv_{\cF,\pa}(\cI)$ precisely along the locus $Z$. In particular $Z$ is regular. The formation of ${J}$ is functorial for smooth surjective morphisms and field extensions.

\end{theorem}
These results follow from Inductive claim \ref{claim:MainInduction} together with Proposition \ref{prop:FCanonicalInvBasicProperties}. The center ${J}$ from Theorem \ref{Th:the center} will be called \emph{$\inv_{\cF}$-maximal center}.

\begin{theorem}[The invariant drops]\label{Th:principalization}
Under the hypotheses of Theorem \ref{Th:the center},
consider the weighted (or cobordant) blowing up $\pi: (X',\cF',E') \to (X,\cF,E)$ by the $\inv_{\cF}$-maximal center ${J}$. Let  $\cF'$ and $\cI'$ be both either the strict transforms or the controlled transforms. Then
\[
\max \inv_{\cF',\pa'}(\cI') < \max \inv_{\cF,\pa}(\cI).
\]
\end{theorem}

This Theorem is a consequence of Theorem \ref{prop:DecreaseInvariant}.

\subsubsection{Nested-regular Rectification Theorem and splitting along transverse sections.}\label{sssec:Foliation}
We must comment on a  technical point. Note that the discussion of the invariant we just performed was done via formal power series, but we claimed our invariant is semicontinous. In order to fill this gap we need to develop new methods in algebraic foliation theory, extending \cite[Theorem 30.1]{Matsumura-ringtheory}.

In the construction of $\cF$-aligned centers, captured in Equation \eqref{eq:FinvPresentationSimple} on page \pageref{eq:FinvPresentationSimple}, we must provide local coordinates $(x,y)$ with two qualitatively different meanings with respect to the foliation $\cF$, which is partially captured in equation \eqref{eq:FallignedPresentationSimple}. In particular, the $x$-coordinates are $\cF$-transverse, in the strong sense that $\partial_x$ belongs to the foliation. In order to build these coordinates, we are naturally led to use one of the most fundamental tools in analytic foliation theory:

\begin{theorem}[Rectification Theorem, see e.g. {\cite[Th. 1.18]{IY}}] \label{thm:FlowBox} Suppose that $X$ is a coherent analytic space. Let $\partial \in \cD_{X,\pa}$ and $x\in \cO_{X,\pa}$ be such that $\partial(x)$ is a unit. There exists an Euclidean coordinate system $(x,y) = (x, y_1,\ldots,y_{n-1})$ centered at $\pa$ such that $\partial = V(x,y) \partial_{x}\sim \partial_{x}$, where $V(x,y)$ is a unit.
\end{theorem}

This result plays a fundamental role in this work. In general, nevertheless, it is not possible to obtain a rectification theorem over regular or \'{e}tale coordinate systems.

\begin{example} \label{EX:folNonFB}
Consider $\partial=\partial_x-y\partial_y $ with first integral  $y':=e^xy$.  Then $\partial=\partial_x$ with respect to the Euclidean coordinate system $(x,y')$. But there is no regular or \'{e}tale coordinate system $(x,\bar{y})$ such that $\partial = \partial_x$.
\end{example}

Note that a formal version of the rectification Theorem would not be enough in order to show the semi-continuity of the invariant $\inv_{\cF}$. We therefore prove a \emph{nested-regular rectification Theorem}, see Theorem \ref{thm:FRFlow}, where we work with intermediate local rings $\cO_{X,\pa} \subset \widehat{\cO}_{X,(x)} \subset \widehat{\cO}_{X,\pa}$, see Definition \ref{def:ForReg} and the remark that follows, which allows us to work algebraically on the subvariety $V(x)$. As the main consequence, we establish a \emph{local splitting of foliations along a transverse section}, see Theorem \ref{thm:splittingFoliation}. Later on, maximal contact hypersurfaces will provide these transverse sections. It will be necessary to understand the restriction of the foliation to them, as well as the lifting information from the restricted foliation to the original ambient space,  for the inductive scheme of resolution of singularities.

\subsubsection{An example}
To illustrate how the algorithm changes in the presence of a foliation, let $\cF$ be the foliation on $\AA^2=\Spec(K[x,y])$ generated by $\partial_x$, and  let $\cI=(x^5+y)$. Then the associated maximal center at the origin is the Rees algebra $\cA=\cO_X[xt^{1/5},yt]$ with the invariant $(5,\infty+1)$.
Note that $t$-gradation is given by the ideal $\cA_{1}=(x^5,y)$, while  the admissibility condition  can be stated  as inclusion of $t$-gradations as
\[
\cI=(x^5+y)\subset (x^5,y).
\]
On first glance such a condition heavily depends on the choice of the second coordinate $y$. For example for $y':=x^5+y$ we would simply have $\cI\subset (y')$ leading to smaller invariant $(0,\infty+1)$. However in an $\cF$-aligned center, we only allow $y'$ for which $\partial_x(y') \in (y')$.

\subsubsection{Preserving properties of foliations} The underlying result of Theorem \ref{thm:PrincipalizationFoliated}\eqref{It:non-singular} is the following:

\begin{theorem}[Aligned blow-ups preserve properties of foliations] \label{Th:aligned center}
Let $(X,\cF,E)$ be a smooth foliated logarithmic variety and consider an $\cF$-aligned weighted (or cobordant)  blowing up  $\pi: (X',\cF',E') \to (X,\cF,E)$. Assume $\cF$ is either logarithmically smooth or $\cK$-monomial. Then the controlled and strict transforms of $\cF$ coincide and are correspondingly logarithmically smooth or $\cK$-monomial. If $\cF$ is \Sm, \emph{and $\pi$ is the cobordant blow-up,} and $\cF'$ is taken in the split sense, then $\cF'$ is \Sm.
\end{theorem}

The  result for $\mathcal{K}$-monomial foliations and smooth blow-ups was proven in the complex analytic category in \cite[Prop. 4.4]{BelRACSAM}.

\subsection{Algebraic and analytic categories}\label{sec:IntroAlgebraicAnalytic}
Our main results are stated in both the algebraic and complex analytic categories. A great deal of the foundation theory for these categories are very similar. But there are some notable exceptions associated to considerations about topology and morphisms. It will be, therefore, convenient to chose one of the two categories and use their terminology consistently --- and our choice is algebraic geometry. It is, nevertheless, essential to establish a dictionary for the complex analytic category, which we now present. 

By a variety $X$, we mean either a reduced scheme of finite type over a field $K$
 of characteristic zero, or a coherent complex analytic space, where $K = \CC$. We will denote by $\cO_X$ the corresponding sheaf of functions. Note that several of the basic notions for $X$, such as being smooth, reduced, irreducible, normal, a subvariety (a complex  subspace), etc, are all classically defined in complex geometry, see eg. \cite{Grauert-Remmert-Coherent}. The notation $\GG_m$ will stand for $\CC^*$ in the complex case.
Faithfully flat descent is available in the complex analytic theory, see \cite{Kiehl}.

\begin{definition}[Regularity of functions]\label{def:RegularityConvention}
Let $\cO_{X,\pa}$ be a local ring of a point $\pa$ on an algebraic variety $X$. A function $f\in \widehat{\cO}_{X,\pa}$ is  {\it formal, regular or \'etale} at $\pa$ if it is respectively in $\widehat{\cO}_{X,\pa}$, ${\cO}_{X,\pa}$ or in the henselianization ${\cO}_{X,\pa}^h$.

Let $\cO_{X,\pa}$ be a local ring of a point $\pa$ on a coherent analytic space $X$. A function $f\in \widehat{\cO}_{X,\pa}$ is  {\it formal, or Euclidean} at $\pa$ if it is respectively in $\widehat{\cO}_{X,\pa}$, or in $\cO_{X,\pa}$.
\end{definition}

It is convenient to make a consistent distinction between the Euclidean and analytic-Zariski topology. By an open set $U\subset X$, we mean an open set in the Euclidean topology. By a Zariski-open set $U \subset X$, we mean that $X \setminus U$ is a closed subvariety of $X$ --- of course, in the algebraic category, every open set is a Zariski-open set. We extend this convention to the notion of ``generic".

Consider now a morphism $\varphi : X \to Y$ between analytic spaces. Given $\pa \in X$, we denote by $\varphi^{\ast}_{\pa} :\cO_{Y,\varphi(\pa)} \to  \cO_{X,\pa}$ the induced morphism of local rings. We say that $\varphi$ is \'{e}tale  (resp. smooth), if it is everywhere a local isomorphism (resp. local submersion). By a finite morphism $\varphi$, we mean that $\varphi$ is continuous and closed (with respect to the Euclidean topology) and has finite fibers (this terminology is standard in complex geometry, see \cite[page 47]{Grauert-Remmert-Coherent}). The word ``generic" in these contexts (eg. generically smooth or generically finite) are always going to be in reference to the analytic-Zariski topology.

More generally, recall that a mapping $\varphi: X \dashrightarrow Y$ is meromorphic, following Remmert \cite[Definition 16]{Remmert-Meromorphic},
 if there exists a Zarsiki-open set $U \subset X$ such that: (i) $\varphi_{|U}: U \to Y$ is analytic; and (ii) the closure $\overline{\Gamma} \subset X \times Y$ of the graph $\Gamma = \Gamma_{\varphi_{|U}}$ is an analytic subset of $X\times Y$ which projects properly to $X$. A bimeromophic map $\varphi: X \dashrightarrow Y$ is then a meromorphic map whose inverse is also meromorphic. Finally, we will sometimes say that $\varphi$ is birational instead of bimeromorphic to unify terminology --- we will insist in the word bimeromophic in definitions and in the statements of our main results, in order to avoid confusion, but we won't make the distinction in proofs. 

When defining the blow-up in the complex analytic category, it is usual to do it via a very explicit elementary construction --- in essence, the projection of an almost everywhere local graph. We will prefer to use  adapted versions of Spec and Proj instead, called $\operatorname{Specan}$ and $\operatorname{Projan}$ in \cite[Definitions 1.5.2 and 1.7.1]{AHV-analytic}. To unify the discussion, the notation $\Spec$ and $\Proj$ in the complex analytic category will mean $\operatorname{Specan}$ and $\operatorname{Projan}$. A broader context is provided in Hakim's book \cite{Hakim}: one can define a relative scheme $\Spec_Y(A)$ and $\Proj_Y(A)$ for a graded algebra finitely generated by coherent sheaves over an arbitrary locally ringed space $Y$, in particular a coherent analytic space, and then analytify the result.

\begin{remark}[On more general analytic spaces]\label{rk:MissingAnalyticCategories}
It should be possible to extend the results to the real analytic and $\mathbb{K}$-analytic categories, where $\mathbb{K}$ is an algebraic closed complete valued field of characteristic zero. It should likewise be possible to work in the generality of regular quasi-excellent schemes {\em with enough derivations} in the sense of \cite[1.3.2(iii)]{Temkin-qe}, including the category of formal varieties.\footnote{As far as we know, \cite{ATW-relative} is the only paper so far in which the approach of working with arbitrary schemes with enough derivations was implemented.}

 Addressing this requires work orthogonal to this paper, as foliations on formal schemes behave differently from algebraic foliations,\footnote{A warning sign is Equation \eqref{eq:Euler} on page \pageref{eq:Euler}.} and the alternative formulation of manifolds with corners (\cite{Pan}, see Remark \ref{rk:OnManifoldCorners}) belong to a different subject altogether. We have postponed these questions to future work.

\end{remark}

\subsection{Structure of the paper} In Section~\ref{sec:Foliation} we provide the definition of foliated logarithmic varieties. In Section \ref{Sec:rectification} we prove the two algebraic foliation theory results discussed in Section~\ref{sssec:Foliation}. In Section~\ref{sec:MainObjects}, we introduce the main objects of the resolution of singularities algorithm: Rees algebras $\cR$, $\cF$-order, $\cF$-aligned centers, $(\cR,\cF)$-admissible centers, coefficient Rees algebras, maximal contact and the foliated replacement Lemma. In Section~\ref{sec:algorithm} we establish the algorithm to define the invariant and choosing the $\cF$-aligned center. Up to this point of the paper, we do not need blow-ups, and we can work over smooth varieties or  analytic manifolds only, see Section~\ref{sssec:Orbifolds}. In Section~\ref{sec:BlowUp} we define cobordant and stack-theoretic weighted blow-ups and we prove that the invariant $\inv_{\cF}$ drops after blowing up the $\inv_{\cF}$-maximal center. In Section \ref{Sec:nice-foliations} we describe $\cK$-monomial and logarithmically smooth foliations, show that they are preserved under aligned blow-ups, and provide a logarithmic resolution of $\cK$-monomial foliations, Theorem \ref{thm:LogSmoothnessMonomial}. Finally, in Section~\ref{sec:Final} we collect the proofs of the main results of the paper. Appendix~\ref{App:transforms} collects constructions and results on pullbacks of saturated and non-saturated foliations, as well as \'etale charts for foliated weighted blow-ups.

\subsection*{Acknowledgements} Dan Abramovich thanks IHES and the Hebrew University  for their hospitality during 2024-2025. We thank CIRM and BNF for hospitality at meetings where final touches were made. We thank the participants of the CIRM workshop \emph{Foliations, birational geometry and applications}, 3 – 7 February, 2025, where many stimulating discussions occurred. Thanks to St\'ephane Druel and Jorge Vot\'orio Pereira for help with references, and to Paolo Cascini, Pereira, and Calum Spicer for suggestions related to Sections \ref{sec:IntroResTotallyIntegrable} and \ref{sec:IntroResTransverseLocus}.

\section{Singular Foliations}\label{sec:Foliation}

\subsection{Logarithmic varieties}\label{ssec:Log}
Let $X$ be either a smooth irreducible variety over a field $K$ of characteristic $0$, or a coherent complex analytic space,
and let $E$ be a simple normal crossing (SNC) divisor on $X$. For convenience of discussion we also assume $X$ irreducible. We  shall call the pair $(X,E)$ a {\it logarithmic variety}, which we take to include both the algebraic and analytic settings. A morphism between two logarithmic varieties $\phi:(X,E) \to (Y,F)$ is a morphism between $X$ and $Y$ such that the reduced divisor associated to $\phi^{-1}(F)$ is {contained in} $E$ --- this is consistent with the notion of morphisms of logarithmic structures in the sense of Kato--Fontaine--Illusie \cite{Kato-log}. A partial coordinate system $x_1,\ldots,x_n$ on an open subset $U$ of $X$ {\it is adapted to  $E$} if the restriction $E_{i|U}$ of any component  $E_i$ in $E$ intersecting $U$ is of the form $V(x_i)$. If $\pa\in V(x_i)$, where $V(x_i)$ determines a component in $E_{|U}$, then we say that $x_i$ is {\it divisorial} at $\pa$. Otherwise, we say that $x_i$ is {\it free} at $\pa$.

\subsection{Foliated logarithmic varieties}\label{ssec:2FoliatedLog}
Now, consider the coherent sheaf of $K$-derivations $\cD_{X}=\cD_{X/K}$, and its subsheaf of logarithmic derivations $\cD^{\log}_{X}= \cD_{X}(-\log E)$: given a local coordinate system $x_1,\ldots,x_n$ adapted to $E$, the sheaf  $\cD^{\log}_{X}$ is generated by the operators $\partial_{x_i}$ associated with  free coordinates,  along with $x_j\partial_{x_j}$ associated with divisorial coordinates.  

\begin{definition}[Distributions and foliations] \label{Def:foliations}
\hfill
\begin{enumerate} 
\item By a {\it distribution} (respectively a {\it logarithmic distribution}) $\Delta$ on $(X,E)$ we mean any coherent subsheaf of $\cD_{X}$
 (respectively $\cD^{\log}_{X}$).  
 \item We say that $\Delta$ is a foliation, and we denote it by ${\mathcal F}$, if it is involutive,  that is closed
 under Lie bracket, so for any $\partial^1,\,\partial^2\in \cF(U)$ on open subset $U$ the Lie bracket
 \[
 [\partial^1,\partial^2]=\partial^1\partial^2-\partial^2\partial^1\in \cF(U).
 \]

\item
The {\it rank} of the foliation ${\mathcal F}$ is its rank as a sheaf. 
\item\label{ItDef:LogSmooth} The {\it singular} (respectively {\it logarithmically singular}) set of ${\mathcal F}$, which we denote by $\mbox{Sing}({\mathcal F})$, respectively $\mbox{Sing}^{\log}({\mathcal F})$, is the set of points $\pa \in X$ where the quotient $\cD_{X,\pa}/ {\mathcal F}_{\pa}$ (respectively $\cD_{X,\pa}^{\log} / {\mathcal F}_{\pa}$) is not locally free; otherwise we say that $\cF$ is {\it (logarithmically) {\Sm}  at $\pa$}. The set of {\Sm} points of a foliation $\cF$ on  a {\Sm} variety $X$ will be denoted by $(X,\cF)^{\nsing}=X\setminus  \mbox{Sing}({\mathcal F})$.

\item  If $\cD_{X,\pa}/ {\mathcal F}_{\pa}$ (respectively, $\cD_{X,\pa}^{\log}/ {\mathcal F}_{\pa}$) is  torsion-free, then we say that $\cF$ is saturated (respectively, logarithmically saturated) at $\pa$.
\end{enumerate}
\end{definition}

Note that the singular set is a proper  subvariety; if ${\mathcal F}$ is saturated then $\mbox{Sing}({\mathcal F})$ is of codimension $\geq 2$, and similarly for the logarithmic notion.

When $\cF$ is {\Sm} at $\pa$  the exact sequence $$0\to \cF_\pa\to \cD_{X,\pa}\to \cD_{X,\pa}/ {\mathcal F}_{\pa}\to 0$$ splits and both $\cF_\pa$ and  $\cD_{X,\pa}/ {\mathcal F}_{\pa}$ are free $\cO_{X,\pa}$-modules.

\begin{remark}[On the definition of foliations]\label{rk:LiteratureDefFoliations}
In the birational geometry literature, going back at least to the foundation work of Baum and Bott \cite{BaumBott}, it is common to define a foliation as being everywhere saturated. This choice is not usually imposed in references in differential geometry, see eg. \cite{Lavau}, and we do not need this condition for our considerations. In fact, it will be convenient to keep the definition of foliation more flexible in order to introduce the notion of \emph{controlled transform} of a foliation, see Definition \ref{def:TransformFoliation}.
\end{remark}

\begin{definition}[Foliated logarithmic variety]\label{def:LogFoliatedVariety}
We call the triple $(X,\cF,E)$, where $\cF \subset \cD^{\log}_X$ is a foliation, a \emph{foliated logarithmic variety}.
\end{definition}

\subsubsection{Functoriality  of foliations}

{In Appendix~\ref{App:transforms} we collect several constructions and results on the transformation of saturated and nonsaturated foliations by various types of morphisms.} With this, we are ready to define several notions of functoriality, used in the main theorems, via the non-saturated pull-back:

\begin{definition}[Smooth morphism of foliated logarithmic variety]\label{def:MorphismLogFoliatedVariety}
We say that  $\phi: (X,\cF,E) \to (Y,\cG,F)$ is a smooth morphism between foliated logarithmic varieties if $\phi$ is a smooth morphism of logarithmic varieties such that $\cF = \dispull\phi(\cG)$, see Definition \ref{Def:fol-pullback}.
\end{definition}

\begin{definition}[G-invariant foliation]\label{def:GsemiinvariantFol}
Let $\phi: X\to X$ be an authomorphism. We say that $\cF$ is \emph{$\phi$-invariant} if $\cF=\dispull{\phi}(\cF)$, again Definition \ref{Def:fol-pullback}. More generally, given a group of authomorphisms $G \subset \mbox{Aut}(X)$, we say that $\cF$ is \emph{$G$-invariant} if $\cF$ is $\phi$-invariant for all elements $\phi \in G$. 

A morphism  of foliated manifolds $\psi: (Y,\cF_Y) \to (X,\cF_X)$ is \emph{$G$-equivariant} if it is equivariant and the foliations are $G$ invariant.
\end{definition}

A useful infinitesimal version is the following:

\begin{definition}[$\delta$-invariant foliation]\label{def:deltainvariantFol}
Let $\delta\in \cD^{\log}_X$ be a derivation. We say that $\cF$ is \emph{$\delta$-invariant} if $$[\cF,\delta] :=\{[\partial,\delta]: \partial \in \cF\}\quad \subseteq \quad \cF.$$ 
A morphism  of foliated manifolds $\psi: (Y,\cF_Y) \to (X,\cF_X)$ is \emph{$\delta$-equivariant} if $\delta$ lifts to $Y$ and the foliations are $\delta$-invariant.
\end{definition}

We will also need one more type of functoriality --- compatibility with field extensions (see also Section \ref{Sec:relative-pullback}). 

\begin{definition}[Ground field extension]\label{def:fieldextension}
Assume that $(X,\cF,E)$ is a foliated $K$-variety and $L/K$ is a field extension. Then by $(X,\cF,E)_L$ we denote the foliated $L$-variety consisting of $X_L=X\otimes_XL$, $E_L=E\otimes_KL$ and $\cF_L=\cF\otimes_KL$, where the latter is the pullback under the morphism of schemes $X_L\to X$, which is defined similarly to the pullback under smooth morphisms.
\end{definition}
 In particular all transforms along $X_L \to X$ coincide. In fact any $K$-derivation on $X$ uniquely extends to an $L$-derivation on $X_L$, so the simplest description of $\cF_L$ is as follows: one extends the derivations of $\cF$ to $L$-derivations and takes the $L$-linear span.


\section{Rectification and splitting of foliations}\label{Sec:rectification}
\subsection{Nested-regular rectification theorem}
In this section we provide an algebraic version of the Rectification Theorem for general foliations, see Theorem \ref{thm:FRFlow} below, which allows us to define regular blow-up centers later on. 

There is a tension we address here: foliations are by nature a transcendental notion, and it is natural to use formal coordinates, but in order for centers to be regular,  the coordinates must satisfy certain coherency properties and they cannot be purely formal. To address this point, we consider an intermediate class of formal coordinates. Recall that given a noetherian ring $R$ and an ideal $I$ of $R$, we may associate to it the $I$-adic completion $\widehat{R}_I$, given by:
\[
\widehat{R}_I:=\lim_{\leftarrow} R/I^n,
\]
which is itself a noetherian ring. More specifically, for any ideal $I \subset{\cO}_{X,\pa}$ denote by $\widehat{\cO}_{X,I}$ the completion of $\cO_{X,\pa}$ at $I$. In what follows, we work with completions with respect to ideals which admit a simple \emph{formal presentation}, so it is convenient to follow the following convention: we say that a formal ideal $I \subset{\widehat{\cO}}_{X,\pa}$ is \emph{algebraizable} if the ideal $I':=I  \cap \cO_{X,\pa} \subset{\cO}_{X,\pa}$ is such that $I'\cdot \widehat{\cO}_{X,\pa} = I$; in this case, we establish the notation $\widehat{\cO}_{X,I} := \widehat{\cO}_{X,I'}$. We are ready to define:

\begin{definition}[Nested-regular coordinates]\label{def:ForReg}
A partial local coordinate system
\[
(x_1,\dots x_k)\in \widehat{\cO}_{X,\pa}
\]
at a point $\pa$ will be called {\it nested-regular} if $x_1\in \cO_{X,\pa}$ and for $i=1,\ldots, k$ each formal ideal $I_i=(x_1,\ldots,x_i)$ is algebraizable and $x_{i+1}\in \widehat{\cO}_{X,I_i}$.
\end{definition}

\begin{warning}[Notation $\widehat{\cO}_{X,I} = \widehat{\cO}_{X,(x)}$] Given a partial nested-regular coordinate system $(x_1,\ldots,x_k)$ we will denote by $\widehat{\cO}_{X,(x)}$ the ring $\widehat{\cO}_{X,I}$, where $I = (x_1,\ldots,x_k)$ is algebraizable. 
\end{warning}

\begin{definition}[Lifting of nested-regular coordinates]\label{def:Lifting}
Given a partial nested-regular coordinate system
\(
(x_1,\dots x_k)\in \widehat{\cO}_{X,\pa}
\)
at a point $\pa$, set $H := V(x_1,\ldots,x_k)$. We say that an injective homomorphism:
\[
\ell: \cO_{H,\pa}=\widehat\cO_{X,\pa}/(x) \rightarrow \widehat{\cO}_{X,(x)},
\]
is a \emph{lifting} if it is a right inverse of the natural surjection $\widehat{\cO}_{X,(x)}\to \cO_{H,\pa}$.
\end{definition}

More specifically, we work with a special type of lifting:

\begin{definition}[Lifting associated to $(x,\partial)$]\label{def:LiftingAssociatedToD}
Let $x \in \cO_{X,\pa}$ be a coordinate and $\partial$ be a derivation such that $\partial(x) \in \cO_{X,\pa}^{\times }$. Set $H=V(x)$. We say that a lifting $\ell:\cO_{H,\pa} \to \cO_{X,(x)}$ is \emph{associated to $(x,\partial)$} if:
\[
\ell(\cO_{H,\pa}) = \cO^{\partial} := \{f\in \widehat{\cO}_{X,(x)};\, \partial(f)\equiv 0\} \subset \widehat{\cO}_{X,(x)}.
\]
In this case, we may denote $\ell$ by $\ell_{x,\partial}$.
\end{definition}

We are ready to state the main result of this section, an extension of \cite[Theorem 30.1]{Matsumura-ringtheory}:

\begin{theorem}[Nested-regular rectification] \label{thm:FRFlow} 
Let $(X,E)$ be a logarithmic variety of dimension $n$ and $\pa \in X$ be a point. Let $\partial \in \cD^{\log}_{X,\pa}$ be a logarithmic derivation on $\cO_{X,\pa}$ and let $x\in \cO_{X,\pa}$ be a germ such that $\partial(x)$ is a unit. Set $H:=V(x)$. 

\begin{enumerate}
\item There exists a nested-regular coordinate system $(x,\widehat{y})=(x,\widehat{y}_1,\ldots,\widehat{y}_{n-1}) \in \widehat{\cO}_{X,(x)}$ adapted to $E$ such that:
\[
\partial = u \, \partial_{x}\sim \partial_{x}, \quad \text{is a derivation over }\widehat{\cO}_{X,(x)},
\]
where $u \in \widehat{\cO}_{X,(x)}^{\times}$ and the elements $y_i=\widehat{y}_{i|H}\in \cO_{H,\pa}$ are regular (in particular, all the formal ideals  $(x,\widehat{y}_i)$ are algebraizable). Moreover, there exists a lifting $\ell_{x,\partial}: \cO_{H,\pa} \to \widehat{\cO}_{X,(x)}$ associated to $(x,\partial)$ given by $\ell_{x,\partial}(y_i) =\widehat{y}_i$.

\item The statement is compatible with smooth morphism and field extension: given $(X,E, x)$ which pulls back to $(X',E', x')$, and a lift $\partial'$ of $\partial$, the pullback  of $(x,\widehat y)$ can be completed to a coordinate system satisfying the theorem for   $(X',E'), x'$ and $\partial'$.
\end{enumerate}
\end{theorem}
\begin{proof}
We can assume without loss of generality that $\partial(x)\equiv 1$ by rescaling $(1/\partial(x))\partial$. Let $y\in \cO_{X,\pa}$.  

We first claim that there exists  a sequence $(y_d)_{d\in \mathbb{N}}$, functorial for smooth morphisms and field extensions, where $y_0 = y$ and $y_d\in \cO_{X,\pa}$ for all $d$, such that:

(a) $y_{c}-y_d\in (x^{c+1})$  for $c\leq d$, \quad  (b) $\partial(y_d)\in (x^d)$,
\quad 
(c) If $y$ is divisorial so is $y_d$.

In particular, the limit of the sequence $(y_d)$ defines a unique element $\widehat{y} \in \widehat{\cO}_{X,(x)}$ such that $\partial(\widehat{y}) \equiv 0$. 

Note that the Theorem follows from the claim. Indeed, consider any coordinate system $(x,y_1,\ldots,y_{n-1})$ at $\pa\in X$ adapted to $E$. 
Associate to each element $y_i$ the formal limit $\widehat{y}_i$ from the claim. 

We now turn to the proof of the claim. First, put $y_0:=y$ and set
\[
y_d:= y_0 - \sum_{j=1}^d\frac{x^j}{j!} \partial^j(y_{0}).
\]
To prove (a), note that:
\[
y_c- y_d = \sum_{j=c+1}^d\frac{x^j}{j!} \partial^j(y_{0}) \in (x^{c+1}).
\]
Property (b) follows by noticing that:
\[
\partial(y_d) = \partial(y_0) - \sum_{j=0}^{d-1}\frac{x^{j}}{j!} \partial^{j+1}(y_{0}) - \sum_{j=1}^d\frac{x^{j}}{j!} \partial^{j+1}(y_{0}) = \frac{x^d}{d!} \partial^{d+1}(y_0) \in (x^d).
\]
For Property (c), if $y$ is divisorial, $\partial(y_0) \in (y_0)$, and hence $y_d \in (y_0)$, allowing us to conclude that it is also divisorial.
\end{proof}

\begin{example} \label{EX:folNonFB2} In Example \ref{EX:folNonFB} consider $\widehat{y}:=e^xy \in \widehat{\cO}_{X,(x)}$, which is such that $\partial(\widehat{y})\equiv 0$. Then the ideal
\(
(x,\widehat{y})=(x,e^xy)=(x,y)
\)
is algebraizable, implying that $(x,\widehat{y})$ is a nested-regular system of parameters. In particular, note that
\(
\widehat{y}_{|V(x)}=e^xy_{{|V(x)}}=y_{{|V(x)}} \in \cO_{V(x)}
\)
is also regular.
\end{example}

\begin{example}\label{ex:DifferentLiftings}
Note that different derivations provide different lifting. Indeed, let $(x,y)$ be coordinates of $\cO_{\mathbb{C}^2,0}$ and set $H=V(x)$. Note that $\cO_{H,0}$ is generated by $\bar{y} = y_{|H}$. Consider the two derivations $\partial = \partial_{x}$ and $\partial' = \partial_{x} + \partial_{y}$, which satisfy $\partial(x)=\partial'(x) =1$. Denoting by $\ell=\ell_{x,\partial}$ and $\ell'=\ell_{x,\partial'}$ we get:
\[
\ell(\bar{y}) = y, \quad \ell'(\bar{y}) = y - x
\]
providing two different lifting from $\cO_{H,0}$ to $\widehat{\cO}_{\mathbb{C}^2,(x)}$.
\end{example}

\begin{remark}[Euclidean coordinates]\label{rk:FormRegAnalytic} If $X$ is a complex coherent analytic space, then the nested-regular rectification Theorem \ref{thm:FRFlow} holds true with Euclidean coordinates instead of nested-regular coordinates. In fact, the usual proof of the rectification Theorem \ref{thm:FlowBox} given in \cite{IY}, for instance, already satisfies the extra conditions given in Theorem \ref{thm:FRFlow}. To see this, note that the change of coordinates there boils down to:
\[
(\tilde{x},\tilde{y}_1,\ldots,\tilde{y}_{n-1}) = \exp(x\partial)(0,y_1,\ldots,y_{n-1}) = \varphi(x,y),
\]
which is Euclidean. Now, the Taylor expansion of $\varphi$ is (e.g. \cite[Chapter 5.3]{SagleWalde}):
\[
f(\varphi(x,y)) = \sum_{j=1}^{\infty} \frac{x^j}{j!}\partial^j(f), \quad \forall f\in \cO_{X,\pa},
\]
and we can verify all functoriality conditions just as in the proof of the Theorem \ref{thm:FRFlow}.
\end{remark}

\subsection{Local splitting of foliations along transverse sections}\label{sec:Fol:LocalSplittingFoliations}

In this section, we establish a local splitting of foliations along a transverse hypersurface. The result was previously established in \cite[Prop. 8.6.6]{Belthesis} in the analytic category (but without the functorial statements), and its proof follows a well-known foliation theory argument, which we learned from \cite[Th. 3.1]{MY04}:

\begin{theorem}[Local splitting along transverse hypersurfaces]\label{thm:splittingFoliation} 
Consider a foliated logarithmic variety $(X,\mathcal{F},E)$. Let $x\in \cO_{X,\pa}$ be a coordinate and $\partial \in \cF_{\pa}$ be a derivation such that $\partial(x)$ is invertible. Set $\widehat{\cF}_{(x)}:=\widehat{\cO}_{X,(x)}\cdot \cF$ and $H=V(x)$. 
\begin{enumerate}
\item  The foliation $\widehat{\cF}_{(x)}$  {\it splits} at $H$, that is, there exists a nested-regular coordinate system $(x,y)\in \widehat{\cO}_{X,(x)}$ such that $\partial = u \partial_x$ with $u\in \widehat{\cO}_{X,(x)}^{\times}$, and a lifting $\ell_{x,\partial_x}: \cO_{H,\pa} \to \widehat{\cO}_{X,(x)}$ associated to $(x,\partial_x)$ which provides an embedding of the restricted foliation $\cF_{|H}$, see Definition \ref{Def:fol-pullback}, into $\widehat{\cF}_{(x)}$ and
\[
\widehat{\cF}_{(x)}=\spa_{\widehat{\cO}_{X,(x)}}(\partial_x, d\ell_{x,\partial_x}(\cF_{|H})).
\]
\item The statement is compatible with smooth morphisms and field extensions as explained in Theorem \ref{thm:FRFlow}.
  \end{enumerate}
 \end{theorem} 
\begin{proof}
By the nested-regular rectification Theorem \ref{thm:FRFlow}, there exists nested-regular coordinate system $(x,y) \in \widehat{\cO}_{X,(x)}$ such that $\partial\sim\partial_x$ and a lifting $\ell: \cO_{H,\pa} \to \widehat{\cO}_{X,(x)}$ such that $\partial(\ell(f))\equiv 0$ for all $f\in \cO_{H,\pa}$. In particular, $\partial =u \partial_x$, where $u$ is a unit, and the derivation $\partial_x$ belongs to $\widehat{\cF}_{(x)}$. The functorial statements of the Theorem now hold true because of the functorial statements of Theorem \ref{thm:FRFlow} and by observing that $\partial_x = \partial(x)^{-1} \partial$. 

It remains to show the splitting formula. In fact, we claim that there exist derivations $\nabla_1,\ldots,\nabla_r \in \widehat{\cF}_{(x)}$, such that $\nabla_j(x) \equiv 0$ and $[\nabla_j,\partial_x] \equiv 0$ for $i=1,\ldots,s$ and
\[
\widehat{\cF}_{(x)} = \spa_{\widehat{\cO}_{X,(x)}}(\partial_x, \nabla_1,\ldots,\nabla_{s}).
\]
Note that Theorem follows from the claim. In fact, by Lemma \ref{Lem:restriction-and-ideal}:
\[
\cF_{|H}=  \spa_{\cO_{H,\pa}}(\nabla_{1|X},\ldots,\nabla_{s|X})
\]
and, from the fact that $\nabla_j(x) \equiv 0$ and $[\nabla_j,\partial_x] \equiv 0$, we get that $d\ell_{x,\partial_x}(\nabla_{j|X})=\nabla_j$ as needed.

We turn to the proof of the claim. Let $\partial_1 ,\ldots , \partial_r \in \widehat{\cF}_{(x)}$ be arbitrary derivations such that \(
\widehat{\cF}_{(x)} = \mbox{Span}_{\cO_X}({\partial_x}, \partial_1,\ldots , \partial_r).
\)
We now modify the derivatives $\partial_1 ,\ldots , \partial_r$ in order to obtain all properties of the claim. Start by noting that, after replacing $\partial_i$ by $\partial_i - \partial_i(x){\partial_x}$, we may suppose that $\partial_i(x)\equiv 0$. Combined with the fact that $\cF$ is involutive, we get:
\[
[{\partial_x},\partial_i] = \sum^r_{j=1} A_{ij}(x,y) \partial_j, \quad i=1,\ldots,r
\]
where $A_{ij} \in \widehat{\cO}_{X,(x)}$. Consider a derivation of the generic form $\nabla=\sum^r_{i=1} \mu_i \partial_i$, where $\mu_i \in \widehat{\cO}_{X,\pa}$ will play the role of indeterminate. Then:
\[
[{\partial_x},\nabla]= \sum^r_{j=1} \left( \partial_x(\mu_j) +  \sum^r_{i=1} \mu_i A_{ij}(x,y) \right)\, \partial_j,
\]
which leads to the system of equations:
\[
{\partial_x}(\mu_j) +  \sum^r_{i=1} \mu_i A_{ij}(x,y) = 0, \quad j=1,\ldots,r.
\]
This system admits as fundamental solution
\[
\mu= \exp \left(-\int_0^x A(t,y)\, dt\right),
\]
where $A = [A_{ij}] \in M(r \times r, {\widehat{\cO}_{X,(x)}})$. Note that in the analytic case $A$ and, therefore, $\mu$ are both convergent. In the algebraic case, note that since
\[
\int_0^x A(t,y) dt=xB(x,y)\in x\cdot M(r\times r,{\widehat{\cO}_{X,(x)}})
\]
for some matrix $B$ with ${\widehat{\cO}_{X,(x)}}$-coefficients, the fundamental solution  $\mu=e^{-{xB(x)}}$ is a well defined matrix with entries in $\widehat{\cO}_{X,(x)} \simeq \cO_{V(x)}\lb x \rb$, 
where the isomorphism is defined via the lifting $\ell$. Note that only finitely many powers $(xB)^k$ contribute to each coefficient $C_i\in M(r\times r,{\ell({\cO}_{V(x),\pa})})$ of
\[
\mu:=C_0+C_1x+\ldots +C_ix^i+\ldots.
\]
Therefore, there exists $r$ formal solutions $\vec{\mu}_i = (\mu_{i1},...,\mu_{ir})$ such that $\vec{\mu}_i(0) = e_i = (0,\ldots,0, \allowbreak1,0,\ldots,0)$ (where the $1$ is on the $i$-position). We denote by \(\nabla_i := \sum^r_{j=1} \mu_{ij} \partial_{j}\) the derivation associated to each of these solutions. One checks by direct computation that they satisfy all properties of the claim, by construction.
\end{proof}

The proof highlights the following notion, which plays an important role in the notion of $\cF$-aligned centers given in Definition \ref{def:Faligned} below:

\begin{definition}[Derivation independent of $(x,\partial)$]\label{def:DerIndependentx} Consider a local ring $\cO_{X,\pa} \subset \cO \subset \widehat{\cO}_{X,\pa}$, a coordinate $x\in \cO$ and derivation over $\cO$ such that $\partial(x)=1$ (in practice, $\cO = \cO_{X,\pa}$ or $\widehat{\cO}_{X,(x)}$ or $\widehat{\cO}_{X,\pa}$). A derivation $\nabla$ over $\cO$ is said to be \emph{independent of $(x,\partial)$} if $\nabla(x) \equiv 0$ and $[\nabla,\partial] \equiv 0$.
\end{definition}

\section{Centers, invariant and maximal contact for Rees algebras}\label{sec:MainObjects}

\subsection{Rees algebras and ideal sheaves}\label{sec:Rees-and-Q-ideals} We start by recalling the basic notions of ideal sheaves and Rees algebras we need, starting by the notion of {\it valuative $\QQ$-ideals ${J}$} (or simply  {\it $\QQ$-ideals ${J}$}), introduced in \cite{ATW-weighted} and closely related to \emph{idealistic exponents} of \cite{Hironaka-idealistic}. In the simplest case they could be directly related to the notion of  {\it rational powers of ideals} considered by Huneke-Swanson (\cite[Section 10.5]{HS06}). They naturally generalize the notion of ideals and can be used for a concise description of centers. The following definition is an adaptation of the more general definition from \cite{ATW-weighted} to a simpler case as in \cite{Wlodarczyk-cobordant}. It applies also in the analytic setting.

\begin{definition}[$\QQ$-ideal] By  a {\it $\QQ$-ideal ${J}$} on a smooth variety $X$ we mean an equivalence  class  of  formal expressions  ${J}^{1/n}$, where ${J}$ is the ideal on $X$ and $n\in \NN$ is a natural number.  We say that ${J}^{1/n}$ and $\cI^{1/m}$ are {\it equivalent} if the integral closures of
${J}^{m}$ and $\cI^{n}$ are equal:
\[
({J}^m)^\inte=(\cI^m)^\inte.
\]
\end{definition}

\begin{remark} If ${J}^{1/n}$ and $\cI^{1/m}$ represent the same $\QQ$-ideal for some ideals $\cI$ and ${J}$ and $m,n\in \NN$, then in Rush \cite{Rush07}
one says that  the ideals $\cI$ and ${J}$ are {\it projectively equivalent} with respect to the coefficient $m/n$.
	
\end{remark}
  One naturally considers  the operations of addition $\cI^{1/n}+{J}^{1/m}=(\cI^m+{J}^n)^{1/mn}$, and multiplication $\cI^{1/n}\cdot{J}^{1/m}=(\cI^m{J}^n)^{1/mn}$. Any  $\QQ$-ideal ${J}=\cI^{1/n}$ on $X$ defines the {\it ideal of sections ${J}_X:=\{f\in \cO_X\mid f^n\in \cI^{\inte}\}$}. In particular, $\cI^1_X=\cI^{\inte}$ is the integral closure of $\cI$. The ideals
\[
  \cI_{m/n}:=(\cI^{m/n})_X
 \]
 are known as the {\it rational powers of $\cI$} and were considered by Huneke--Swanson in \cite{HS06}. Consequently one can associate  with any $\QQ$-ideal ${J}$  the graded algebra of $\QQ$-ideals
\[
  \cO_X[{J} t]=\cO_X\oplus {J} t\oplus {J}^2 t^2\oplus\ldots,
  \]
 and the associated (integer) Rees algebra
\[
 \cO_X[{J} t]_X:=\cO_X\oplus {J}_X t\oplus ({J}^2)_X t^2\oplus\ldots,
 \]
as in \cite{ATW-weighted} and \cite{Wlodarczyk-cobordant}. This leads to the observation made by Quek in \cite{Quek} that $\QQ$-ideals are in bijection with integrally closed Rees algebras. Under this correspondence the $\QQ$-ideal ${J}=\cI^{1/n}$ corresponds  to the integral closure $\cO_X[\cI t^n]^\inte$. In particular, as was observed in \cite{ATW-weighted} we associate with a $\QQ$-ideal ${J}$ locally of the form ${J}=(x_1^{1/w_1},\dots,x_k^{1/w_k})$ the algebra
\[
\cA_J=\cO_X[x_1t^{w_1},\ldots,x_kt^{w_k}]^\inte.
\]
In this paper we are going to consider Rees algebras on a smooth variety $X$ with grading given by  finitely generated additive subsemigoups $\Gamma\subset\QQ_{\geq 0}$.

\begin{definition}[Rationally graded Rees algebras: {\cite[Def. 2.2.1]{Wlodarczyk-cobordant}}]\label{def:ReesAlgebra} By a {\it rationally graded Rees algebra} or simply {\it Rees algebra} we mean a finitely generated $\cO_X$-algebra (or $\cO_{X,\pa}$, or $\widehat{\cO}_{X,(x)}$) which can be  written of the form:
\[
\cR=\bigoplus_{a\in \Gamma} R_at^a\subset \cO_X[t^{1/w_R}],
\]
 where $w_R\in \QQ_{>0}$ is the smallest rational number such that $\Gamma\subseteq (1/w_R)\cdot \ZZ_{\geq 0}$, and the ideals $R_a\subseteq \cO_X$ satisfy: (1) $R_0=\cO_X$ and (2) $R_a\cdot R_b\subseteq R_{a+b}$. Moreover, given a Rees algebra $\cR$ we consider:
\begin{enumerate}
\item The Rees algebra $\cR^{\ext}:=\cR[t^{-1/w_R}]$ will be called the {\it extended Rees algebra}.
\item The integral closure of $\cR$ in $\cO_X[t^{1/w_R}]$ is denote $\cR^\inte$.
\item  The {\it vertex}  $V(\cR)$  of $\cR$ (or $V(\cR^{\ext})$ of $\cR^{\ext}$) is the vanishing locus
\[
V(\cR):= V\left(\sum_{a>0} R_a \right)
\]
\end{enumerate}
\end{definition}

\begin{warning}
From now on, by a Rees algebra we mean a \emph{rationally graded} Rees algebra.
\end{warning}

\begin{remark}
We do not assume here that $R_a\subseteq R_{b}$  if $a\geq b$ for $a,b\in \Gamma$. However this condition is satisfied if
$\cR=\cR^\inte$ is integrally closed in $\cO_X[t^{1/w_R}]$.
\end{remark}

When making operations with two Rees algebras $\cR_1$ and $\cR_2$, such as their sum, we may run into the inconvenience that $w_{R_1} \neq w_{R_2}$. Of course, it is possible to remedy this by simply considering $\cR_1$ and $\cR_2$ as Rees algebras contained in $\cO_{X}[t^{1/w}]$ with $w$ being the least common multiple of $w_{R_1}$ and $w_{R_2}$. In order to avoid making this operation precise every time, nor keeping track of the actual $w$ which is necessary, we establish the following convention:

\begin{notation}
Given a Rees algebra $\cR$, we denote by:
\[
\cR^{\Int}  = \left(\cR \cdot \cO_X[t^{1/nw_R}] \right)^{\inte},
\]
where the integrable closure is taken on $\cO_X[t^{1/nw_R}]$ for some unspecified $n\in \mathbb{N}$.
\end{notation}

In other words, $\cR^{\Int} $ is the ``integrable closure of $\cR$ in $\cO_X[t^{1/nw_R}]$, for some sufficiently big $n\in \mathbb{N}$". This notation will be useful whenever we need to make an operation with Rees algebras where keeping track of the precise $n$ is unnecessary. The termonology of \emph{valuative $\QQ$-ideals of \cite{ATW-weighted} has the advantage that it identifies all these integral closures as the same object.}

\begin{remark}[Relation between $\cR$ and $\cR^{\Int}$]\label{rk:ReExtendedNonExtended}
Let $f \in R_a$ and let $b\in \mathbb{Q}_{>0}$ such that $b<a$, be fixed. Then we may always assume that $ft^{a-b} \in \cR^{\Int}$. Indeed, since we work with the integral closure, we may suppose without loss of generality that $a\in \mathbb{N}$ and $b = 1/q$ for $q\in \mathbb{N}$. In this case, since
\(
f^{aq-1} \in R_{a(aq-1)},
\)
and $f\in R_0$, we get that $f^{aq} \in R_{a(aq-1)}$. By taking the integral closure, we get that $f t^{a-1/q} \in R_{a-1/q}t^{a-1/q}$ as we wanted to prove.
\end{remark}

Let us finish by explaining how to make a dictionary between (usual) ideal sheaves and Rees algebras.

\begin{remark}[Relation between ideal sheaves and Rees algebras]\label{rk:Dictionary} 
Let $\cI$ be a coherent ideal sheaf over $\cO_X$. Consider the standard Rees algebra $\cR^0_\cI := \cO_X[t\cI] := \oplus_{a\in \NN} \cI^at^a $. We may replace it by its integral closure $\cR_\cI = \cO_X[t\cI]^{\inte}$, and note that, as usual, it is finitely generated since our rings are Japanese.

Conversely, consider a finitely generated Rees algebra $\cR=\oplus_{a \in \Gamma}  \cI_a t^a$ for some cyclic subgroup $\Gamma\subset \QQ$ over $\cO_X$. Say $\cR$  is locally  generated by a finite number of terms $f_i t^{b_i/a_i}$ for $i=1,\ldots,r$, where $f_i \in \cO_{X}$  and $b_i,a_i \in \mathbb{Z}_{>0}$. Now, note that $f_i^{a_i} t^{b_i}$ also belongs to $\cR$.

Following \cite{Wlodarczyk-cobordant} let $b = \mbox{lcm}(b_1/a_1,\ldots,b_r/a_r)$ be the smallest positive rational number which is a common integer multiple $$b=n_1b_1/a_1=\ldots=n_kb_k/a_k,$$ with $n_i\in \NN$,   
  and consider:
\begin{equation}\label{eq:ReesToIdeals}
\cI_{b}(\cR):= \Span_{\cO_{X}}( f_1^{n_1}, \ldots, f_r^{n_r} )^{\inte}.
\end{equation} 
It is immediate to show that, up to rescaling the power of $t$ used, 
\[
\cR(\cI_{b}(\cR))^{\Int} = \cR^{\Int} \quad \text{ and } \quad \cI_b(\cR(\cI))^{\Int} = \cI^{\Int}.
\]
In particular, every notion that we introduce for a Rees algebra $\cR$ can be extended to ideal sheaves $\cI$ by considering  $\cR(\cI)$, see also \cite[Corollary 2.11]{Quek}. From now on, we focus our exposition in finitely generated Rees algebras instead of ideal sheaves.
\end{remark}

\subsection{Preliminaries on weighted resolution of singularities}\label{ssec:CanonicalInvariant}
We recall two important notions introduced in \cite{ATW-weighted} and \cite{Wlodarczyk-cobordant}: $E$-\emph{adapted} centers and the \emph{canonical invariant} of a Rees algebra.

In \cite{ATW-weighted}, a \emph{center} is a $\QQ$-ideal locally  of the form $(x_1^{a_1},\ldots,x_k^{a_k})$, where $(x_1,\ldots,x_k)$ is a partial coordinate system.  We shall always assume that $a_1\leq\ldots\leq a_k$. As in \cite{Wlodarczyk-cobordant} we identify any center $J=(x_1^{a_1},\ldots, x_k^{a_k})$  with the associated Rees algebra
 \[
 \cA_J=\cO_{X,\pa}[x_1t^{1/a_1},\ldots,x_l t^{1/a_l}]^{\Int} =: \cO_{X,\pa}[x_it^{1/a_i}]^{\Int}_{l},
 \]
where we use the concise notation $\cO_{X,\pa}[x_it^{1/a_i}]^{\Int}_l$ where $l$ stands for the number of coordinates $x_i$. Over a logarithmic variety, we will further use:

\begin{definition}[Adapted center, compare {\cite[Def 3.1.7]{Wlodarczyk-cobordant}}]\label{def:EAdaptedCenter}
 Let $E$ be an SNC divisor on a smooth variety $X$. We say that a center $\cA_J$ is {\it adapted to  $E$} at point $\pa\in X$ if it can be written as
\[
\begin{aligned}
\cA_J \cdot \cO_{X,\pa}&=\cO_{X,\pa}[x_1t^{1/a_1},\ldots,x_lt^{1/a_l},z_1t^{1/c_1},\ldots,z_st^{1/c_s}]^\Int \\
&=:\cO_{X,\pa}[x_it^{1/a_i},z_kt^{1/c_k}]^\Int_{l,s},
\end{aligned}
\]
where $x_1,\ldots,x_l$ are $E$-free variables, $z_1,\ldots,z_s$ are $E$-divisorial variables, and $a_1 \leq \ldots, \leq a_l$ and $c_1 \leq \ldots \leq c_s$.
 \end{definition}

We may now introduce the canonical invariant $\inv$ used in the weighted resolution of singularities following \cite{ABQTW}, which is a variation of \cite{ATW-weighted,Wlodarczyk-cobordant}. The invariant was originally described in terms of $\QQ$-ideals, and the following Rees algebra version follows ideas from W{\l}odarczyk in \cite{Wlodarczyk-cobordant}. To this end, we consider the set:
\[
\QQ_+:=\QQ_{\geq 0}\,\amalg\, \{\infty + c \mid c\in \QQ_{\geq 0}\}.
\]
which can be endowed with a total order in an obvious way.

\begin{definition}[Canonical invariant for centers]\label{def:CanonicalInvCenters}
The canonical invariant of an adapted center
\(
\cA_J=\cO_{X,\pa}[x_it^{a_i},z_kt^{1/c_k}]^{\Int}_{l,s},
\)
is:
\[
\inv(\cA_J):=(a_1,\ldots,a_l,\infty + c_1,\ldots,\infty+c_s) \in (\QQ_+)^{l+s}.
\]
\end{definition}

We will later prove that this notion is well-defined, see Lemma \ref{lem:BasicPropertyFInvariantCenter}. Note that the canonical invariant of an adapted center takes value in:
\[
(\QQ_{+})^{\leq n}\ \  :=\ \  \bigsqcup_{l\leq n}\QQ_{+}^{l},
\]
where $n=\mbox{dim}X$. We endow \((\QQ_{+})^{\leq n}\) with a total order: for any sequence $\vec{v} \in (\QQ_+)^l$, consider its completion into a vector with $n$-entries by adding $\infty + \infty$ in entries $l+1$ until $n$. Then, the lexicographical order is total over $(\QQ_+)^{\leq n}$.\footnote{For example, $(2,3)<(2,4)<(2)<(3,3) < (\infty+1,\infty+5).$}

\begin{definition}[Canonical invariant for Rees algebras, {compare {\cite[Def 3.1.7]{Wlodarczyk-cobordant}}}]\label{def:CanonicalInvGeneral} \hfill

\noindent Given a  Rees algebra $\cR$, we say that an adapted center $\cA_J$ is \emph{$\cR$-admissible} at a point $\pa\in X$ if \(\cR_{\pa} \subset  {\cA}_{J,\pa}\)  The {\it canonical invariant} {$\inv_{\pa}(\cR)$} is given by:
  \[
  \inv_{\pa}(\cR):=\max \{ \inv_{\pa}(\cA_J)  \mid \, \cA_J\text{ is an adapted $\cR$-admissible center at $\pa$}\}.
\]

\end{definition}

See \cite[Section 2.4]{ATW-weighted} for admissibility in terms of $\QQ$-ideals, and a precursor for non-weighted centers  in \cite[Section 9.12]{Bravo-Villamayor-Book}.

In \cite{ABQTW}, it is proven that the above notion is well-defined and functorial, that is, for any Rees algebra $\cR$ over $\cO_{X}$, there exists a unique $\cR$-admissible center $\cA_{J}$ at $\pa$ which is $\inv$-maximal, and its formation commutes with smooth morphisms and field extensions. We nevertheless deduce this result later, following Proposition \ref{prop:FCanonicalInvBasicProperties} below applied to $\cF=\cD^{\log}_X$. We call the center $\cA_J$ {\it the $\inv$-maximal center at $\pa$}.

\subsection{$\cF$-order} We now introduce a basic invariant that relates a foliation with ideal sheaves and Rees algebras, generalizing the usual order. Let $(X,\cF,E)$ be a foliated logarithmic variety. Given an ideal $\cI$ on $X$, by $\cF(\cI)$ we mean the ideal generated by $f\in \cI$ and $\partial(f)$, where $f\in \cI$ and $\partial\in \cF$. Consequently we set inductively
\[
\cF^n(\cI):=\cF(\cF^{n-1}(\cI)),\quad \cF^\infty(\cI):=\bigcup \cF^{n}(\cI).
\]
\begin{definition}[$\cF$-invariant ideals and Rees algebras]\label{def:FinvariantIdeal}
An ideal sheaf $\cI$ is $\cF$-invariant if \(
\cF(\cI) \subset \cI.
\) Similarly, a Rees algebra $\cR=\bigoplus R_a$ is $\cF$-invariant if $\cF(R_a) \subset R_a$ or all $a$.
\end{definition}

\begin{definition}[$\cF$-order]\label{def:Forder}
By the $\cF$-order of an ideal $\cI$ at a point $\pa$ we mean
\[
\ord_{\cF,\pa}(\cI):=\min  \{\{n \in \mathbb{N} \mid  \cF^n(\cI)=\cO_{X,\pa}\} \cup \{\infty\}\}.
\]
The notion of $\cF$-order can be extended to elements $f \in \cO_{X,\pa}$ by taking:
\[
\ord_{\cF,\pa}(f):=\ord_{\cF,\pa}(\cO_X\cdot f).
\]
\end{definition}

The following is immediate by considering generators of $\cF$ and $\cI$:
\begin{lemma}[Invariance and functoriality of $\cF$-order]\label{Lem:functorial} Consider a smooth morphism $X' \to X$ or field extensions $X':= X_L \to X_K = X$ with pullback foliation $\cF'$ and pullback ideal $\cI' = \cI \cO_{X'}$. 

We have  $(\cF^n(\cI))' = {\cF'}^n(\cI')$. If $\cI$ is $\cF$-invariant then $\cI'$ is $\cF'$-invariant; the converse holds if $X' \to X$ is surjective.
Finally, if $\pa' \mapsto \pa$ we have  $\ord_{\cF,\pa}(\cI) = \ord_{\cF',\pa'}(\cI')$.  
\end{lemma}

\begin{remark}[On the literature of the $\mathcal{F}$-order]\label{rk:ForderLiterature}\hfill
\begin{enumerate}
\item The notion of $\cF$-order generalizes the usual order of ideals. If $X$ is smooth and $\cF=\cD_X$ then
 \[
 \ord_{\pa}(\cI)=\min\{d\mid \cI_{\pa}\subset m^d_{X,\pa}\}=\ord_{\cD_X,\pa}(\cI).
 \]
Moreover, if $X$ is logarithmically smooth and $\cF=\cD_X^{\log}$ then $\ord_{p,\cD_X}(\cI)$ is the logarithmic order introduced in \cite{ATW-principalization}. More generally if $X\to B$ is logarithmically smooth then the $\cF$-order for the sheaf of relative logarithmic derivations $\cF=\cD_{X/B}^{\log}$ is the relative logarithmic order considered in \cite{ATW-relative}.

\item A very similar notion of $\mathcal{F}$-order, called $\mathcal{F}$-tangency order, appears in \cite{Belthesis,BelJA,BelRACSAM,BelBmon}. In fact, it is defined as:
\[
\nu_{\cF,\pa}(\cI):=\min \{n \in \mathbb{N} \mid  \cF^n(\cI)=\cF^{n+1}(\cI)\}.
\]
which coincides with $\ord_{\cF,\pa}(\cI)$ whenever the latter is finite. The invariant $\nu_{\cF,\pa}$ was used to obtain partial results of resolution of singularities of foliated manifolds, a monomialization of quasianalytic morphisms, and in applications to foliation theory \cite{BelPan}.
\end{enumerate}
\end{remark}

It follows immediately from the definition that the condition $\ord_{\cF,\pa}(\cI)>d$ means simply that $\cF^d(\cI)_{\pa}\subsetneqq \cO_{X,\pa}$ or equivalently $\pa\in V(\cF^d(\cI))$. In other words:
\begin{equation}
\label{eq:Prop1Forder}
\supp_{\cF}(\cI,d):=\{\pa\in X\mid \ord_{\cF,\pa}(\cI)\geq d\}=V(\cF^{d-1}(\cI)),
\end{equation}
for every $d\in \mathbb{N} \cup \{\infty\}$, where we use the convention that $d-1 =\infty$ when $d= \infty$. Combining this observation with the Leibniz rule, we immediately get:

\begin{lemma}[Basic properties of the $\mathcal{F}$-order]\label{lem:BasicPropForder}
Given the ideal $\cI$ on a foliated variety $(X,\mathcal{F})$, the $\cF$-order $\ord_{\cF,\pa}(\cI)$ is an upper semicontinous function on $X$. Moreover, given coherent ideals $\mathcal{I}$ and ${J}$ and foliations $\cF$ and $\cG$:
\begin{enumerate}
\item  $\ord_{\cF,\pa}(\cI\cdot {J})=\ord_{\cF,\pa}(\cI)+\ord_{\cF,\pa}({J})$
\item If $\ord_{\cF,\pa}(\cI)=a$ then for $d<a$ we have $\ord_{\cF,\pa}(\cF^{d}(\cI))=a-d$.
\item If $\mathcal{I} \subset {J}$ then $\ord_{\cF,\pa}(\cI)\geq \ord_{\cF,\pa}({J})$
\item If $\cF\subset \cG$ then	$\ord_{\cF,\pa}(\cI)\geq \ord_{\cG,\pa}(\cI)$.
\end{enumerate}
\end{lemma}

We finish this section by remarking that the definition and properties of the $\cF$-order  extend to Rees algebras. In fact:

\begin{definition}[$\cF$-order for Rees algebras]\label{def:ForderRees}

By the $\cF$-order of a finitely generated Rees algebra $\cR=\bigoplus R_bt^b$, we mean:
\[
\ord_{\cF,\pa}(\cR)\quad := \quad \min \{\,\ord_{\cF,\pa}(R_b)/b\  \mid\  b>0\,\}.
\]
\end{definition}

\begin{remark}\label{rk:ForderCompatibilityIdealRees}
Recall that to an ideal sheaf $\cI$ we associate the Rees algebra $ \cO_{X}[t\cI]$ (Remark \ref{rk:Dictionary}). It follows from the definition that, for every $\pa\in X$:
\[
\ord_{\cF,\pa}(\cI) = \ord_{\cF,\pa}(\cO_{X}[t\cI]).
\]
\end{remark}

Now, if $\cR=\bigoplus R_at^a$ 
is finitely generated by  $R_{a_1},\ldots,R_{a_k}$, then there exists $a\in \mathbb{N}$ such that
\(
\ord_{\cF,\pa}(\cR)=\min_i \{\ord_{\cF,\pa}(R_{a_i})/a_i\}=\ord_{\cF,\pa}(R_{a})/a.
\)
 With any such choice of $a$, we generalize equation \eqref{eq:Prop1Forder} by:
\begin{equation}
\label{eq:Prop1ForderRees}
\supp_{\cF}(\cR,d)=\{\pa\in X\mid \ord_{\cF,\pa}(\cR)\geq d\}=V(\cF^{\lceil{da-1}\rceil}(R_a)),
\end{equation}
for $d\in \mathbb{N} \cup\{\infty\}$,   where we use the convention $\lceil{da-1}\rceil=\infty$ if $d=\infty$. We may now extend the relevant parts of Lemma \ref{lem:BasicPropForder} to Rees algebras:

\begin{lemma}
[Basic properties of the $\mathcal{F}$-order for Rees algebras]\label{lem:BasicPropForderRees}
Given a finitely generated Rees algebra $\cR$ on a foliated variety $(X,\mathcal{F})$, the $\cF$-order $\ord_{\cF,\pa}(\cR)$ is an upper semicontinous function on $X$. Moreover, given finitely generated Rees algebras $\cR$ and $\cR'$ and foliations $\cF$ and $\cG$:
\begin{enumerate}
\item If $\cR\subset \cR'$ then $\ord_{\cF,\pa}(\cR)\geq \ord_{\cF,\pa}(\cR')$
\item If $\cF\subset \cG$ then	$\ord_{\cF,\pa}(\cR)\geq \ord_{\cG,\pa}(\cR')$.
\end{enumerate}
\end{lemma}

\subsection{$\cF$-aligned and $(\cR,\cF)$-admissible centers}
Let $(X,\cF,E)$ be a foliated logarithmic variety. We now establish the central notion of blow-up center which is well-posed with respect to the foliation $\cF$ and the divisor $E$:

\begin{definition}[$\cF$-aligned centers, compare {\cite[Section 4.1]{Belthesis}}]\label{def:Faligned} An adapted center $\cA_J$ is {\it $\cF$-aligned} at $\pa\in V(\cA_J)$ if there exist nested-regular coordinates $$(x,y,z)=(x_1,\dots x_l,y_1,\ldots,y_r,z_1,\ldots,z_s)\in \widehat{\cO}_{X,(x)}$$ adapted to $E$ (see Definition \ref{def:ForReg})   and a presentation
\[
\begin{aligned}
\cA_{J}\cdot \widehat{\cO}_{X,(x)} &=\widehat{\cO}_{X,(x)}[x_1t^{1/a_1},\dots x_lt^{1/a_l},y_1t^{1/b_1},\ldots,y_rt^{1/b_r}, z_1t^{1/c_1},\ldots,z_st^{1/c_s}]^{\Int}\\
& =:\widehat{\cO}_{X,(x)}[x_it^{1/a_i},y_jt^{1/b_j}, z_kt^{1/c_k}]_{l,r,s}
\end{aligned}
\]
where $(x,y)$ are free variables, $z$ are $E$-divisorial,
with $a_1\leq\ldots\leq a_l$, $b_1\leq\ldots\leq b_r$ and  $c_1\leq\ldots\leq c_s$, and a presentation
\[
\widehat{\cF}_{(x)}=\widehat{\cO}_{X,(x)}\cdot \cF=\spa(\partial_{x_1},\ldots,\partial_{x_l}, \nabla_1,\ldots,\nabla_m),
\]
called $\cF$-aligned presentation of $\cA_J$ and $\cF$, such that:
 \begin{enumerate}
\item The derivations $\nabla_j$ are independent of {$(x_i,\partial_{x_i})$}, see Definition \ref{def:DerIndependentx}.

\item The Rees algebra $\widehat{\cO}_{X,(x)}[y_jt^{1/b_j}, z_kt^{1/c_k}]_{r,s}^{\Int}$ is $\widehat{\cF}_{(x)}$-invariant, see Definition \ref{def:FinvariantIdeal}.

\end{enumerate}
If $X$ is a coherent analytic space, moreover, then we require the coordinate system to be Euclidean, instead of only nested-regular, see Remark \ref{rk:FormRegAnalytic}.
\end{definition}

Note the different roles of coordinates: the coordinates $x_i$ are $\cF$-transverse, detected by $\cF$-derivatives. The coordinates $y_i$ are the remaining free coordinates, detected by logarithmic derivatives. The coordinates $z_i$ are only detected by non-logarithmic derivatives.

\begin{definition}[$\cF$-invariant center]\label{def:Finvariant}
An $\cF$-aligned center $\cA_J$ at a point $\pa\in X$ is said to be \emph{$\cF$-invariant} at $\pa$ if there exist regular coordinates $(y,z)\in \cO_{X,\pa}$ adapted to $E$, and a presentation
\[
\cA_J=\cO_{X,\pa}[y_jt^{1/b_j},z_k t^{1/c_k}]_{r,s}^{\Int},
\]
with $b_1\leq\ldots\leq b_r$ and $c_1\leq\ldots\leq c_s$ such that $\cA_J$ is $\cF_{\pa}$-invariant.
\end{definition}

In other words, an $\cF$-invariant center is an $\cF$-aligned center with $l=0$.

It will be useful to define a weaker, formal version of $\cF$-aligned centers.

\begin{definition}[Formally $\cF$-aligned center]\label{def:FormalFaligned}
A center $\cA_J$ is \emph{formally $\cF$-aligned} if one can choose a coordinate system $(x,y,z) \in \widehat{\cO}_{X,\pa}$ satisfying the properties in Definition \ref{def:Faligned}. The corresponding presentation of $\cA_J$ will be called {\it formally $\cF$-aligned} at $\pa$.
\end{definition}

\begin{remark}[{On $\cF$-aligned centers}]\hfill
\begin{enumerate}
\item We saw in the introduction that $\cF$-aligned centers do not possess a presentation in regular coordinates. Our definition suffices to guarantee that the center is algebraic and can be blown up on the original variety.

\item We will also work with formal completions. For example, once existence of a center on $X$ is established we can check uniqueness after formal completion. This allows more flexibility with the choice of coordinates and transformations between them. In Corollary \ref{cor:CriteriaFormallyFaligned} below, we provide a simpler criterion for the existence of a formally $\cF$-aligned presentation.

\item Particular versions of  $\cF$-aligned centers were used for log resolution of morphisms $X\to B$  in \cite{ATW-relative}, where $\cF=\cD_{X/B}$, and weighted resolution of varieties \cite{ATW-weighted,Wlodarczyk-cobordant,Quek} for $\cF=\cD_{X}$.
For general foliations, the corresponding notion of $\cF$-aligned centers appeared in \cite{Belthesis,BelRACSAM,BelJA,BelRI}, under the name $\cF$-\emph{admissible}, in the case that $X$ is a smooth coherent analytic space and $\cA_J$ is a \emph{smooth} center.
\end{enumerate}
\end{remark}

We now show that an $\cF$-aligned center always admits a regular (not necessarily $\cF$-aligned) presentation, even if the original $\cF$-aligned presentation was not yet known to come from a regular center.

\begin{proposition}[Regular presentation of $\cF$-aligned centers]\label{prop:RegLiftingFAligned}
Consider a nested-regular coordinate system $(x,y,z)$ centered at a point $\pa$ and a Rees algebra:
\[
\cA = \widehat{\cO}_{X,(x)}[x_it^{1/a_i},y_jt^{1/b_j}, z_kt^{1/c_k}]^{\Int}_{l,r,s}
\]
satisfying properties (1) and (2) of definition \ref{def:Faligned}. There exist regular coordinates $\tilde{x}_i, \, \tilde{y}_j, \tilde{z}_k\in \cO_{X,\pa}$ such that
\begin{enumerate}
 \item $\tilde{x}_1=x_1$ is regular.
 \item $\tilde{x}_i-x_i\in (x_1,\ldots,x_{i-1})\subset \widehat{\cO}_{X,(x)}$ for all $i=2, \ldots,k$.
 \item  $\tilde{y}_j-y_j\in (x_1,\ldots,x_{k})^{\lceil a_k/b_1\rceil}$;
 \item  $\tilde{z}_k = z_k\xi_k(x,y,z)$, where $\xi(0)\neq 0$;
 \end{enumerate}
so that the center $\cA_J=\cO_{X,\pa} [\tilde{x}_it^{1/a_i},\tilde{y}_jt^{1/b_j}, \tilde{z}_kt^{1/c_k}]^{\Int}_{l,r,s}$ is adapted to $E$ and
\(
\cA_J \cdot \widehat{\cO}_{X,(x)}= \cA.
\)
 \end{proposition}

\begin{remark}[Regular, nested-regular and formal centers and their presentation]\label{rk:CentersPresentation}\hfill

The above Proposition explains why the name we have chosen in definition \ref{def:Faligned} should not lead to confusion. In principle, we must distinguish between the property of a \emph{center} being nested-regular or regular (that is, defined over $\cO_{X,\pa}$ or $\widehat{\cO}_{X,(x)}$) and of the \emph{$\cF$-aligned presentation} being nested-regular or regular. The above Proposition shows that every nested-regular center actually stems from a regular center, so the adjective nested-regular can be safely associated to the presentation.

We extend this convention to our use of the word ``formal": the centers in this work are always assumed to be \emph{regular} (unless we explicitly state otherwise) and the adjective \emph{formal} refers to the  presentation, as in Definitions \ref{def:FormalFaligned} and \ref{def:FCanonicalInvariantCenters}, see Lemmas \ref{lem:ReplacementFoliated} and \ref{lem:BasicPropertyFInvariantCenter}, and in proof of uniqueness in Case I of the inductive Claim \ref{claim:MainInduction}.
\end{remark}

\begin{proof}
First, note that, by Theorem \ref{thm:FRFlow} the divisorial formal coordinates $z_k$ can be replaced up to a unit by regular lifts $\tilde{z}$. Thus the considered lifting $\cA_J$ is adapted to $E$. Next, let us show that:
\[
\widehat{\cA}_J = \cA_J \cdot \widehat{\cO}_{X,\pa}=\mathcal{A} \cdot \widehat{\cO}_{X,\pa} = \widehat{\mathcal{A}},
\]
allowing us to conclude the proof, in particular the existence of $\tilde x_i, \tilde y_i$ follows since the rings have the same completions. In fact, start by noting that the monotonicity condition for integrally closed Rees algebras ($A_{a}\subseteq A_{b}$ if $a\geq b$) allows us to conclude that:

\begin{enumerate}
\item[(i)] $\tilde{x}_it^{1/a_i}-x_it^{1/a_i}\in (x_1,\ldots,x_{i-1})t^{1/a_i}\in \widehat{\cA}$ for $i=1,\ldots,k$.\\

\noindent
Indeed, note that $x_j t^{1/a_i}\in \widehat{\cA}$ for all $j<i$ since ${1/a_i}\leq {1/a_j}$. We conclude from the hypothesis that $\tilde{x}_it^{1/a_i}-x_it^{1/a_i}\in  (x_1,\ldots,x_{i-1})t^{1/a_i}\subset \widehat{\cA}$.

\smallskip

\item[(ii)] $\tilde{y}_jt^{1/b_j}-y_jt^{1/b_j}\in (x_1,\ldots,x_{i-1})^{\lceil a_{k}/b_1 \rceil} t^{1/b_j}\in \widehat{\cA}$ for all $1\leq j \leq s$.\\

\noindent
Indeed, by the above $(x_1,\ldots,x_{k})^{\lceil a_{k}/b_1 \rceil}t^{{\lceil a_{k}/b_1 \rceil}/a_k}\in \widehat{A}$, and since
\(
1/b_j\leq 1/b_1\leq {\lceil a_{k}/b_1 \rceil}/a_k
\)
we obtain from the hypothesis that
\(
\tilde{y}_jt^{1/b_j}-y_jt^{1/b_j}\in (x_1,\ldots,x_{k})^{\lceil a_{k}/b_1 \rceil} t^{1/b_j}\in \widehat{A}_{\pa}.
\)
\end{enumerate}

To finish, consider a formal automorphism $\phi: \widehat{\cO}_{X,\pa} \to \widehat{\cO}_{X,\pa}$ mapping  $x_i\mapsto \tilde{x}_i$, $y_j\mapsto \tilde{y}_j$, $z_k\mapsto \tilde{z}_k$. By (i) and (ii), we get that $\phi(\cA_{a})\subseteq \cA_{a}$, implying  that $\phi(\cA_{a})= \cA_{a}$ for every $a\in \mathbb{N}$; integral closure implies the same for $a\in \mathbb{Q}_{\geq 0}$. Thus it lifts to an automorphism of $\widehat{\cA}$ so that
\(
\widehat{\cA} = \phi(\widehat{\cA})= \widehat{\cO}_{X,\pa}[\tilde{x}_it^{1/a_i},\tilde{y}_jt^{1/b_j}, \tilde{z}_kt^{1/c_k}]^{\Int}_{l,r,s} = \widehat{\cA}_J,
\)
finishing the proof.
\end{proof}

Now that we have established the centers which we will allow, we may define the notion of admissibility with respect to $\cR$:

\begin{definition}[$(\cR,\cF)$-admissible centers]\label{def:FadmissibleCenter} We say that a center $\cA_J$ is \emph{admissible for $(\cR,\cF)$ at $\pa\in X$,} if $\cA_J$ is $\cF$-aligned and
\(
\cR \cdot \cO_{X,\pa} \subset \cA_J \cdot \cO_{X,\pa}.
\)
\end{definition}

Since every center $\cA_J$ is integrally closed, it follows that a center $\cA_J$ is $(\cR,\cF)$-admissible if and only if it is $(\cR^{\Int},\cF)$-admissible. It will be convenient to work with integrally closed Rees algebras later on.

\subsection{$\cF$-maximal contact and $\cF$-coefficient Rees algebras}\label{Sec:maximal-contact}

We now turn to the definition of two key objects in resolution of singularities, namely maximal contact hypersurfaces and the coefficient Rees algebra --- a replacement for coefficient ideals. We start by the maximal contact:

\begin{definition}[$\cF$-maximal contact for Rees algebras {(compare \cite[Section~3.6]{Wlodarczyk-cobordant})}]\label{def:FMaximalContactRees}
Let $\cR$ be a Rees algebra over $\cO_{X,\pa}$ such that $\ord_{\cF,\pa}(\cR) = a<\infty$. By an $\cF$-maximal contact we mean a local parameter $x_1$ with $\cF(x_1)_\pa = \cO_{X,\pa}$
 such that there exists $b\in \mathbb{N}$ such that $ab\in \mathbb{N}$ and $x_1 \in \cF^{ab-1}(R_b)$.
\end{definition}

This generalizes the notion of maximal contact of an ideal which appears in many previous works, beginning with \cite{Giraud, AHV-maximal}. Note that a maximal contact always exists (in characteristic 0) by the definition of $\ord_{\cF,\pa}$, see Definition \ref{def:ForderRees}. The pullback of a maximal contact under a smooth morphism or field extension is maximal contact by  Lemma \ref{Lem:functorial}. Moreover, its definition immediately implies that:

\begin{lemma}[Local property of $\cF$-maximal contact]\label{lem:FMaximalContactNeighborhood}
Let $\cR$ be a finitely generated Rees algebra over $\cO_X$ such that $\ord_{\cF,\pa}(\cR)<\infty$, and Let $x_1$ be a maximal contact element of $\cR_{\pa} = \cR \cdot \cO_{X,\pa}$ and consider an open neighborhood $U$ of $\pa$ over which $x_1$ admits a representative. Then
\[
\ord_{\cF,\pb}(\cR) < \ord_{\cF,\pa}(\cR),\quad \forall \pb\in U\setminus V(x_1).
\]
\end{lemma}

Next, in order to work by induction with respect to maximal contact hypersurfaces, we need to consider:

\begin{definition}[$\cF$-coefficient Rees algebra, {see \cite[Def 3.5.2]{Wlodarczyk-cobordant}}]\label{def:CoeffRA}
Consider a Rees algebra
\[
\cR=\cO_{X,\pa}[f_1t^{b_1},\ldots, f_st^{b_s}]
\]
such that $\ord_{\cF,\pa}(\cR)= a<\infty$. We define the coefficient Rees algebra of $(\cR,\cF)$ by:
\[
\cC(\cR,\cF)= \cC(\cR,\cF,a) =\cO_{X,\pa}[ \partial_1 \cdots \partial_{\alpha}(f)t^{b-\alpha/a}\,\,\mid \,\partial_1,\ldots,\partial_{\alpha}\in \cF_{\pa},\, f\in R_b, \, \alpha< ab]^{\Int}.
\]
\end{definition}

This is a generalization of the differential Rees algebra appearing in  \cite{Encinas-Villamayor-Rees} as a replacement of the coefficient ideals.

Note that $\cC(\cR,\cF)$ is a finitely generated Rees algebra, since both $\cR$ and $\cF_{\pa}$ are finitely generated. By Lemma \ref{Lem:functorial} the formation of $\cC(\cR,\cF)$ is compatible with field extensions and smooth morphisms.

\subsubsection{Splitting of Coefficients}

We now establish three fundamental properties of coefficient Rees algebras, analogously to the usual properties of coefficient ideals. We start by:

\begin{lemma}[$\cF$-coefficient Rees algebra of centers, {compare {\cite[Lemma 3.4.3]{Wlodarczyk-cobordant}}}]\label{lem:CoeffcientReesCenters}
Let $\cA_J$ be an $\cF$-aligned center over $\cO_{X,\pa}$. Then $\cC(\cA_J, \cF) = \cA_J$.
\end{lemma}
\begin{proof}
Note that it is enough to prove the result formally. By the definition of $\cF$-aligned center \ref{def:Faligned}, there exists formal presentations:
\[
\begin{aligned}
\widehat{\cA}_{J}&=\widehat{\cO}_{X,\pa}[x_it^{1/a_i},y_jt^{1/b_j}, z_kt^{1/c_k}]^\Int_{l,r,s}.\\
\widehat{\cF}_{\pa}&=\spa_{\widehat{\cO}_{X,\pa}}(\partial_{x_1},\ldots,\partial_{x_l},\nabla_1,\ldots,\nabla_m)
\end{aligned}
\]
where $\nabla_j$ are independent of {$(x_i,\partial_{x_i})$} and preserve $\widehat{\cA}_{J}$. Therefore, given an element $g(x,y,z)t^{b} \in \widehat{\cA}_{J}$, we know that $\nabla_j(g(x,y,z)t^{b}) \in \widehat{\cA}_{J}$, and we conclude that $t^{-1/a_1}\nabla_j(g(x,y,z)t^{b}) \in \widehat{\cA}_{J}$ since $-1/a_1$ is negative, see Remark \ref{rk:ReExtendedNonExtended}. Next, fixing $j=1,\ldots, k$, we claim that $t^{-1/a_1} \partial_{x_j}$ preserves $\widehat{\cA}_{J}$. In fact, it trivially preserves the $y$ and $z$-coordinates and all $x_it^{1/a_i}$ with $i\neq j$. Indeed, given a monomial $x^{\alpha} t^{b} = x_1^{\alpha_1} \cdots x_k^{\alpha_k} t^{b} \in \widehat \cA_{J,b}$, we have that
\(
\sum_{i=1}^{k} \alpha_i/a_i \geq b.
\)
Now, we have that
\(
t^{-1/a_1} \partial_{x_j} x^{\alpha} t^{b} = x^{\alpha}/x_j t^{b -1/a_1},
\)
which belongs to $\widehat \cA_J$ since:
\[
b - 1/a_1 \leq  b - 1/a_j \leq \sum_{i\neq j}^{k} \alpha_i/a_i  + (\alpha_j -1)/a_j.
\]
We conclude from the definition that $\widehat{\cA}_J = \cC(\widehat{\cA_J}, \cF)  = \widehat{\cC(\cA_J, \cF)} $.
\end{proof}

The following, more general  result is valid for general Rees algebras, and establishes a desired presentation of coefficient Rees algebras via $\cF$-maximal contacts.  {Compare \cite[Proposition 4.4.1]{ATW-weighted}} and {\cite[Section~3.5.4]{Wlodarczyk-cobordant}}:

\begin{proposition}[Splitting of coefficient Rees algebra] \label{prop:CoefficientSplitting}
Let $\cR$ be a finitely generated Rees algebra over $\cO_{X,\pa}$ such that $\ord_{\cF,\pa}(\cR)=a<\infty$. Let $x_1$ be a maximal contact element of $\cR$ (see Definition \ref{def:FMaximalContactRees}) and $\partial_1 \in \cF_{\pa}$ be a derivation such that $\partial_1(x_1)$ is a unit. Denote by $H = V(x_1)$. Then $\cC(\cR,\cF)$ contains the element $x_1t^{1/a}$ and the Rees algebra $\cC_{(x_1)} := \cC(\cR,\cF)\cdot \widehat{\cO}_{X,(x_1)}$ admits a splitting:
\[
\cC_{(x_1)} = \widehat{\cO}_{X,(x_1)}[ x_1t^{1/a} , \ell (\cC_{|H}) ]^{\Int},
\]
where $\ell=\ell_{x_1,\partial_1}$ is the lifting associated to $(x,\partial_1)$, see Definition \ref{def:LiftingAssociatedToD}.
\end{proposition}

\begin{proof}
By definition of $\cF$-maximal contact, there exists $b\in \mathbb{N}$ such that $x_1 \in \cF^{ab-1}(R_b)$ and $\ord_{\cF,\pa}(x_1)=1$. The first claim of the Lemma now follows by observing that
\[
 x_1 t^{1/a}\in (\cF t^{-1/a})^{ab-1}(R_bt^b)\subset \cC(\cR,\cF).
\]
Next, by Theorem \ref{thm:splittingFoliation}, we may complete the coordinate $x_1$ into a nested-regular coordinate system $(x_1,y_1,\ldots,y_{n-1}) \in \widehat{\cO}_{X,(x_1)}$ such that $\partial_1 \sim \partial_{x_1} \in \cF \cdot \widehat{\cO}_{X,(x_1)}$ and there exists a lifting $\ell_{x_1,\partial_1}$ associated to $(x_1,\partial_1)$. Since $x_1t^{1/a}$ and $\ell_{x_1,\partial_1}(\cC_{|H})$ are clearly contained in $\cC_{(x_1)}$, we have the inclusion
\(
 \widehat{\cO}_{X,(x_1)}[ x_1t^{1/a} , \ell_{x_1,\partial_1}(\cC_{|H})  ]^{\Int} \subset \cC_{(x_1)}.
\)
We now  prove the other inclusion.

    By Theorem \ref{thm:splittingFoliation},  $\widehat{\cF}_{(x_1)}$ admits a system of generators $\{\partial_{x_1},\nabla_1,\ldots,\nabla_q\}$ where $\nabla_i$ is independent of $(x_1,\partial_{x_1})$. Since it is enough to show that a set of generators of $\cC_{(x_1)}$ belongs to $\widehat{\cO}_{X,(x_1)}[ x_1t^{1/a} , \ell_{x_1,\partial_1}(\cC_{|H})  ]^{\Int}$, we only need to show that every element:
\[
{t^{b-\alpha/a} \partial_1 \cdots \partial_{\alpha}(f)\,\text{ such that } \, f\in R_b, \, \alpha< ab,\, \partial_1, \ldots, \partial_{\alpha} \in \{\partial_{x_1},\nabla_1,\ldots,\nabla_q\}},
\]
belongs to $\widehat{\cO}_{X,(x_1)}[ x_1t^{1/a} , \ell_{x_1,\partial_1}(\cC_{|H})  ]^{\Int}$. {To this end, fix $b \in \mathbb{N}$, $\alpha <ab$, $f\in R_b$ and derivations $\partial_1, \ldots, \partial_{\alpha} \in \{\partial_{x_1},\nabla_1,\ldots,\nabla_q\}$. Since $[\partial_{x_1},\nabla_i] \equiv 0$, we may suppose without loss of generality that there exists $\beta \leq \alpha$ such that $\partial_i = \partial_{x_1}$ for $i\leq \beta$ and $\partial_i \in \{\nabla_1,\ldots,\nabla_q\}$ for $i>\beta$.} We claim that:
\[
g\,t^{b-\alpha/a}\  := \ \partial_{x_1}^{\beta}({\partial_{\beta+1} \cdots \partial_{\alpha}}(f))\, t^{b-\alpha/a} \ \ \in\ \  \widehat{\cO}_{X,(x_1)}[ x_1t^{1/a} , \ell_{x_1,\partial_1}(\cC_{|H})  ]^{\Int}.
\]
Indeed, since $f\in \widehat{\cO}_{X,(x_1)}$, it admits an expansion:
\[
f = \sum_{j=0}^{\infty} x_1^{j} f_j(y), \quad f_j(y) \in \ell_{x_1,\partial_1}(\cO_{H,\pa}).
\]
It follows that:
\[
\begin{aligned}
g\,t^{b-\alpha/a} &= t^{b-\alpha/a} \sum_{j=\beta}^{\infty} \frac{j!}{(j-\beta)!} x_1^{j-\beta} {\partial_{\beta+1} \cdots \partial_{\alpha}}(f_j(y))
\end{aligned}
\]
Note now that if $j-\beta \geq ab-\alpha$, then $x_1^{j-\beta}t^{b-\alpha/a} \in {\widehat{\cO}}_{X,(x_1)}[x_1t^{1/a}]^{\Int}$. Thus:
\[
\begin{aligned}
g\,t^{b-\alpha/a} &= t^{b-\alpha/a} \sum_{j=\beta}^{ab-\alpha + \beta-1} \frac{j!}{(j-\beta)!} x_1^{j-\beta} {\partial_{\beta+1} \cdots \partial_{\alpha}}(f_j(y)) + \widetilde{F},\\
&=  \sum_{j=\beta}^{ab-\alpha + \beta-1} \frac{j!}{(j-\beta)!} ({t^{1/a}}x_1)^{j-\beta} { t^{b-(\alpha+j)/a} \partial_{\beta+1} \cdots \partial_{\alpha}}(f_j(y)) + \widetilde{F},
\end{aligned}
\]
where $\widetilde{F} \in {\widehat{\cO}}_{X,(x_1)}[x_1t^{1/a}]^{\Int}$. {It, therefore, remains to show that:}
\[
{ t^{b-(\alpha+j)/a} \partial_{\beta+1} \cdots \partial_{\alpha}(f_j(y)) \in \widehat{\cO}_{X,(x_1)}[ \ell_{x_1,\partial_1}(\cC_{|H})  ]^{\Int},}
\]
{for $j=\beta,\ldots, ab-\alpha+\beta-1$. Indeed,} since ${\widehat{\cO}}_{X,(x_1)}[x_1t^{1/a}]^{\Int}$ is $t^{-1/a}\partial_{x_1}$-invariant:
\[
\begin{aligned}
(t^{-1/a}\partial_{x_1})^{c} g\,t^{b-\alpha/a} &= t^{b-(\alpha+c)/a} \sum_{j=\beta+c}^{ab-\alpha + \beta-1} \frac{j!}{(j-\beta-c)!} x_1^{j-\beta-c} {\partial_{\beta+1} \cdots \partial_{\alpha}}(f_j(y)) + \widetilde{F}_c
\end{aligned}
\]
for every $c \leq  ab-\alpha-1$, where $\widetilde{F}_c \in {\widehat{\cO}}_{X,(x_1)}[x_1t^{1/a}]^{\Int}$. Since each of these elements belong to $\cC_{(x_1)}$, by definition, we conclude that their restriction to $H$, that is,
\[
t^{b-(\alpha+c)/a}{\partial_{\beta+1} \cdots \partial_{\alpha}}(f_{\beta+c}(y)), \quad \text{ for } c=0,\ldots,ab-\alpha -1,
\]
belongs to $\cC_{(x_1)|H} = \cC_{|H}$. Therefore, \(g\,t^{b-\alpha/a} \in \widehat{\cO}_{X,(x_1)}[ x_1t^{1/a} , \ell_{x_1,\partial_1}(\cC_{|H})  ]^{\Int}\) finishing the proof.
\end{proof}

We finish this section by stating the third property of coefficient Rees algebras, which relates admissible centers of $\cR$ and $\cC(\cR,\cF)$,  compare {{\cite[Lemma 3.5.11]{Wlodarczyk-cobordant}}}. It is analogous to established results about coefficient ideals:

\begin{lemma}[Admissibility condition of coefficient Rees algebras]\label{lem:CoefficientAdmissible}
Let $\cR$ be a finitely generated Rees algebra over $\cO_{X,\pa}$ and $\cA_{J}$ be an $\cF$-aligned center. Suppose that $\ord_{\cF,\pa}(\cR) = \ord_{\cF,\pa}(\cA_{J}) = a <\infty$ and that the element $x_1$ is an $\cF$-maximal contact for $\cR$ and $\cA_{J}$. Let $H = V(x_1)$. Then:
\[
(1) \ \cR\subseteq  \cA_J \iff (2) \ \cC(\cR,\cF) \subseteq   \cA_J \iff (3) \ \cC(\cR,\cF)_{|H} \subseteq \cA_{J|H}.
\]
\end{lemma}
\begin{proof}
It is enough to prove the result over the completion. The implication $(1)\Rightarrow  (2)$ follows from Lemma \ref{lem:CoeffcientReesCenters}, while $(2) \Rightarrow  (1)$ follows from the fact that $\cR \subset \cC(\cR,\cF)$. Next, $(2)\Rightarrow  (3)$ is immediate from the definition, and $(3)\Rightarrow (2)$ follows from Proposition \ref{prop:CoefficientSplitting}.
\end{proof}

\subsection{Foliated Replacement Lemma}\label{ssec:FolReplacementLemma}
As a first approximation, replacement Lemmas substitute technical results concerning the invariance of the resolution of singularities construction with respect to different maximal contact hypersurfaces (in particular, we won't need the notions of MC-invariant or homogenized ideals in this work). Our goal in this section is to extend the Replacement Lemma {\cite[Lemma 3.1.29]{Wlodarczyk-cobordant}} to the foliated case, see Lemma \ref{lem:ReplacementFoliated} below.

We start by two guiding remarks. First, in contrast with the usual construction, the process of constructing $\cF$-aligned centers, see Definition \ref{def:Faligned}, relies on two types of choices: the choices of the coordinates $x$ (and $y$) in the presentation of the centers and the choices of the derivatives $\partial\in \cF$ associated with $x$-coordinates. The latter choice is critical as it affects the whole coordinate system, in particular the $y$-coordinates. Second, we only need to prove Replacement Lemmas in \emph{formal} coordinates. In particular, we do not need to work with $\cF$-aligned centers, but only with formally $\cF$-aligned centers as in Definition \ref{def:FormalFaligned}.

The notion of formally $\cF$-aligned center is, nevertheless, very unstable for coordinate changes. In the inductive process of the Foliated Replacement  Lamma \ref{lem:ReplacementFoliated} below,  it is more convenient to work with \emph{weakly $\cF$-aligned centers}, an apparently more relaxed condition, see Corollary \ref{cor:CriteriaFormallyFaligned} below:

\begin{definition}[Weakly $\cF$-aligned] A center $\cA_J$ is  {\it weakly $\cF$-aligned } at a point $\pa \in V(\cA_J)$ if there exist formal coordinates $(x,y,z) =(x_1,\dots x_l,y_1,\ldots,y_r,z_1,\ldots,z_s)\in \widehat{\cO}_{X,\pa}$ adapted to a divisor $E$ and a formal presentation
\[
\widehat{\cA}_J=\widehat{\cO}_{X,\pa}\cdot \cA_J=\widehat{\cO}_{X,\pa}[x_it^{1/a_i},y_jt^{1/b_j}, z_kt^{1/c_k}]^\Int_{l,r,s}
\]
for which
\(
a_1\leq\ldots\leq a_l, \,b_1\leq \ldots \leq b_r \text{ and }c_1\leq\ldots \leq c_s,
\)
 such that
 \begin{enumerate}
\item $\widehat{\cF}=\widehat{\cO}_{X,\pa}\cdot \cF=\spa(\partial_{1},\ldots,\partial_{l}, \nabla_1,\ldots,\nabla_m),$ where  $\det(\partial_i(x_j))$  is invertible at $\pa$ (so, in particular $(x_1,\ldots,x_l)$ is $\cF$-transverse) and $\nabla_j$ are independent of $(x_i,\partial_{x_i})$, see Definition \ref{def:DerIndependentx};

\item The Rees algebra $\widehat{\cO}_{X,\pa}[y_jt^{1/b_j}, z_kt^{1/c_k}]^\Int_{r,s}$  is $\widehat{\cF}$-invariant, see Definition \ref{def:FinvariantIdeal}.
\end{enumerate}
\end{definition}

Note the difference from the formally $\cF$-aligned property: the derivatives $\partial_i$ are not associated to the individual coordinates $x_i$.

We are now ready to state the main result of this section, implying that the notions are nevertheless equivalent (see the corollary):

\begin{lemma}[Foliated Replacement Lemma] \label{lem:ReplacementFoliated}
Let
\[
\widehat{\cA}_{J}=\widehat{\cO}_{X,\pa}[x_i't^{1/a_i},y_j't^{1/b_j}, z_k't^{1/c_k}]^\Int_{l,r,s} ,
\]
with $l\geq 1$, be a  weakly $\cF$-aligned presentation of a  center $\cA_J$ at a point $\pa\in V(\cA_J)$.
Let $x_1$ be an $\cF$-maximal contact for $\cA_J$ (see Definition \ref{def:FMaximalContactRees}) and $\partial_1 \in \cF_{\pa}$ be a derivation such that $\partial_1(x_1)$ is a unit. Then there exist formal coordinates $(x,y,z)$ extending $x_1$ centered at $\pa$, giving a \emph{formally} $\cF$-aligned presentation of the same center
\[
\widehat{\cA}_{J}=\widehat{\cO}_{X,\pa}[x_it^{1/a_i},y_jt^{1/b_j}, z_kt^{1/c_k}]^\Int_{l,r,s}.
\]
Moreover, $\widehat{\cA}_{J}$ admits a splitting:
\[
\widehat{\cA}_{J,\pa} = \widehat{\cO}_{X,\pa}[x_1t^{1/a_1}, \ell (\widehat{\cA}_{J |H}) ]^{\Int},
\]
where $\widehat{\cA}_{J |H}$ is formally $\cF_{|H}$-aligned and $\ell$ is the lifting associated with $(x_1,\partial_1)$.
\end{lemma}
\begin{proof}
We prove the result by induction on the dimension of $X$, where we add to the statement of the Lemma the following immediate result: if $l=0$, then $\widehat{\cA}_{J,\pa}$ is weakly $\cF$-aligned if and only if it is formally $\cF$-invariant, if and only if, it is formally $\cF$-aligned. Moreover, note that if the dimension of $X$ is zero, the Lemma is trivial. We suppose the Lemma proved whenever the dimension of the variety is smaller than $n$, and we fix $X$ of dimension $n$ and a center $\widehat{\cA}_J$ with $l>0$.

\emph{Replacing $x_1'$:} We start by noting that we may suppose that $x_1 = x_1'$. Indeed, by Lemma \ref{lem:CoeffcientReesCenters} and Proposition \ref{prop:CoefficientSplitting}, we know that $x_1't^{1/a_1}\in \cA_J$. Now, after a linear coordinate change of ${x}_i$ adapted to $E$ we can assume that $\ord_{\cF,\pa}(x_1-x_1') >1$. This determines an automorphism of $\widehat{\cO}_{X,\pa}$,
\[
{x}_1'\mapsto {x}_1,\quad  {x}_i'\mapsto {x}_i',\quad i \geq 2, \quad  {y}_j'\mapsto {y}_j', \quad  {z}_k'\mapsto {z}_k',
\]
which takes $\widehat{\cA}_{J,\pa}$ onto $\widehat{\cA}_{J,\pa}$ leading to the new presentation
 \[
 \widehat{\cA}_{J,\pa}=\widehat{\cO}_{X,\pa}[x_1t^{1/a_1}, x_2't^{1/a_2},\ldots,x_l't^{1/a_l},y_j't^{1/b_j}, z_k't^{1/c_k}]^\Int_{r,s} ,\]
Moreover, let $\partial'_1,\ldots,\partial'_l$ be such that $\det(\partial'_i(x_j'))_{i,j\geq 1}$ is a unit. Without loss of generality, we may suppose that $\partial_1'(x_1)$ is a unit. Now, we consider the derivations
\[
\partial_i = \partial_i' - \frac{\partial_i'(x_1)}{\partial_1'(x_1)}\partial_1', \quad i=2,\ldots,l
\]
and note that $\partial_i(x_1) \equiv 0$ for $i=2,\ldots,k$. It follows that $\det(\partial_i(x_j'))_{i,j\geq 2}$ is a unit, so that $\det(\partial_i(x_j'))_{i,j\geq 1}$ is also a unit. We conclude that the new presentation with $x_1=x_1'$ is weakly $\cF$-aligned  with respect to the derivatives $\{\partial_1,\ldots,\partial_l\}$.

\emph{Induction:} Next, denote by $H=V(x_1)$ and let $\ell_{x_1,\partial_1}:\cO_{H,\pa} \to \widehat{\cO}_{X,(x_1)}$ be the lifting associated to $(x_1,\partial_1)$. By Theorem \ref{thm:splittingFoliation}, Lemma \ref{lem:CoeffcientReesCenters} and Proposition \ref{prop:CoefficientSplitting}:
\[
\widehat{\cF} =  \Span_{\widehat{\cO}_{X,\pa}} (\partial_{x_1}, \cF_{|H}) \quad \text{ and }\quad  \widehat{\cA}_J =  \widehat{\cO}_{X,\pa}[x_1t^{1/a_1}, \ell_{x_1,\partial_i}(\widehat{\cA}_{J|H}) ]^{\Int}.
\]
Now, let $x_i = (x_i')_{|H}$, $y_j = (y_j')_{|H}$ and $z_k = (z_k')_{|H}$. Note that $\det(\partial_i(x_j))_{i,j\geq 2}$ is a unit and
\[
\widehat{\cO}_{H,\pa}[y_jt^{1/b_j}, z_kt^{1/c_k}]^{\Int}_{r,s} = \left(\widehat{\cO}_{X,\pa}[y_j't^{1/b_j}, z_k't^{1/c_k}]_{r,s}\right)_{|H}^{\Int}
\]
is $\widehat{\cF}_{|H}$-invariant. It follows that $\widehat{\cA}_{J|H}$ is a weakly $\cF$-aligned center over $\widehat{\cO}_{H,\pa}$. By the induction hypothesis, it admits a formally $\cF$-aligned presentation. We conclude  by lifting the new presentation via $\ell_{x_1,\partial_1}$.
\end{proof}

Note that, since a formally $\cF$-aligned center is clearly weakly $\cF$-aligned, the above Lemma immediately implies that:

\begin{corollary}[Criteria for formally $\cF$-aligned center]\label{cor:CriteriaFormallyFaligned}
 A center $J$ is formally $\cF$-aligned at $\pa\in V(J)$ if, and only if, it is weakly $\cF$-aligned at $\pa$.
\end{corollary}

\subsection{Canonical $\cF$-invariant}

 We shall now modify the invariant $\inv$, see Definition \ref{def:CanonicalInvGeneral}, to the context of foliated logarithmic varieties $(X,\cF,E)$. Just as in the case of $\inv$, we start by establishing the meaning of the new invariant over centers.

\begin{definition}[Canonical $\cF$-invariant for centers]
\label{def:FCanonicalInvariantCenters}
Let $\cA_{J}$ be a formally $\cF$-aligned center at $\pa$ with $\cF$-aligned presentation :
\[
\cA_{J} = \widehat{\cO}_{X,\pa}[x_it^{1/a_i},y_jt^{1/b_j}, z_kt^{1/c_k}]_{l,r,s}^{\Int}.
\]
We define:
\[
\inv_{\cF,\pa}(\cA_{J}) = (a_1,\ldots,a_l,\infty + b_1,\ldots,\infty+b_r, \infty+\infty + c_1,\ldots, \infty + \infty + c_s).
\]
\end{definition}

The following lemma shows that this invariant is well defined on centers; its very definition shows that it is compatible with smooth morphisms and field extensions.

\begin{lemma}[Basic property of $\inv_{\cF}$ for centers]\label{lem:BasicPropertyFInvariantCenter}
The canonical $\cF$-invariant for centers is well-defined, that is, it does not depend on the formal $\cF$-aligned presentation. Moreover, given a formally $\cF$-aligned center $\cA_J$ such that $\ord_{\cF,\pa}(\cA_{J}) < \infty$ and a maximal contact element $x_1$ of $\cA_J$,  set $H=V(x_1)$. We have that:
\[
\inv_{\cF,\pa}(\cA_J) =  (\ord_{\cF,\pa}(\cA_J),\inv_{\cF_{|H},\pa}(\cA_{J|H})).
\]
\end{lemma}
\begin{proof}
We prove the result by induction on the dimension of the ambient variety $X$. Note that the result is trivial if the dimension of $X$ is zero. We now divide the proof in two distinct cases:

\medskip
\noindent
\textbf{Case I: $\ord_{\cF,\pa}(\cA_J) < \infty$.} Fix a variety $X$ of dimension $n$ and a center $\cA_J$ such that $\ord_{\cF,\pa}(\cA_J)<\infty$. We consider two formally $\cF$-aligned presentations
\[
\begin{aligned}
\widehat{\cA}_{J} &= \widehat{\cO}_{X,\pa}[x_it^{1/a_i},y_jt^{1/b_j}, z_kt^{1/c_k}]_{l,r,s}\\
 &= \widehat{\cO}_{X,\pa}[x_i't^{1/\alpha_i},y_j't^{1/\beta_j}, z_k't^{1/\gamma_k}]^{\Int}_{l',r',s'}
\end{aligned}
\]
Note that $\ord_{\cF,\pa}(\cA_J) = a_1 = \alpha_1$, so $x_1$ and $x_1'$ are two $\cF$-maximal contact elements. By the foliated Replacement Lemma \ref{lem:ReplacementFoliated} we may assume without loss of generality that $x_1=x_1'$ and that the restricted center to $H= V(x_1)$ admits two formally $\cF$-aligned presentations:
\[
\begin{aligned}
\widehat{\cA}_{J|H} &= \widehat{\cO}_{H,\pa}[\tilde{x_2}t^{1/a_2}, \ldots, \tilde{x}_kt^{1/a_k},\tilde{y}_jt^{1/b_j}, \tilde{z}_kt^{1/c_k}]_{r,s}\\
 &= \widehat{\cO}_{H,\pa}[\tilde{x}_2't^{1/\alpha_2}, \ldots, \tilde{x}_l't^{1/\alpha_{l'}},\tilde{y}_j't^{1/b_j}, \tilde{z}_k't^{1/c_k}]_{r',s'}
\end{aligned}
\]
where $\tilde{x}_i = (x_i)_{|H}$, $\tilde{y}_i = (y_i)_{|H}$, $\tilde{z}_i = (z_i)_{|H}$, $\tilde{x}_i' = (x_i')_{|H}$, $\tilde{y}_i' = (y_i')_{|H}$ and $\tilde{z}_i' = (z_i')_{|H}$. We conclude by induction.

\medskip
\noindent
\textbf{Case II: $\ord_{\cF,\pa}(\cA_J) = \infty$.} In this case, we will argue by induction with respect to three cases: (i) $\cF = \cD_X$; (ii) $\cF = \cD_X^{\log}$ and (iii) $\cF\subset \cD_X^{\log}$.

In fact, note that in case (i), then we necessarily have $\ord_{\cF,\pa}(\cA_J) < \infty$ and the proof was finished in Case I. Next, if we assume that $(ii)$ holds true, then $\ord_{\cF,\pa}(\cA_J) = \infty$ implies that $\cA_J$ is $\cD^{\log}_X$ invariant, and only depends on divisorial components. The result now follows by noting that $\inv_{\cD^{\log}_X,\pa}(\cA_J) = \infty + \inv_{\cD_X,\pa}(\cA_J)$ and case $(i)$, where we used the convention that
\(
\infty + (e_1\ldots,e_r) = (\infty + e_1, \ldots,\infty + e_r).
\). Similarly, if assume that $(iii)$ holds true, then $\ord_{\cF,\pa}(\cA_J) = \infty$ implies that $\cA_J$ is $\cF$-invariant. The result then follows by noting that $\inv_{\cF,\pa}(\cA_J) = \infty + \inv_{\cD_X^{\log},\pa}(\cA_J)$ and case $(ii)$.
\end{proof}

We can now introduce a total order over the set of invariants by completing the vectors with $\infty+\infty+\infty$ at the end, similarly to what was done in Section~\ref{ssec:CanonicalInvariant}. We are ready to define:

\begin{definition}[Canonical $\cF$-invariant]\label{def:FCanonicalInvariantGeneral}
Given a foliated logarithmic variety $(X,\cF,E)$ and a finitely generated Rees algebra $\cR$, the {\it canonical invariant} $\inv_{\cF,\pa}(\cR)$ of $\cR$ at a point $\pa\in X$ is given by:
  \[
  \inv_{\cF,\pa}(\cR):=\max \{ \inv_{\cF,\pa}(\cA_J)  \mid \, \cA_J\text{ is } \text{$(\cR,\cF)$-admissible at }\pa\}.
\]
Any $(\cR,\cF)$-admissible center $\cA_J$ for which $\inv_{\cF,\pa}(\cA_J)=\inv_{\cF,\pa}(\cR)$ will be called the {\it $\inv_{\cF,\pa}$-maximal center} for $\cR$ at $\pa$.
\end{definition}

Note that the above invariant and center are not, a priori, well-defined. In fact, it is not clear that the set of $(\cR,\cF)$-admissible centers is non-empty, nor that  the set of invariants admits a maximal element instead of only a supremum (which would have entries in $\mathbb{R}_{>0}^{+}\cup \{\infty,\infty+\infty\}$), or that the center achieving the maximum is unique. This point will be adressed in Section~\ref{sec:algorithm}, following the construction of the resolution of singularities algorithm, which allows us to prove that:

\begin{proposition}[Basic properties of $\inv_{\cF}$]\label{prop:FCanonicalInvBasicProperties}
The canonical invariant $\inv_{\cF,\pa}(\cR)$ is well-defined for every $\pa\in X$.  The function $\inv_{\cF}(\cR)$ defined over $X$ takes values in a well-ordered set, is upper semicontinuous, and functorial for smooth morphisms and field extensions. 
 Furthermore, the $\inv_{\cF,\pa}$-maximal center at $\pa$ is unique and functorial for smooth morphisms and field extensions.

Finally, the following holds true:
\begin{enumerate}
\item If $\cR\subseteq \cR'$ then $\inv_{\cF,\pa}(\cR)\geq \inv_{\cF,\pa}(\cR')$.
\item If $\cF\subseteq \cF'$ then $\inv_{\cF,\pa}(\cR)\geq \inv_{p,\cF'}(\cR)$.	\qed
\end{enumerate}
\end{proposition}

Moreover that properties (1) and (2) are immediate from the definition \ref{def:FCanonicalInvariantGeneral}, as long as $\inv_{\cF,\pa}$ is well-defined. We only use this Proposition starting from Section~\ref{sec:BlowUp}, after proving it --- in more precise form --- in Section \ref{sec:algorithm}.

\section{Construction of the $\inv_{\cF}$-maximal admissible center}\label{sec:algorithm}

\subsection{Inductive statement} We prove the following result by induction on $\dim X$:

\begin{Inductiveclaim}\label{claim:MainInduction}
Let $(X,\cF,E)$ be a foliated logarithmic variety and consider a Rees algebra $\cR$ over $\cO_X$. For every $\pa\in X$, there exists a unique $\inv_{\cF,\pa}$-maximal $(\cR,\cF)$-admissible center $\cA_J$ defined in a neighborhood $U$ of $\pa$. This center satisfies the following conditions:
\begin{enumerate}
\item The center $\cA_J$ is also maximal with respect to formal $\cF$-aligned centers, that is, $\inv_{\cF,\pa}(\cA_J)$ is also equal to:
\[
\max \{ \inv_{\cF,\pa}(\cA_{J'})  \mid \, \cA_{J'}\text{ is } \text{formally $\cF$-aligned and }\cR\cdot \cO_{X,\pa} \subset \cA_{J'}\cdot \cO_{X,\pa} \}.
\]
\item\label{It:induction-wo-usc} The invariant $\inv_{\cF}$ takes values in a well-ordered set, and is upper semi-continuous.
\item\label{It:claim-functoriality}  The invariant $\inv_{\cF,\pa}(\cR) $ and the $\inv_{\cF}$-maximal $(\cR,\cF)$-admissible center $\cA_J$ are functorial under smooth morphisms, field extensions, group actions preserving $\cR$, and  derivations $\delta$ such that both $\cF$ and $\cR$ are $\delta$-invariant,  see Definitions \ref{def:MorphismLogFoliatedVariety}, \ref{def:GsemiinvariantFol}, \ref{def:deltainvariantFol}, and \ref{def:fieldextension}. Here $\cR$ is $\delta$-invariant if $\delta(\cR):=\{\delta(f)\mid f\in \cR \}\subseteq \cR$.

\end{enumerate}
\end{Inductiveclaim}

\begin{remark}
One can also prove the following variant of the $\delta$-invariance statement of Part (\ref{It:claim-functoriality}) in the inductive claim, which may be of value in applications, though we are not aware of one: given a \emph{distribution} $\Delta$ such that $[\Delta,\cF] \subset \cF$ and $\Delta(\cR) \subset \cR$, then $\Delta(\cA_J) \subset \cA_J$.   Note that the condition $[\Delta,\cF] \subset \cF$ is very strong.
\end{remark}

If the dimension of $X$ is $0$, the statement is trivial. So assume that the result is proved whenever 
$\dim (X)\leq\,n-1$ and consider a foliated logarithmic variety $(X,\cF,E)$ of dimension $n$. Fix a point $\pa \in X$; note that for proving property (1) it is enough to consider a Rees algebra $\cR$ over $\cO_{X,\pa}$. In addition, we may also suppose that $\cR$ is integrally closed by the following reason: on the one hand $\cR \subset \cR^\Int$, so every $(\cR^{\Int},\cF)$-admissible center is also $(\cR,\cF)$-admissible. On the other hand, given an $(\cR,\cF)$-admissible center $\cA_J$:
\[
\cR \subset \cA_{J} \implies \cR^{\Int} \subset \cA_{J}^{\Int} = \cA_{J}
\]
implying that it is also $(\cR^{\Int},\cF)$-admissible. The proof  is now divided into two cases.

\subsection{Case I: $\ord_{\cF,\pa}(\cR) < \infty$} 
\subsubsection{Existence of an $\inv_\cF$-maximal center} Suppose that $\ord_{\cF,\pa}(\cR) = a < \infty$. Let $x_1$ be an $\cF$-maximal contact element with respect to $\cR$, see Definition \ref{def:FMaximalContactRees}, and let $\partial_1 \in \cF_{\pa}$ be a derivation such that $\partial_1(x_1)$ is a unit. Set $H = V(x_1)$. By Theorem \ref{thm:splittingFoliation}, there exists a nested-regular coordinate system $(x_1,x_2,\ldots,x_k,y_1,\ldots,y_r) \in \widehat{\cO}_{X,(x_1)}$ such that
\[
\cF_{(x_1)} = \cF \cdot \widehat{\cO}_{X,(x_1)} = \Span_{\widehat{\cO}_{X,(x_1)}} (\partial_{x_1}, \cF_{|H})
\]
where $\cF_{|H}$ is {the foliation} over $\cO_{H,\pa}$ which we associate to a subset of $\cF_{(x_1)}$ via the lifting $\ell_{x_1,\partial_1}: \cO_{H,\pa} \to \widehat{\cO}_{X,(x_1)}$ associated to $(x_1,\partial_1)$, {see Section \ref{sec:Fol:LocalSplittingFoliations}}. In particular $\cF_{|H}$ is generated by derivations independent of {$(x_1,\partial_{x_1})$}. Now, consider the coefficient Rees Algebra, see Definition \ref{def:CoeffRA}, over $\widehat{\cO}_{X,(x_1)}$:
\[
\cC:= \widehat{\cO}_{X,(x_1)}\cdot \cC(\cR,\cF).
\]
By Proposition \ref{prop:CoefficientSplitting}, the Rees algebra $\cC$ admits a local splitting:
\[
\cC=\widehat{\cO}_{X,(x_1)}[{x}_1t^{1/{a_1}},{\ell}(\cC_{|H})]^{\Int},
\]
where $\ell=\ell_{x_1,\partial_1}$ is a lifting associated to $(x_1,\partial_1)$, see Definition \ref{def:LiftingAssociatedToD}.
Now, by induction on the dimensions, there exists an $(\cC_{|H},\cF_{|H})$-admissible center $\cA_{J'}$ defined over $\cO_{H,\pa}$ which is $\inv_{\cF_{|H}}$-maximal; in particular, it admits an $\cF_{|H}$-aligned presentation. Consider the $\widehat{\cO}_{X,(x_1)}$ Rees algebra:
\[
\mathcal{A} = \widehat{\cO}_{X,(x_1)}[{x}_1t^{1/{a_1}}, {\ell}(\cA_{J'})]^{\Int}
\]
By Proposition \ref{prop:RegLiftingFAligned}, there exists a regular center $\cA_J$ over $\cO_{X,\pa}$ such that $\cA_{J} \cdot \widehat{\cO}_{X,(x)} = \mathcal{A}$. We now claim that $\cA_{J}$ is $(\cR,\cF)$-admissible. Indeed, it is $\cF$-aligned by construction, and it follows from the definition of coefficient Rees algebra \ref{def:CoeffRA} and Lemma \ref{lem:CoefficientAdmissible} that:
\[
\cR \subset \cC(\cR,\cF) \subset \mathcal{A}_{J}.
\]
\subsubsection{Uniqueness and Statement (1)}
Let $\cA_{J'}$ be a formally $\cF$-aligned center such that $\cR \subset \cA_{J'}$ and $\inv_{\cF,\pa}(\cA_{J'}) \geq \inv_{\cF,\pa}(\cA_J)$; let us prove that $\cA_J=\cA_{J'}$. Note that we only need to prove the equality formally. 

We claim that $ \widehat{\mathcal{A}}_{{J'}}$ satisfies all of the conditions of the foliated Replacement Lemma \ref{lem:ReplacementFoliated}. Indeed, note that $\cR \subset \mathcal{A}_{{J'}}$ implies that $\ord_{\cF,\pa}(\mathcal{A}_{{J'}}) \leq \ord_{\cF,\pa}(\cR) = a$ by Lemma \ref{lem:BasicPropForder}. On the other hand, $\inv_{\cF,\pa}(\mathcal{A}_{{J'}}) \geq \inv_{\cF,\pa}(\mathcal{A}_{J})$ implies that $\ord_{\cF,\pa}(\mathcal{A}_{{J'}}) \geq \ord_{\cF,\pa}(\mathcal{A}_{J}) =a$ by Definition \ref{def:FCanonicalInvariantCenters}, therefore implying that $\ord_{\cF,\pa}(\mathcal{A}_{{J'}})=a$. Finally by Lemma \ref{lem:CoefficientAdmissible}:
\[
\cC(\widehat{\cR},\widehat{\cF}) \subset \widehat{\mathcal{A}}_{{J'}}
\]
which implies that $x_1t^{1/a} \in \widehat{\mathcal{A}}_{{J'}}$ by Proposition \ref{prop:CoefficientSplitting}. 

We may now apply the foliated Replacement Lemma \ref{lem:ReplacementFoliated} with respect to $x_1$ and $\partial_1$ --- that is, the same maximal contact element and the same derivation used in the construction of $\cA_{J}$, leading to the same lifting $\ell=\ell_{x_1,\partial_1}$ --- and suppose without loss of generality that $\cA_{{J'}}$ admits a formally $\cF$-aligned presentation:
\[
\widehat{\mathcal{A}}_{{J'}} = \widehat{\cO}_{X,\pa}[x_1t^{1/a_1},{\ell}(\mathcal{A}_{{J'}|H})]^{\Int}.
\]
Now, by Lemma \ref{lem:CoefficientAdmissible}, the restricted center $(\mathcal{A}_{{J'}})_{|H}$ is $(\cC_{|H},\cF_{|H})$-admissible. By the formal $\inv_{\cF_{|H},\pa}$-maximality of $\cA_{J'}$ (property (1) of the induction claim) and Lemma \ref{lem:BasicPropertyFInvariantCenter}, we conclude that $\mathcal{A}_{{J'}|H}=\mathcal{A}_{J'}$ over $\cO_{H,\pa}$. It follows that:
\[
\widehat{\mathcal{A}}_{{J'}} = \widehat{\cO}_{X,\pa}[x_1t^{a_1},{\ell}(\mathcal{A}_{{J'}|H})]^{\Int} = \widehat{\cO}_{X,\pa}[x_1t^{1/a_1},{\ell}(\mathcal{A}_{J'})]^{\Int} = \widehat{\mathcal{A}}_{J},
\]
proving that $\cA_J$ is unique and $\inv_{\cF,\pa}$-maximal. 

This construction, moreover, proves the following result, similar to  \cite[Section 5.1]{ATW-weighted} and \cite[Lemma 3.1.24]{Wlodarczyk-cobordant}:
\begin{lemma}[Inductive property I]\label{lem:FcanonicalInvForFiniteRees}
If $\ord_{\cF,\pa}(\cR)<\infty$ then
\[
\inv_{\cF,\pa}(\cR)=(\ord_{\cF,\pa}(\cR),\inv_{\cF_{H},\pa}(\cC(\cR,\cF)_{|H})),
\]
where $H$ is a locally defined $\cF$-maximal contact hypersurface.
\end{lemma}

\subsubsection{Semicontinuity, well-ordering, and functoriality}
 Let $\cR$ be a finitely generated Rees algebra over $\cO_X$ such $\ord_{\cF,\pa}(\cR) = a < \infty$. Denote by $U$ a neighborhood of $\pa$ where $x_1\in \cO_U$  is  maximal contact. Let $\cR_U=\cR \cdot \cO_U$. It follows from Lemma \ref{lem:FMaximalContactNeighborhood} that $\inv_{\cF,\pb}(\cR)<\inv_{\cF,\pa}(\cR)$ for every point $\pb\in U \setminus H$ (where $H = V(x_1) \subset U$). Next, note that over $H$, the invariant is given by:
\[
\inv_{\cF,\pb}(\cR) = (\ord_{\cF,\pb}(\cR),\inv_{\cF_{|H},\pb}(\cC_{|H})), \quad \forall \pb\in H
\]
The result now follows from the fact that $\ord_{\cF}(\cR)$ is upper semicontinous by Lemma \ref{lem:BasicPropForder} and $\inv_{\cF_{|H}}$ takes values in a well-ordered set and is upper semicontinous by induction. 

Similarly, the functoriality of $\inv_{\cF,\pa}(\cR)$ and the $\inv_{\cF}$-maximal admissible centers $\cA_J$ follows from the same argument, together with the functoriality of derivations, maximal contact, and coefficient Rees algebras, See Lemma \ref{Lem:functorial} and Section \ref{Sec:maximal-contact}.

\subsubsection{$\delta$-invariance}\label{Sec:delta-invariance} The idea here  is that $\delta$-invariance is in essence invariance under infinitesimal group actions. Ideally we would integrate such infinitesimal group action to a formal group action. However the present paper does not treat foliations on formal schemes, as these would introduce technicalities that go beyond simple generalization. {With further work one can show in Inductive Claim \ref{claim:MainInduction} the uniqueness of \emph{formal} centers of maximal invariant admissible for an ideal on an algebraic scheme.} Instead we integrate the action analytically, and show using the Lefschetz principle that the complex analytic case implies the algebraic case.

We therefore start with a \emph{complex analytic} manifold $X$ and consider the product ${X_\tau} := X \times \CC$ with the natural morphism $\pi:{X_\tau}\to X$. The local rings of $X_\tau$ are isomorphic to $\cO_X\{\tau \}$, power series in $\tau$ with coefficient in $\cO_X$ that are each convergent on an open set.  It carries an induced Rees algebra $\cR_\tau = \cR \cO_{X_\tau}$ and a fiberwise induced foliation $\cF_\tau = \cF \times \{0\} \subset \cD^{\log}_{X_\tau/\CC}$.

We associate with the derivation $\delta\in D^{\log}_{X,a}$ the  flow $\bar\phi_\tau^\delta$ defined in a neighborhood of $\pa\in X$, which in turn gives a locally-defined morphism from $X_\tau$ to itself, whose corresponding action on local rings is given by 
$$e^{\tau\delta}:=I+\delta \tau+\frac{1}{2!}\delta^2\tau^2+\ldots$$  
so that $(\bar\phi_\tau^\delta)^* (f) = e^{\tau\delta}(f)$, see e.g. \cite[Chapter 5.3]{SagleWalde}.

This automorphism  preserves the Rees algebra ${\cR_\tau}$. Given $\partial\in \cF$ we have $$(e^{\tau\delta})^*(\partial)=(e^{\tau\delta})(\partial)e^{-\tau\delta}=(I+[\delta,\,\cdot]\tau+1/2[\delta,[\delta,\,\cdot]]\tau^2+\cdots)(\partial)\in {\cF_\tau}$$ so ${\cF_\tau}$ is also preserved.

The center $\cO_{{X_\tau}}\cdot \cA_{J}=\cA_{J}\{\tau\}$ is maximal 	$\cF$-aligned at $\pa\in X_\tau$. It is  thus $e^{\tau\delta}$-invariant. The latter implies that $\delta(\cA_{J})\subseteq \cA_{J}$ since for any $ft^b\in \cA_{{J}}\subset \cA_{{J}}\{\tau\}$ we have $$e^{\tau\delta}(ft^b)=f+ \delta(ft^b)\tau+\ldots\in \cA_{{J}}\{\tau\},$$
whence $\delta(ft^b)\in \cA_{{J}}.$ This completes the complex analytic case.

Next, note that the complex algebraic case immediately follows from the
analytic case, since every complex algebraic variety admits an analytic
structure, and $\cA_J$ being $\delta$-invariant can be tested on analytifications.

        This implies the result when the base field is any subfield $K_0 \subset \CC$, since, given $(X,\cF,\cR,\delta)$ and the corresponding center $\cA$ over $K_0$, the statement that $\delta(\cR_a) = 0 \in \cO/\cR_a$ over $K_0$ is equivalent to the statement after extending to $\CC$. If $K$ is general, all the data $(X,\cF,\cR,\delta)$ are defined over a finitely generated subfield $K_0 \subset \CC$, and the same is true for the unique center $\cA$. Once again, the statement that $\delta(\cR_a) = 0 \in \cO/\cR_a$ over $K$ is equivalent to the statement over $K_0$.

\subsection{Case II: $\ord_{\cF,\pa}(\cR) = \infty$}\label{ssec:Case2} We proceed similarly to Case II of \ref{lem:BasicPropertyFInvariantCenter}.
 We argue by induction in three cases: (i) $\cF = \cD_X$; (ii) $\cF = \cD_X^{\log}$ and (iii) $\cF\subset \cD_X^{\log}$.

In fact, note that in case (i), we necessarily have $\ord_{\cF,\pa}(\cR) < \infty$ and the proof has finished in Case I. So suppose that we are either in the hypothesis (ii), or in hypothesis (iii) and that (ii) has been already proved. Suppose that $\ord_{\cF,\pa}(\cR) = \infty$ and consider the Rees algebra:
\begin{equation}\label{eq:Rinfty}
\cR^{\infty} = \cF^{\infty}(\cR) = \{ \partial^{k}(ft^{b});\, \partial \in \cF_{\pa}, \, ft^b\in R_bt^b,\, k\in \mathbb{N} \}.
\end{equation}
Note that $\cR \subset \cR^{\infty}$ by definition, and that $\cR^{\infty}$ is $\cF$-invariant. Now, under (ii) we consider the $\inv_{\cD_X}$-maximal center $\cA_{{J}}$ for $\cR^{\infty}$; under (iii), the $\inv_{\cD_X^{\log}}$-maximal center $\cA_{{J}}$ for $\cR^{\infty}$. In particular, $\cR \subset \cR^{\infty} \subset \cA_{{J}}$, and $\cA_{{J}}$ is $\cR$-admissible. 

By the property of $\delta$-invariance in the Case I --- with $\delta$ taken to be an arbitrary section of $\cF$, and $\cF$ taken to be $\cD$ or  $\cD^{\log}$ as the case may be ---  the center $\cA_{{J}}$ is $\cF$-invariant; in particular, it is $(\cR,\cF)$-admissible.

We conclude that there exists a $(\cR,\cF)$-admissible center. Now, let us prove that it is unique center that maximizes the invariant. Let $\cA_{J'}$ be another $(\cR,\cF)$-admissible center such that $\inv_{\cF,\pa}(\cA_{J'}) \geq \inv_{\cF,\pa}(\cA_{{J}})$. Since $\cA_{{J}}$ is $\cF$-invariant, its $\cF$-order is $\infty$ and we conclude that $\ord_{\cF,\pa}(\cA_{J'}) = \infty$ by Definitions \ref{def:FCanonicalInvariantCenters} and  \ref{def:FCanonicalInvariantGeneral}. Since $\cA_{J'}$ is $(\cR,\cF)$
-admissible, it is therefore $\cF$-invariant. Now:
\[
\cR \subset \cA_{{J'},\pa} \implies  \cF^{\infty}(\cR) \subset  \cF^{\infty}(\cA_{{J'},\pa}) = \cA_{{J'},\pa},
\]
and $\cA_{J'}$ is also admissible for $\cR^{\infty}$. Since $\cA_{{J}}$ was the unique maximal center for $\cR^{\infty}$, we conclude that $\cA_{{J}} =\cA_{J'}$. The construction, moreover, implies that:

\begin{lemma}[Inductive property II]\label{lem:FcanonicalInvForInvariantRees} If  $\ord_{\cF,\pa}(\cR)=\infty$ then
\[
\begin{aligned}
\inv_{\cF,\pa}(\cR)&=\infty+\inv_{\cD_X^{\log},\pa}(\cF^\infty(\cR)), &\quad &\text{ in case (iii),}\\
\inv_{\cD_X^{\log},\pa}(\cR)&=\infty+\inv_{\cD_X,\pa}(\cF^\infty(\cR)), & \quad & \text{ in case (ii).}
\end{aligned}
\]
where we used the convention that
\(
\infty + (e_1\ldots,e_n) = (\infty + e_1, \ldots,\infty + e_n).
\) The $\inv_{\cF}$-maximal $(\cR,\cF)$-admissible center is, moreover, $\cF$-invariant.
\end{lemma}

The {well-ordering of the value set and semicontinuity of $\inv_{\cF,\pa}(\cR)$ follow} by the induction with respect to (i), (ii) and (iii) and the above Lemma. Again, the functoriality of $\inv_{\cF,\pa}(\cR)$ and the $\inv_{\cF}$-maximal admissible centers ${J}$ is a consequence of the induction with respect to (i), (ii) and (iii), and the functoriality of derivations, orders, maximal contacts, and coefficient Rees algebras, See Lemma \ref{Lem:functorial} and Section \ref{Sec:maximal-contact}.

\section{The blown up space}\label{sec:BlowUp}

\subsection{Cobordant and stack-theoretic blow-ups}\label{ssec:CobordantBU}

We follow \cite[Section~2.4.6]{Wlodarczyk-cobordant}. Let $X$ be a smooth variety  and $E$ a SNC divisor over $X$. Consider an $E$-adapted center {given locally by:}
\[
\cA_J = \cO_X[x_1t^{1/a_1},\ldots, x_lt^{1/a_l}]^{\Int}.
\]
{Let $w$ be any integer multiple of  $\lcm(a_1,\ldots,a_l)$} and take the extended Rees algebra:
\[
\cA_J^\ext=\cO_{X}[t^{-1/w},{x}_1t^{1/a_1},\ldots,{x}_lt^{1/a_l}].
\] {The flexibility of choice of $w$ will be necessary when restricting to a maximal contact, see Theorem \ref{prop:DecreaseInvariant}.}

 Set $w_i = {w/a_i}$ and write $t_B = t^{1/w}$, in particular  \[
t^{-1/w} = t_B^{-1} \quad \text{ and } \quad t^{1/a_i} = t_B^{w_i}.
\]
We also introduce the notation $s_B = t_B^{-1}$. 

\begin{definition}[{Weighted blow-up cobordism}, {\cite[Section~2.4.6]{Wlodarczyk-cobordant}}]\label{CobordantBlow-up}\hfill
\begin{enumerate}
\item 
The {\it weighted blow-up cobordism}  of the center $\cA_J$  is given by the $X$-scheme:
\footnote{We recall  that in the analytic category the notation $\Spec_X$ stands for the analytic spectrum, sometimes written $\operatorname{Specan}_X$, see \cite[Definition 1.5.2]{AHV-analytic}, and that in this case $\GG_m$ stands for $\CC^*$. See Section \ref{sec:IntroAlgebraicAnalytic}.}
\[
\sigma:\  B:= \Spec_{X}(\cA_J^{\ext})
\quad \longrightarrow\quad X.
\]
We identify the new coordinates by $x_i' = {x}_i t_B^{w_i}={x}_i s_B^{-w_i}$, so this space is given by:
\begin{align*}
B &\ = \ \Spec_{X}\left(\cO_X[t_B^{-1},x_1',\ldots,x_l'] \ /\ (x_1- x_1' t_B^{-w_1}\ ,\ \ldots\ ,\  x_l- x_l' t_B^{-w_l})\right) \\
&\  =\  \Spec_{X}\left(\cO_X[s_B,x_1',\ldots,x_l'] \ /\ (x_1- x_1' s_B^{w_1}\ ,\ \ldots\ ,\ x_l- x_l' s_B^{w_l})\right).
\end{align*}

\item 
The {\it cobordant blow-up} of $\cA_{J}$ is the restriction $\sigma_+: B_+ \to X$ of $\sigma$ to
\(
B_+:=B\setminus \mbox{Vert}(B),
\)
where \(\mbox{Vert}(B):= V(x_1',\ldots,x_l')\) is the {\it vertex} of $B$.
\item 
The \emph{exceptional divisor} of $\sigma$ or $\sigma_+$  is given by $F_B = V(t^{-1}_B)  = V(s_B)$. If $(X,E)$ is a logarithmic variety, then $(B,E_B)$ is also a logarithmic variety, where $E_B = \sigma^{-1}(E) \cup F_B$, inducing a logarithmic structure on $B_+$.  In particular $\cD_B^{\log} = \cD_B(-\log E_B)$.

\end{enumerate}

\end{definition}

The name used in \cite{Wlodarczyk-cobordant} for $B$ is the \emph{full cobordant blow-up}. 

We note that $B$ comes along  with its natural projection $B  \to \AA^1:=\Spec K[s_B],$ identifying it with the \emph{degeneration to the weighted normal cone} of $J$. 
For every $s_B\neq 0$ the fiber of $B \to \AA^1$ is isomorphic to $X$, whereas the fiber over $s_B=0$ is an $\AA^l$-bundle, the \emph{weighted normal cone of $J$ in $X$.}

The grading of $\cA_J^{\ext}$, with $s_B$ having weight $-1$ and $x_i$ having weight $w_i$, induces an action of $\GG_m$ on $B$, equivariant for $B \to \AA^1$, stabilizing $V(x_1',\ldots,x_l')$ and $B_+$, and the action on $B_+$ has finite stabilizers:
\begin{equation*}
\forall \, \xi\in \GG_m, \quad \xi \cdot (t_B^{-1},x_1',\ldots,x_l') \ =\  (\xi^{-1} t_B^{-1}, \xi^{w_1} x_1',\ldots,\xi^{w_l} x_l').
\end{equation*}
Equivalently
\begin{equation*}
 \xi \cdot (s_B,x_1',\ldots,x_l')\  =\  (\xi^{-1} s_B, \xi^{w_1} x_1',\ldots,\xi^{w_l} x_l').
\end{equation*}

\begin{definition}[Stack-theoretic weighted blow-up]\label{Def:weighted-blow-up}\hfill
\begin{enumerate}
\item The stack-theoretic \emph{weighted blow-up} of $\cA_J$ is the stack $Bl_J(X) :=[B_+ / \GG_m]$ with its induced morphism, denoted $\tau:Bl_J(X) \to X$. 

Here $[\ \cdot\  / \ \cdot\ ]$ stands for the stack-theoretic quotient.

\item 
The \emph{exceptional divisor of $\tau$ is $F:= [F_B/\GG_m] \subset X':=Bl_J(X)$}. If $(X,E)$ is a logarithmic variety, then $(X',E')$ is a logarithmic stack, where $E' = \tau^{-1}(E) \cup F$.  In particular $\cD_{X'}^{\log} = \cD_{X'}(-\log E')$.
\end{enumerate}
\end{definition}

The  stack $Bl_J(X)$ is proper and tame over $X$. Since we are working in characteristic 0, the stack $Bl_J(X)$ is a Deligne--Mumford stack. We note that $\dim(B) = \dim(X)+ 1$ and, since $\GG_m$ is one-dimensional, $\dim(Bl_J(X)) = \dim (X)$. Moreover, since $\GG_m$ acts freely on $B\setminus F_B$ with quotient $[B\setminus F_B\ /\ \GG_m] \simeq X$, the blow-up $\tau$ is birational.

\begin{example}[Blow-up of smooth centers]
Consider the ideal sheaf $\cI = (x_1,\ldots,x_l)$ and the associated Rees algebra $\cA_{\cI} = \cO_{X}[x_1t,\ldots,x_l t]$. Let
\(
\sigma: B_{+} \to X
\)
be the cobordant blow-up by $\cA_{\cI}$. One shows, see \cite{Quek-Rydh}, \cite[Remark 2.4.9]{Wlodarczyk-cobordant},
\[
B_{+}/\GG_m = \Proj_X \left( \oplus_{n\geq 0} \cI^n \right)
\]
and the induced mapping of schemes $B_{+}/\GG_m \to X$ is the usual blow-up by $\cI$.  The scheme $B$ is simply the degeneration of $X$ to the normal cone of $I$.

Once again, in the analytic case the notation $\Proj$ stands for the analytic version, see \cite[Definition 1.7.1]{AHV-analytic}.

\end{example}

\subsection{Transforms of Rees algebras by cobordant and weighted blow-ups}\label{Sec:transforms-Rees}

Let $\cI$ be an ideal sheaf, and consider the cobordant {and weighted} blow-ups of an $\cI$-admissible center
\[
\cA_J^\ext=\cO_X[t^{-1/w},{x}_1t^{1/a_1},\ldots,{x}_lt^{1/a_l}].
\]
In particular, recall that $\cO_X[t \cI ] \subset \cA_J$. We now describe its transforms on the scheme $B= \Spec_X(\cA_J^\ext)$ and the stack $Bl_J(X)$, see \cite{Encinas-Villamayor-Rees}, \cite[Section 3.7.11]{Wlodarczyk-cobordant}. The transform on $B_+$ is obtained by restriction. It is $\GG_m$-equivariant, inducing the transform on the weighted blow-up $Bl_J(X)$.

When performing the cobordant blow-up, we will consider the variable $t$ from the Rees algebra gradation from $\cR$ and from the construction of $B$ as \emph{independent}, see e.g. \cite{Wlodarczyk-cobordant}. {To minimize confusion}, the pull-back of $t$ in the Rees algebra $\cR$ will be {indicated by $T$} for a new Rees algebra on $B$.

It is convenient to indicate $J = (x_1^{a_1},\ldots,x_l^{a_l})$ as a  $\QQ$-ideal, see Section \ref{sec:Rees-and-Q-ideals}.
We first note (see \cite[Definition 4.3.3(iv)]{Quek-Rydh}): 
\begin{lemma}\label{Lem:transformed-center} The center $J$ itself admits a natural transform $J' = ({x_1'}^{a_1},\ldots,{x_l'}^{a_l})$ on $B$, such that $J\cO_B = s_B^w J'$.
\end{lemma}
Note that the factor $s_B^w$ is exceptional. We call $J'$ the \emph{transform} of $J$ --- it serves below as both \emph{strict} and \emph{controlled} transform of $J$. It naturally restricts to $B_+$, where $J'$ is trivial, and, since $J'$ is $\GG_m$-invariant, to $Bl_J(X)$.

\begin{proof}  
Indeed, writing $x_i = x_i's_B^{w_i}$ and noting that $a_i w_i = w$, we have
$$J\cO_B = (x_1^{a_1},\ldots,x_l^{a_l}) = \left((s_B^{w_1}x_1)^{a_1},\ldots,(s_B^{w_l}x_l)^{a_l}\right) = s_B^w ({x_1'}^{a_1},\ldots,{x_l'}^{a_l}).$$
This is a local statement depending on the choices of $x_i$. However, the element $s_B$ is the canonical section defining the exceptional locus, so the center $J'$ is well-defined across charts. 
\end{proof}

The lemma implies in particular that, if $f t^a \in \cA_J$ then $fT^a \in \cA_J \cO_B$ can be factored: $f = s_B^{a\cdot w} f'$ for some $f' \in J'$. Denote $a_f := \max\{c: t_B^c f = s_B^{-c} f \in \cO_B\}$. Note in particular that $a_f \geq a \cdot w$.

\begin{definition}[Transforms of elements] \label{def:ControlledTransformElements}\hfill
\begin{enumerate} 
\item  The  \emph{strict transform} of $f\in \cO_X$ is $\sigma^s(f) :=  {t_B^{b_f} f = s_B^{-b_f} f \in \cO_B.}$
\item If $f t^a \in \cA_J$, the \emph{controlled  transform} of $f$ is   $\sigma^c(f) :=  t_B^{a} f = s_B^{-a} f \in \cO_B$. 
\end{enumerate}
\end{definition}

With this we can transform Rees algebras and ideals, see {\cite{Encinas-Villamayor-Rees,Wlodarczyk-cobordant}}:
\begin{definition}[Total and controlled transform of Rees algebras and ideals]\label{def:ControlledTransformReesAlgebra}
\hfill

\begin{enumerate} 

\item Let $\cR$ be a Rees algebra such that $A_J$ is admissible:  $\cR \subset \cA_J$. 
The \emph{total} transform of $\cR$ is given by $\cO_B \cdot \cR$.
\item  The \emph{controlled transform} of $\cR$ by $\sigma$ is given by controlled transforms of all terms, noting that $R_as_B^{-a\cdot w} \cO_B = \{\sigma^c(f) : f \in R_a\}\cO_B$:
\[
\sigma^{c}(\cR) = \sigma^{c}\left(\bigoplus R_a t^a\right)\quad  :=\quad  \bigoplus \left( \cO_{B} \cdot R_a \cdot t_B^{a \cdot w} \right)   T^a\ = \ \bigoplus \left( \cO_{B} \cdot R_a \cdot s_B^{-a \cdot w} \right)   T^a. 
\]
\item Given an ideal sheaf $\cI$,  its controlled transform is:
\[
\sigma^{c}(\cI) := \mathcal{O}_B \cdot \mathcal{I} \cdot t_B^{w} \subset \cO_B.
\]
\end{enumerate}
\end{definition}

\begin{definition}[Strict transform of Rees {algebras} and ideals]\label{def:StrictTransformReesAlgebra}\hfill
\begin{enumerate} 
\item Let $\cR$, $\sigma: B \to X$, $\cA_J$ be as above. The \emph{strict} transform of $\cR$ is given by the strict transforms of all terms:
\[
\sigma^{s}(\cR) = \cO_B\cdot\{\sigma^s(f)\, T^a\mid f\in R_a\}
\]
\item {Similarly, given an ideal sheaf $\cI$, the strict transform of $\cI$}
under a weighted blow-up cobordism $\sigma$ is given by:
\[
\sigma^{s}(\cI) := \cO_B\cdot \{\sigma^s(f)\mid f\in \cI\},
\]
\item
Finally the \emph{strict transform} $\sigma^{s}(Y)=Y^s$ of a closed subscheme $Y$ of $X$ is given by $\cI_{Y^s}=\sigma^s(\cI_Y)$ (where $\cI_Z$ stands for the defining ideal of a subscheme $Z$).
\end{enumerate}
\end{definition}

Again $\sigma^c(\cI)$ is compatible with the transform of the corresponding algebra. The key property of these transforms, which follows immediately from Lemma \ref{Lem:transformed-center},  is:
\begin{proposition} Under the same assumptions,
$\cA_{J'}$ is admissible for both $\sigma^c(\cR)$ and $\sigma^s(\cR)$.
\end{proposition}

\begin{example}
Let $\cI=(x^4+y^7+z^{20} + z^{21})$ and consider the $\cI$-admissible center:
\[
\cA_{J}=\cO_X[xt^{1/4},yt^{1/7},zt^{1/20}]^\inte.
\]
Since $w_A=\lcm(4,7,20)=140$ we obtain that
\[
\cA^\ext=\cO_X[t^{-1/140}, xt^{1/4},yt^{1/7},zt^{1/20}],
\] 	
so the weighted blow-up cobordism has as source:
\begin{align*}
B&=\Spec(\cO_X[t^{-1}_B, xt_B^{35},yt_B^{20},zt_B^{7}])\\ &= \Spec\left(\cO_X[s_B, x',y',z']/(x-s_B^{35} x', y-s_B^{20}y', z-s_B^{7}z')\right)
\end{align*}
We may now consider the total transform of $\cI$:
\[
\begin{aligned}
\cO_B \cdot \cI &= ((x't_B^{-35})^4+(y't_B^{-20})^7+(z't_B^{-7})^{20} + (z't_B^{-7})^{21}) \\
&= t_B^{-140}\left((x')^4+(y')^7+(z')^{20} + (z')^{21}t_B^{-7} \right), \\
&=  s_B^{140} \left((x')^4+(y')^7+(z')^{20} + (z')^{21}s_B^{7} \right),
\end{aligned}
\]
as well as the controlled transform, which in this case agrees with the strict transform:
\[
\begin{aligned}
\sigma^{c}(\cI) = \cO_B\cdot \cI \cdot t_B^{140} &= \left((x')^4+(y')^7+(z')^{20} + (z')^{21}t_B^{-7} \right) \\
&= \left((x')^4+(y')^7+(z')^{20} + (z')^{21}s_B^{7} \right) \subset \cO_B. 
\end{aligned}
\]
\end{example}

\begin{remark}[Schematic closure] It is also possible to define the strict transform as a schematic closure. In fact, note that the restricted mapping
\[
\sigma_{|B_-}: B_-=B\setminus V(t_B^{-1})=B\setminus V(s_B)=X\times\GG_m\to X
\]
is trivial over $X$. Now, the strict transform $\sigma^s(\cI)$ of an ideal $\cI$ on $X$ under a weighted blow-up cobordism $\sigma: B\to X$ of $\cA_J$ on $X$ can be {equivalently} defined as the schematic closure on $B$ of
\[
(\cO_B\cdot \cI)_{|B_-}=\cO_{B_-}\cdot \cI{=\{f\in \cO_B\mid f_{|B_-} \in \cI_{|B_-}\}}
\]
Similarly, the strict transform of a closed subscheme $Y$ of $X$ is the schematic closure $Y^s$ on $B$ of
\[
Y\times \GG_m\subset B_-=X\times \GG_m.
\]
Finally, for the cobordant blow-up $\sigma_+: B_+\to X$, the induced strict transform coincides with the standard definition $\sigma_+^s(\cI)$ as the schematic closure on $B_+$ of
\[
\cO_{\sigma_+^{-1}(X\setminus V({J}))}\cdot \cI_{|X\setminus V({J})}=(\cO_B\cdot \cI)_{|{B_+\cap B_-}}.
\]
\end{remark}

\subsection{Controlled and strict transforms of foliations by {the} blow-ups}\label{ssec:TransformFoliations}\label{Sec:transform}

\subsubsection*{The transforms} 
A cobordant blow-up $\sigma_+ : B_+ \to X$ with center $\cA_J$ is a generically smooth morphism, since it factors through the smooth morphism $X \times \AA^1\to X$. We therefore can define the total, controlled, and strict transform of a foliation $\cF \subset \cD_X^{\log}$ following Definitions \ref{def:TotalTransformFoliation} and \ref{def:TransformFoliation}, but alternatively also the split, split-controlled, and split-strict transforms through Definition \ref{Def:split-transforms}. For the split constructions the following diagram encapsulates the situation:

\begin{equation} \label{Eq:bigdiagram}\xymatrix{
B_+ \ar[r]^q\ar[d]_{\sigma_0}\ar@/^1.5cm/[rrd]^{\sigma_+} & Bl_J(X) \ar[rd]^{\tau}\ar[d]_{\tau_0} \\
X \times \AA^1 \ar[r]\ar[d] &X \times \cA^1 \ar[r]_{\varpi_X}\ar[d] & X \ar[d] \\
\AA^1 \ar[r] &\cA^1 \ar[r] & \{pt\}  
}\end{equation}

Here $\AA^1=\Spec K[s_B] = \Spec K[t_B^{-1}] $ and $\cA^1 = [\AA^1 / \GG_m]$. The morphism $Bl_J(X) \to \cA^1$ corresponds to the exceptional divisor, and $B_+ \to \AA^1$ corresponds to its defining section $s_B$. Note that $\sigma_+ = \tau\circ q$. We also write $\pi_X: X \times \AA^1 \to X$.

Each cartesian square in the bottom rectangle falls within the setup of Section \ref{Sec:relative-pullback} regarding relative foliations. We are given a foliation $\cF \subset \cD^{\log}_X$, which pulls back to relative foliations $\cF_{\cA_1} \subset \cD^{\log}_{X\times \cA^1/\cA^1}$ and $\cF_{\AA_1} \subset \cD^{\log}_{X\times \AA^1/\AA^1}$, and the pullback of  $\cF_{\cA_1}$ to $X\times \cA^1$ is $\cF_{\AA_1}$. 

Considering now the right half of the diagram, we have defined $\tau^*{\cF},$ $ \tau^c{\cF}$ and $\tau^s{\cF}$ in Definition \ref{def:TransformFoliation}, which clearly coincides with the construction of Definition \ref{Def:split-transforms}. 

We similarly may consider the outer diagram, where we defined $\sigma_+^*{\cF},$ $ \sigma_+^c{\cF}$ and $\sigma_+^s{\cF}$ in Definition \ref{Def:split-transforms}.

Finally we have the left hand rectangle, which falls under Section \ref{Sec:relative-pullback}, which allows us to identify 
 $$\sigma_+^*(\cF) = q^*(\tau^*(\cF)),  \quad \sigma_+^c(\cF) = q^*(\tau^c(\cF)) \quad \text{ and }\quad \sigma_+^s(\cF) = q^*(\tau^s(\cF)).$$ In particular, $\tau^*(\cF),  \tau^c(\cF)$ and $\tau^s(\cF)$ can be obtained by descent  either from $\sigma_+^*(\cF),$ $  \sigma_+^c(\cF)$ and $\sigma_+^s(\cF)$, which, as it turns out, provides a convenient explicit description, or from the more ``standard" $$\sigma_0^*\left(\folpull{{(\pi_X)}}(\cF)\right)\quad  \sigma_0^c\left(\folpull{{(\pi_X)}}(\cF)\right) \quad \text{ and }\quad\sigma_0^s\left(\folpull{{(\pi_X)}}(\cF)\right). $$

\subsubsection*{Explicit description} Our goal in this section is to provide a local description of these transforms with respect to blow-ups along $\cF$-aligned centers. Since the morphisms $Bl_J(X) \to X \times \cA_1 \to \cA^1$ are the quotients of $B_+ \to X \times \AA_1 \to \AA^1$, it suffices to describe the outer diagram. Since $B_+ \subset B$ is open, we might as well compute things using the following digram instead:

$$\xymatrix@R=5mm{
B \ar[dd]\ar[rrd]^{\sigma} & \\
&& X \ar[d] \\
\AA^1 \ar[rr] & & \{pt\}.  
}$$
An explicit description on \emph{\'etale} charts for $Bl_J(X)$ is provided in Section \ref{Sec:etale-charts}.

\begin{definition}[Split controlled and strict transform of a derivation]\label{def:FullControlledTransformDerivation}
Given a derivation $\partial \in \cD^{\log}_{X}$, we define integers $a_\partial$ and $b_\partial$ as follows:
\[
\begin{aligned}
\sigma^{c}(\partial) &= s_B^{a_{\partial}} \sigma^{{\ast}}(\partial) \in \cD^{\log}_{B/\AA^1} \quad \text{ where } \quad a_{\partial}:= \min\{ a \in \ZZ_{\geq 0}\mid  s_B^{a} \partial \in \cD^{\log}_{B/\AA^1}\}, \\
\sigma^{s}(\partial) &= s_B^{b_{\partial}} \sigma^{{\ast}}(\partial) \in \cD^{\log}_{B/\AA^1} \quad \text{ where } \quad b_{\partial}:= \min\{ b \in \ZZ \mid  s_B^{b} \partial \in \cD^{\log}_{B/\AA^1}\}.
\end{aligned}
\]
In the language of Section \ref{Sec:birational-transform}, $a_\partial = r$ and $b_\partial = m$, for the unique exceptional divisor $D$. 

We note that, in contrast to the general situation of Definition \ref{def:TransformFoliation}, the section $s_B$ is uniquely defined and therefore the transforms are well defined, not only up to a unit. 
\end{definition}

\begin{lemma}[Local expression of transforms of derivations]\label{lem:LocalExpressionTransformDerivation}
Let $\sigma: B\to X$ be a {weighted blow-up cobordism} with center $\cA_J = \cO_X[t^{1/a_i}x_i]_{l}$. Then:
\[
\begin{aligned}
\sigma^{{\ast}}(\partial_{x_i}) &=s_B^{-w_i}\partial_{x_i'}= t_B^{w_i}\partial_{x_i'} &\quad& i=1,\ldots,l,\\
\sigma^{{\ast}}(\partial_{x_i}) &= \partial_{x_i'} & \quad &i=l+1,\ldots,n,\\
\sigma^{c}(\partial_{x_i}) &= \sigma^{s}(\partial_{x_i}) = \partial_{x_i'} & \quad  & i=1,\ldots,n.
\end{aligned}
\]
\end{lemma}
\begin{proof}
{This follows} from the expression $x_i = s_B^{w_i}x_i'$ for $i=1,\ldots,l$, and the {requirement} that $\sigma^{{\ast}}(\partial)(s_B) \equiv 0$.
\end{proof}

The difference between the controlled and the strict transform of derivations is that we allow in the latter all integral powers $b\in\ZZ$.
In general, the controlled and strict transform do not coincide:

\begin{example}
We consider the transforms of the derivation $\partial = x^2\partial_x -y^2 \partial_y$ under the weighted blow-up cobordism of the origin $\cA_J=\cO_{\mathbb{C}^2}[tx,ty]$:
\[
\begin{aligned}
\sigma^{s}(\partial) &= (x')^2\partial_{x'} -(y')^2 \partial_{y'} \\
\sigma^{c}(\partial) &= t_B^{-1} \left((x')^2\partial_{x'} -(y')^2 \partial_{y'}\right)
\end{aligned}
\]
which are not equal. In fact, {the two} transforms are distinct for every derivation in $\mathbb{C}^2$ such that $\partial(m_0) \subset m_0^2$, {as well as in some situations, known as nilpotent singularities,}  when $\partial(m_0) \subset m_0$  --- that is, where the mapping $\overline{\partial}: m_0/m^{2}_0\to m_0/m^{2}_0$ is nilpotent. {For more details, we refer the reader to \cite{IY,BelJA,McQPan} and the references therein.}
 \end{example}

\subsection{Compatibility of transforms of foliations and of invariant  Rees algebras}

\begin{proposition} \label{prop:FInvariantBlowup2} Let $\cR=\bigoplus R_a$ be an $\cF$-invariant Rees algebra on a foliated logarithmically smooth variety $(X,\cF,E)$, and consider an $\cF$-invariant center $\cA_J$. Consider the deformation to the normal cone $\sigma:B \to X$ with center $\cA_J$. Then:
\begin{enumerate}
\item The total and controlled transforms of $\cF$ coincide, i.e. $\sigma^{c}(\cF) = \sigma^{\ast}(\cF)$;
\item The controlled transform $\sigma^c(\cR)$ is $\sigma^{c}(\cF)$-invariant;
\item  The strict transform $\sigma^s(\cR)$ is $\sigma^{s}(\cF)$-invariant.
\end{enumerate}

\end{proposition} 

We rely on the following result, concerning one local derivation  $\delta$:

\begin{proposition} \label{prop:FInvariantBlowup} 
Let $\delta\in \cD^{\log}_X$ be a derivation and let $\cR=\bigoplus R_a$ be a $\delta$-invariant Rees algebra on a foliated logarithmically smooth variety $(X,\cF,E)$, with $\delta$-invariant foliation $\cF$, and consider a $\delta$-invariant center $\cA_J$. Consider the deformation to the normal cone $\sigma:B \to X$ with center $\cA_J$. Then:
\begin{enumerate}
\item $\sigma^*(\delta)$ is a logarithmic derivation,  so $\sigma^{c}(\delta) = \sigma^{\ast}(\delta)$;
\item The controlled transforms $\sigma^c(\cR)$  and $\sigma^c(\cF)$ are $\sigma^{c}(\delta)$-invariant;
\item  The strict transforms $\sigma^s(\cR)$  and $\sigma^s(\cF)$ are $\sigma^{s}(\delta)$-invariant.
\end{enumerate}
\end{proposition}

\begin{proof} We first prove (1). Since $\cA_J$ is $\delta$-invariant, $\delta(f) \in \cA_{J,b}$ for all $f\in  \cA_{J,b}$. In particular $\sigma^*(\delta)$ defines a derivation of $\cO_B$. Since $\sigma^*(\delta)(s) = 0$ and $\delta\in \cD^{\log}_X$, we have that $\sigma^*(\delta) \in \cD^{\log}_B$. It follows that $\sigma^{c}(\delta) = \sigma^{\ast}(\delta)$.

For (2), we first prove that $\sigma^c(R_b)$ is $\sigma^*(\delta)$-invariant, noting that $\sigma^c(R_b)$ is generated by $\{s_B^{-wb}  \sigma^*(f): f\in R_b\}$. We have 
\begin{align*}\sigma^*(\delta)(\sigma^c(f)) &= \sigma^*(\delta) (s_B^{-wb}  \sigma^*(f))\\ &= s_B^{-wb} \sigma^*(\delta)( \sigma^*(f))\\& = s_B^{-wb} \sigma^*(\delta(f)) \quad = \quad \sigma^c(\delta(f)) \quad  \in\quad  \sigma^c(R_b).\end{align*}

Next, since $\sigma^c(\cF)$ is spanned by $\{\sigma^c(\partial): \partial \in \cF\}$ and $[\delta,\partial]\in \cF,$ we can write $$\cD_B^{\log} \ \ni \ [\sigma^*(\delta),\sigma^c(\partial)] =  [\sigma^*(\delta),s_B^{a_\partial}\sigma^*(\partial)] =  s_B^{a_\partial}[\sigma^*(\delta),\sigma^*(\partial)]   = s_B^{a_\partial}\sigma^*([\delta,\partial]),$$ in particular $a_\partial \geq a_{[\delta,\partial]}$ as needed.

For (3) note first that \begin{align*}\cO_B  \ni  \sigma^s(\delta)(\sigma^s(f)) = s_B^{b_\delta} \sigma^*(\delta) (s_B^{-b_f}  \sigma^*(f)) = s_B^{b_\delta-b_f} \sigma^*(\delta)( \sigma^*(f)) = s_B^{b_\delta-b_f}  \sigma^*(\delta(f)),
\end{align*}
in particular $-b_{\delta(f)} \leq b_\delta - b_f,  $ as needed. 

For the invariance of $\cF$, consider a derivation $\nabla \in \sigma^{s}(\cF)$, which can be written as $\nabla = \sum_{i=1}^{r} s_B^{c_i}\sigma^{\ast}(\partial_i)$, where $\partial_1,\ldots,\partial_r$ are local generators of $\cF$ and $c_i \in \mathbb{Z}$. Noting again  that $\sigma^*(\partial)(s) = 0$ we can write $$
\cD_B^{\log} \ \ni \ [\sigma^s(\delta),\nabla] =  \left[s_B^{b_\delta}\sigma^*(\delta),\sum_{i=1}^{r} s_B^{c_i}\sigma^{\ast}(\partial_i) \right]
=  \sum_{i=1}^r s_B^{b_\delta+c_i}[\sigma^*(\delta),\sigma^*(\partial_i)]   =  \sum_{i=1}^r s_B^{b_\delta+c_i}\sigma^*([\delta,\partial_i]),
$$
implying that $[\sigma^s(\delta),\nabla] \in \sigma^s(\cF)$ as needed.
\end{proof}

\begin{proof}[Proof of Proposition \ref{prop:FInvariantBlowup2}]
Point (1) and (2) are immediate from Proposition \ref{prop:FInvariantBlowup}. To prove point (3), let $\partial_1,\ldots,\partial_r$ be local generators of $\cF$ and $\nabla \in \sigma^{s}(\cF)$. By definition \ref{def:StrictTransformFoliation}, this implies that there exists $f_1,\ldots, f_r \in \cO_X(\ast D)$ such that $\nabla = \sum f_i \phi^{\ast}(\partial_i)$. Now, fix $a \in \mathbb{N}$ and $h \in R_a$ and consider $\phi^{s}(h) = y^r \phi^{\ast}(h)$. It follows that:
\[
 \cO_{X} \quad \ni \quad \nabla(\phi^{s}(h)) = \sum \left[f_i \phi^{\ast}(\partial_i)(y^r)\right] \phi^{\ast}(h) +  \sum f_i y^r \phi^{\ast}(\partial_i(h)\quad \in \quad  \sigma^{\ast}(R_a) \otimes \cO_X(\ast D)
\]
allowing us to conclude that $\nabla(h) \in \sigma^s(R_a)$ as needed.
\end{proof}

We point out that the statements of (2) and (3) in the propositions are sharp:

\begin{example}\label{ex:ControlledVsStrict}
Let $X=\AA^2_{\mathbb{C}}$, $\cF=\spa((x^2+y^2)\partial_x + xy\partial_y)$ and $\cI=(xy,x^3,y^3)$. The maximal $\inv$-admissible center is $\cA_J=\cO_{X}[xt,yt]$. Let \[B=\Spec(\mathbb{C}[t^{-1},tx,ty]\to X\] be the weighted blow-up cobordism of $\cA_J$. Then:
\[
\begin{aligned}
\sigma^c(\cI)&=(x'y', t^{-1}_B(x')^3,t^{-1}_B(y')^3) = (x'y', s_B(x')^3,s_B(y')^3), \\
\sigma^s(\cI)&=(x'y', (x')^3,(y')^3), \\
\sigma^c(\cF) &= \spa\left({t^{-1}_B}((x')^2+(y')^2)\partial_{x'} + x'y'\partial_{y})\right) = \spa\left({s_B}((x')^2+(y')^2)\partial_{x'} + x'y'\partial_{y})\right),\\
\sigma^s(\cF) &= \spa\left((x')^2+(y')^2)\partial_{x'} + x'y'\partial_{y
}\right).
\end{aligned}
\]
and it is easy to see that $\sigma^{c}(\cI)$ and $\sigma^s(\cI)$ are $\sigma^{c}(\cF)$ and $\sigma^s(\cF)$-invariant respectively. Nevertheless, note that:
\[
\left((x')^2+(y')^2)\partial_{x'} + x'y'\partial_{y}\right)(x'y') = (y')^3 + 2(x')^2y'
\]
and, since $(y')^3\notin \sigma^{c}(\cI)$, we conclude that $\sigma^{c}(\cI)$ is not $\sigma^{s}(\cF)$-invariant.
\end{example}

\subsection{Transforms by $\cF$-aligned blow-ups}

\begin{proposition}[Pull-backs under $\cF$-aligned
blow-up]\label{prop:TransformFoliationAligned}
Let $\sigma: (B,E) \to (X,F)$ be a weighted blow-up cobordism of an $\cF$-aligned
blow-up.
Then:
\begin{enumerate}
\item The controlled transform of $\cF$ can be described as: 
\[
\sigma^{c}(\cF)\ =\ \folpull{\sigma}(\cF) \cap \cD_{B/\AA^1}^{\log}.
\]
\item Taking controlled transform commutes with completion, that is, given
points $\pa \in X$ and $\pb \in \sigma^{-1}(\pa)$, we have that:
\[
\widehat{\cO}_{B,\pb} \cdot \sigma^{c}(\cF) =
\widehat{\sigma}^{c}_\pb\left( \widehat{\cO}_{X,\pa} \cdot \cF \right)
\]
where $\widehat{\sigma}^{c}_\pb$ stands for the controlled transform
with respect to the homomorphism of local rings
$\widehat{\sigma}_\pb^{\ast}:\widehat{\cO}_{X,\pa} \to
\widehat{\cO}_{B,\pb}$.
\item\label{It:TransformFoliationAligned} Let \(
\widehat{\cF}_{(x)} =  \spa_{\widehat{\cO}_{X,(x)}}(\partial_{x_1},\ldots,\partial_{x_l},\nabla_1,\ldots,\nabla_m)
\)
be a presentation as in Definition \ref{def:Faligned},  
then we have 
\begin{align*}
&&&& \widehat{\sigma}^{c}_\pb(\partial_{x_i}) &= \partial_{x'_i}, &  i=&1,\ldots, l,&&& \\
&&&&\widehat{\sigma}^{c}_\pb(\nabla_j) &= \widehat{\sigma}^{*}_\pb(\nabla_j), &   j=&1,\ldots, m&&&
 \end{align*}
  at any point $\pb$ as above, and $\widehat{\cO}_{B,\pb} \cdot \sigma^{c}(\cF)$ is generated by these elements $$\partial_{x'_i}, \ \  i=1,\ldots, l \quad \text{and}\quad  \widehat{\sigma}^{*}_\pb(\nabla_j),  \ \  j=1,\ldots, m,$$ where $\widehat{\sigma}^{*}_\pb(\nabla_j)$ are independent of $(x_i,\partial_{x_i'})$.
  
\end{enumerate}
\end{proposition}

\begin{proof}

To prove part (3) 
note that, {by Lemma \ref{lem:LocalExpressionTransformDerivation} and Proposition \ref{prop:FInvariantBlowup2}:}
\begin{equation}\label{Eq:formal-transform-derivatives}
\begin{aligned}
\widehat{\sigma}^{c}_\pb(\partial_{x_i}) &= \partial_{x_i'} = {t_B^{w_i} \widehat\sigma_{\pb}^*(\partial_{x_i})} &\quad & i=1,\ldots, l,\\
\widehat{\sigma}^{c}_\pb(\nabla_{j}) &= \widehat{\sigma}^{{\ast}}_\pb(\nabla_j)=:\nabla'_j, &\quad &j=1,\ldots,m,
\end{aligned}
\end{equation}
giving the required formulas.

{
 In particular $\sigma^* (\nabla_j) (x_i') = \sigma^* (\nabla_j) (x_i/s^{w_i}) = 0$. Moreover \[
[\partial_{x_i'},\nabla_j'] = [t_B^{-w_i}\widehat{\sigma}_{\pb}^{{\ast}}(\partial_{x_i}),\widehat{\sigma}_\pb^{{\ast}}(\nabla_j)] = t_B^{-w_i}\widehat{\sigma}_\pb^{{\ast}}\left([\partial_{x_i},\nabla_j]\right) \equiv 0.
\]
giving the required independence.}

To show that these generate  $\widehat{\cO}_{B,\pb} \cdot \sigma^{c}(\cF)$ consider a section $\delta :=\sum g_i \partial_{x_i} + \sum f_j \nabla_j$ of $\cF$ with controlled transform $\sigma^c(\delta) = s^{w} \sigma^*(\delta)\in \cD^{\log}_B$. Write $g_i = g_i's^{w_{g_i}} $ and $f_j = f_i's^{w_{f_j}},$ with $g_i'$ and $f_j'$ the controlled transforms. Applying this to $x_i'$ we obtain
$$\cO_X \ \ni\ \sigma^c(\delta)(x_i') = (\sum g_i' s^{w_{g_i}+ w-w_i} \partial_{x_i'} + \sum f_j' s^{w+ w_{f_j}} \sigma^*\nabla_j)(x_i') = g_i' s^{w_{g_i}+ w-w_i}.$$ Thus $w_{g_i}+ w-w_i\geq 0$, 
so that $\sigma^c(\delta)$ is in the span of the indicated elements, as needed.

We prove Part (2). For every $d>0$, we will consider regular derivations $\partial_i^{d}$ for {$i=1,\ldots,l$} and $\nabla_j^d$ for $j=1,\ldots,m$ such that:
\[
\partial_i^d - \partial_i \in (x)^d \,\widehat{\cF}_{(x)}, \quad \nabla_i^d - \nabla_i \in (x)^d \,\widehat{\cF}_{(x)}.
\]
{By Nakayama and  the faithful flatness of $\widehat{\cO}_{X,\pa}$, for any $d>0$} the finitely generated module $\cF_{\pa}$ is generated by {$\partial_1^d,\ldots,\partial_k^d,\nabla^d_1,\ldots,\nabla_r^d$}, {as does its completion}. Now, for $d$ sufficiently big, the total transform of an element in $(x)^d \widehat{\cF}_{(x)}$ {has no poles, hence it equals} its controlled transform. From the formal equations \eqref{Eq:formal-transform-derivatives} we deduce
\[
\sigma^{c}(\partial_i^d) = t^{-w_i} \sigma^{{\ast}}(\partial_i^d), \quad \sigma^{c}(\nabla_j^d) = \sigma^{{\ast}}(\nabla_j^d),
\]
and $\cO_{B,\pb}\cdot \sigma^{c}(\cF)$ is generated by the controlled transforms of $\partial_1^d,\ldots,\partial_k^d,\nabla^d_1,\ldots,\nabla_r^d$, and therefore also  its completion, as needed.

By Part (2), to prove Part (1) we may work formally. We proceed as in Part (3). Write  $\delta :=\sum g_i \sigma^*\partial_{x_i} + \sum f_j \sigma^*\nabla_j$ for a general element of $\folpull{\sigma}(\cF) \cap \cD^{\log}_{B/\AA^1}$. Factoring $g_i = g_i's^{w_{g_i}}$ with $s \nmid g_i'$, we have $\delta =  \sum g_i's^{w_{g_i}-w_i} \partial_{x_i'} + \sum f_j \sigma^*\nabla_j$. Applying $\delta$ to $x_i'$ we obtain $g_i's^{w_{g_i}-w_i}  \in \cO_B$, hence $w_{g_i}-w_i\geq 0$ as needed.\end{proof}

\begin{lemma}\label{Lem:poles-of-transform} Let $\cA_J$ be an $\cF$-aligned center.
We have the inclusions $$s^{w_1}\sigma^*(\cF)\ \subseteq\ \sigma^c(\cF)\ \subseteq\  \sigma^*(\cF).$$
\end{lemma}
\begin{proof} We need to prove the first inclusion, the other follows from the definitions. We now use Proposition \ref{prop:TransformFoliationAligned}\eqref{It:TransformFoliationAligned}. Since $w_1 \geq w_i$ for $i=1,\ldots l$, we have $s^{w_1} \sigma^*(\partial_{x_i}) \in \sigma^c(\cF)$, and since $\nabla_j'=\sigma^*(\nabla_j)$ we have $s^{w_1} \sigma^*(\nabla_j) \in \sigma^c(\cF)$, as needed. 
\end{proof}

\begin{lemma} \label{Lem:transform-FR}
Let $\cA_J$ be an $(\cR,\cF)$-admissible center of order $a_1$, with resulting deformation to the normal cone $\sigma:B \to X$. Then $$\sigma^c((t^{-1/a} \cF)(\cR)) \subseteq (T^{-1/a} \sigma^c(\cF))(\sigma^c(\cR)).$$
\end{lemma}
\begin{proof}
Let $f\in \cR_b$ and $\partial \in \cF$ be local sections. By Lemma \ref{Lem:poles-of-transform}, $\partial' =s^{w_1}\sigma^*(\partial) \in \sigma^c(\cF)$. Also $\partial'(s) = 0$, so $$\sigma^c((\partial f) t^{b-1/a}) = s_B^{-w(b-1/a)} \sigma^* (\partial) \sigma^*(f) T^{b-1/a} = s^{-wb}\partial' (s^{bw} f') T^{b-1/a} = \partial'(f') T^{b-1/a}, $$
where $f' \in \sigma^c(\cR)_b,$ as needed.
\end{proof}

 Lemma \ref{Lem:transform-FR} is used in a critical point when we prove that invariants drop, see Theorem \ref{prop:DecreaseInvariant}.

\subsection{The invariant $\inv_{\cF}$ drops under cobordant blow-ups}

The main result of this section is the following:

\begin{theorem}[{The invariant drops, see {\cite[Prop.3.7.29]{Wlodarczyk-cobordant}}}]\label{prop:DecreaseInvariant} Let $\cR=\bigoplus R_a$ be a Rees algebra on a foliated logarithmically smooth variety $(X,\cF,E)$, {which we assume nontrivial: $R_a\neq \cO_X$ for some $a>0$.} Let $\cA_J$ be the $\inv_{\cF}$-maximal admissible center for $\cR$, with $\sigma: B \to X$ the associated weighted blow-up cobordism, {and $\tau:X'\to X$ the associated weighted blow-up.}  

Then
\(
\max\inv_{\sigma^c(\cF)}(\sigma^c(\cR)) = \max \inv_{\cF}(\cR),
\)
and, for all  $\pb\in B_{+}$ with image $\pa \in X'$, 
\[
\inv_{\tau^c(\cF),\pa}(\tau^c(\cR)) = \inv_{\sigma^c(\cF),\pb}(\sigma^c(\cR)) \ < \ \max \inv_{\cF}(\cR).
\]
\end{theorem}
This implies, by Proposition \ref{prop:FCanonicalInvBasicProperties}, that also
\(
\inv_{\tau^s(\cF),\pa}(\tau^s(\cR))
 <  \max \inv_{\cF}(\cR).
\)
\begin{proof}
{The result for $\tau$ follows from the result for $\sigma$,} which we prove  by induction on the dimension of $X$; the case that $\dim (X) =0$ being trivial. So assume that the result is proved whenever $\dim X\leq\,n-1$
and consider a foliated logarithmic variety $(X,\cF,E)$ of dimension $n$ and a Rees algebra $\cR$. Since the result is trivially true outside $V(\cA_J)$, we fix a point $\pa \in X$ and we may assume that $\cR$ is a Rees algebra $\cR$ over $\cO_{X,\pa}$. The proof is now divided in two cases.

\medskip
\noindent
{\bf Case 1: $\ord_{\cF,\pa}(\cR)<\infty$.} Suppose that $\ord_{\cF,\pa}(\cR)=a<\infty$. Let $x_1$ be an $\cF$-maximal contact element with respect to $\cR$, see Definition \ref{def:FMaximalContactRees}. By Lemma \ref{lem:FcanonicalInvForFiniteRees} we know that there exists a nested-regular coordinate system $(x,y,z)$, starting with $x_1$, such that:
\[
\cA_J\cdot \cO_{X,(x)} =  \cO_{X,(x)}[x_it^{1/a_i},y_jt^{1/b_j},z_kt^{1/c_k}]_{l,r,s}
\]
where $a_1=a$, and
\(
\inv_{\cF,\pa}(\cR)=(\ord_{\cF,\pa}(\cR),\inv_{\cF_{H},\pa}(\cC(\cR,\cF)_{|H}))\) with $H=V(x_1)$.

\medskip
\noindent
\textbf{Claim:} Given a point $\pb\in \sigma^{-1}(\pa)$, let $\cR^c_\pb=\sigma^{c}(\cR) \cdot \cO_{B,\pb}$ {and $\cF^c:=\sigma^{c}(\cF)$.} Then
\[
\ord_{\cF^c,\pb}(\cR^c_\pb) \leq a
\]
and if the equality holds, then $x_1'$ is a maximal contact for $\cR^c_\pb$.

\begin{proof}By the definition of $\cF$-order, there exists $b\in \mathbb{N}$, $f\in R_b$ and derivations $\partial_{1}, \ldots,\partial_{ab} \in \cF_{\pa}$,  such that:
\[
\partial_1 \cdots \partial_{ab} f \text{ is a unit in }{\cO_{X,\pa}.}
\]
Using Lemma \ref{Lem:transform-FR} inductively we have
\[
{\cO_B^\times\ \ni\ } \sigma^{\ast}\left(  \partial_1 \cdots \partial_{ab} ft^b  \right)  \in (T^{-1/a}\cF^c)^{ab}(\cR^c_bT^b)
\]
and we conclude that $\ord_{\cF^c,\pb}(\cR^c_\pb) \leq a$. Next, by the definition of maximal contact, we may {as well} suppose that $\partial_1 = \partial_{x_1}$ and that
\[
\partial_2 \cdots \partial_{ab} f = x_1.
\]
{So, using Lemma  \ref{Lem:transform-FR} again we have}
\[
\begin{aligned}
x_1'T^{1/a} &  \in (T^{-1/a}\cF^c)^{ab-1}(\cR^c_bT^{b}),
\end{aligned}
\]
implying that $x_1'$ is an $\cF^c$-maximal contact for $\cR^c_\pb$, and proving the Claim.
\end{proof}

Denote by $H'$ the strict transform of $H$ by the weighted blow-up cobordism. As a consequence of the claim, if $q\notin H'$ then $\ord_{\cF^c,\pb}(\cR^c_{\pb}) < a$ and the result holds true. So we may suppose that $\pb\in H'$. Consider now the restriction $\sigma_H: H' \to H$, which is a weighted blow-up cobordism with center $\cA_{J|H}$.\footnote{Note that  $w$ is still an integer multiple of $\lcm(a_2,\ldots,a_k),$ as required.}  

By Lemmas \ref{lem:CoefficientAdmissible} and \ref{lem:FcanonicalInvForFiniteRees}, $\cA_{J|H}$ is the $\inv_{\cF_{|H}}$-maximal admissible center for $\cC(\cR,\cF)_{|H}$. Because of the induction assumption, it remains only to show that:
\[
\sigma_{H}^{c}(\cC(\cR,\cF)_{|H}) \quad \subseteq \quad  \cC(\cR^c_\pb,\cF^c_\pb)_{|H'}
\]
which follows from Lemma  \ref{Lem:transform-FR}  and Proposition \ref{prop:CoefficientSplitting}.

\begin{remark} {While the maximum of the  invariant drops on $B_+$, it remains the same on $B$ itself, as it contains $X \times \GG_m$.}
{The argument above shows} that the invariant satisfies the resolution criterion of \cite{Wlodarczyk-cobordant}, and remains constant over the vertex $V(B)=V(x_1',\ldots,x_l')$.
\end{remark}

\medskip
\noindent
{\bf Case 2: $\ord_{\cF,\pa}(\cR)=\infty$.} We follow a similar inductive proof in terms of the foliation we consider, just as in Section~\ref{ssec:Case2}: (i) $\cF = \cD_X$, (ii) $\cF=\cD_X^{\log}$ and (iii) $\cF \subset \cD_X^{\log}$. Note that (i) is already finished by case 1, and recall the definition of $\cF^{\infty}(\cR)$ given in equation {\eqref{eq:Rinfty}, page \pageref{eq:Rinfty}}. {As $\cA_J$ is $\cF$-invariant, by Proposition \ref{prop:FInvariantBlowup} we have $\cF^c = \sigma^*(\cF)$ and} we have that:
\[
\sigma^c(\cF^{\infty}(\cR)) = (\cF^c)^{\infty}(\sigma^c(\cR)).
\]
We conclude by induction of the Proposition in terms of cases $(i-iii)$, combined with Lemma \ref{lem:FcanonicalInvForInvariantRees}.
\end{proof}

\section{$\cK$-monomial, {smooth}, and logarithmically smooth foliations}\label{Sec:nice-foliations}

As it is well-known in birational geometry of foliations, it is usually impossible to find a birational model of a foliation which is non-singular, and the goal of a resolution of singularities is to find birational models satisfying some ``as simple as possible" normal forms (in fact, some references prefer the term ``reduction of singularities" instead of resolution of singularities, eg. \cite{CRS}). In this section, we discuss two notions of singularities, namely log-smooth and monomial singularities, which play a key role in our applications, as they form thick classes, see Definition \ref{def:thick}.
We equally relate the notion of log-smoothness with smoothness after cobordant blow-ups and comment on the definition of Darboux first integrals.

\subsection{Log-smooth foliations}\label{ssec:LogSmooth}
We start our discussion by log-smooth foliations, see Definition \ref{Def:foliations}\eqref{ItDef:LogSmooth}. 

\begin{example}\label{ex:LogSmooth}
Let $X= \mathbb{C}^2$, and consider the complex-analytic category:
\begin{enumerate}
\item Consider the saddle-node foliation $\cF$, that is, the foliation generated by:
\[
\partial = x \partial_x + y^2\partial_y
\]
and suppose that $E= \{x=0\}$. Then $\cF$ is a log-smooth foliation. In fact, one can check that $\cD_{X}^{\log}/\cF$ is generated by the image of $\partial_y$ by the quotient $\cD_{X}^{\log} \to \cD_{X}^{\log}/\cF$.
\item Consider once again the saddle-node foliation $\cF$, but suppose that $E = \emptyset$. In this case, $\cF$ is \emph{not} log-smooth. We can, nevertheless, transform it into a log-smooth foliation by blowing-up the smooth center $(x=0)$; note that such a center contains points where $\cF$ is smooth, so this choice is not functorial with respect to smooth morphisms. In fact, it is not possible to transform $\cF$ into a log-smooth foliation via a functorial choice of center, as this can not be accomplished by blow-ups with support in the singular locus of $\cF$.
\item Consider the Euler foliation $\cF$ generated by the vector-field \eqref{eq:Euler}, page \pageref{eq:Euler}, and suppose that $E = \emptyset$. There is no blow-up that transforms $\cF$ into a log-smooth foliation. 
\end{enumerate}
\end{example}

Our goal in this section is to prove that triples $(X,\cF,E)$ where $\cF$ is log-smooth forms a thick class, see Definition \ref{def:thick} and Theorem \ref{Th:aligned center log-smooth}.
 To this end, we make use of the following well-known result about sub-bundles, which we specialize to our case:

\begin{lemma}\label{lem:LocalPresentationLogSmooth}
Let $(X,\cF,E)$ be a smooth foliated logarithmic variety. Let $\pa \in X$, let $\{\partial_1,\ldots, \partial_n\}$ be a basis of the free $\cO_{X,\pa}$-module $\cD^{\log}_{X,\pa}$ (respectively, $\cD_{X,\pa}$), and $\{\nabla_1,\ldots,\nabla_t\}$ be a system of generators of $\cF_{\pa}$. Let $A_{ij} \in \cO_{X,\pa}$ be functions such that:
\[
\nabla_i = \sum_{j=1}^{n} A_{ij}\partial_j, \quad i=1,\ldots,t.
\]
The foliation $\cF$ is log-smooth (respectively smooth) at $\pa$ if, and only if, the matrix $A(\pa) = [A_{ij}](\pa)$ has rank equal to $\rank(\cF)$ and, up to re-ordering,  $\{\nabla_1,\ldots,\nabla_{r}\}$, with $r=\rank(\cF)$, generates $\cF_{\pa}$.
\end{lemma}

\begin{theorem}
\label{Th:aligned center log-smooth}
The class of triples $(X,\cF,E)$ where $\cF$ is log-smooth is thick. Moreover, let $(X,\cF,E)$ be a smooth foliated logarithmic variety.
\begin{enumerate}
\item Suppose that $\cF$ is log-smooth and that $\tau: (X',\cF',E') \to (X,\cF,E)$ is an $\cF$-aligned weighted (or cobordant) blow-up. The controlled and strict transforms of $\cF$ coincide and are log-smooth.
\item Suppose that $\cF$ is \emph{smooth} and that $\sigma: (X'',\cF'',E'') \to (X,\cF,E)$ is an $\cF$-aligned cobordant blow-up. The controlled and strict transforms of $\cF$ coincide and are \emph{smooth}.
\end{enumerate}
\end{theorem}
\begin{proof}
Note that condition \eqref{defItem:thickProduct} from Definition \ref{def:thick} is immediate. We concentrate in conditions \eqref{defItem:thickFaligned} and \eqref{defItem:thickRestriction}. Both statements are local, so we may work over a point $\pa \in V(\cA_J)$. Moreover, by faithful-flatness of the completion, it is enough to verify that a foliation is log-smooth (respectively, smooth) in the level of formal power series. We argue, therefore, formally. We start by presenting the center $\cA_J$ in a unified way before specializing to the proof of each one of the two properties.

Let $\rho$ be the rank of $\cF$. Consider a formally $\cF$-aligned presentation of $\cA_J$, that is, there are coordinate systems $(x,y,z) \in \widehat{\cO}_{X,\pa}$ adapted to $E$, such that:
\[
\begin{aligned}
\widehat{\cA}_{J} &= \widehat{\cO}_{X,\pa}[x_it^{1/a_i},y_jt^{1/b_j}, z_kt^{1/c_k}]_{l,r,s}\\
\widehat{\cF}_{\pa} & =\spa_{\widehat{\cO}_{X,\pa}}\{\partial_{x_1},\ldots,\partial_{x_l}, \nabla_1,\ldots,\nabla_m  \}
\end{aligned}
\]
satisfying the conditions from Definition \ref{def:Faligned}. Note that:
\[
\cD^{\log}_{X} \cdot \widehat{\cO}_{X,\pa} = \spa_{\widehat{\cO}_{X,\pa}}\{\partial_{x_i}, \partial_{y_j}, z_k\partial_{z_k}, \partial_{h}\}.
\]
where $i=1,\ldots,l$, $j=1,\ldots,r$, $k=1,\ldots,s$, $h=1,\ldots,n-(l+r+s)=:t$ and the derivatives $\partial_h$ may be assumed to be independent of $(x_i,\partial_{x_i})$, $(y_j,\partial_{y_j})$ and $(z_k,\partial_{z_k})$. Moreover, as $\nabla_j$ are independent of $(x_i,\partial_{x_i})$, we get that:
\[
\nabla_i = \sum_{j=1}^{r} A_{ij} \partial_{y_j} + \sum_{k=1}^{s} B_{ik} z_k\partial_{z_k} + \sum_{h=1}^{t} C_{ih}\partial_h, \quad i=1,\ldots,m.
\]
where, from the second condition of Definition \ref{def:Faligned}, we conclude that $A_{ij}(\pa) = 0$ for all $i$ and $j$. We conclude from Lemma \ref{lem:LocalPresentationLogSmooth} that the matrix $[B_{ik},C_{ik}](\pa)$ has rank equal to $\rho-l$.

\medskip
\noindent
\emph{On $\cF$-aligned blow-ups and Definition \ref{def:thick}~\eqref{defItem:thickFaligned}:} We address the cobordant situation,  the weighted blow-up situation follows as log smoothness can be tested after a smooth pullback. We note that (2) implies, for the associated weighted blow-up, that $\tau^c(\cF)$ is a subbundle of $\cD_{X'/\cA^1} = \cD^{\log}_{X'}$, so it is logarithmically smooth but not necessarily smooth.

Denote $\sigma: B \to X$ the cobordant blow-up. By Proposition \ref{prop:TransformFoliationAligned}, the controlled transform $\sigma^{c}(\cF)$ is generated by the derivations 
$\sigma^{c}(\partial_{x_i})$ and $\sigma^{c}(\nabla_j)$. Now, by Lemma \ref{lem:LocalExpressionTransformDerivation} and Proposition \ref{prop:FInvariantBlowup}, we know that:
\[
\begin{aligned}
\sigma^{c}(\partial_{x_i}) &=  \partial_{x_i'}\\
\sigma^{c}(\nabla_i)= \sigma^{\ast}(\nabla_i) &= \sum_{j=1}^{r} \sigma^{\ast}(A_{ij}) s_{B}^{-b_i}\partial_{y_j'} + \sum_{k=1}^{s} \sigma^{\ast}(B_{ik}) z_k'\partial_{z_k'} +  \sum_{h=1}^{t} \sigma^{\ast}(C_{ih})\partial_h',
\end{aligned}
\]
where $\partial_h' = \sigma^{c}(\partial_h)$. In particular, $[\sigma^{\ast}(B_{ik}),\sigma^{\ast}(C_{ih})](\pb) = [B_{ik},C_{ih}](\pa)$ has rank $\rho-l$ at every point $\pb$ in the fiber $\sigma^{-1}(\pa)$. Therefore, by Lemma \ref{lem:LocalPresentationLogSmooth}, $\sigma^{c}(\cF)$ is log-smooth at every point $\pb \in \sigma^{-1}(\pa)$. Since $\sigma^{c}(\cF)$ is saturated, moreover, it must be equal to the strict transform. 

Finally, if $\cF$ was smooth, the same proof shows that $\sigma^{c}(\cF)$ is smooth {as well} after the cobordant blow-up.

\medskip
\noindent
\emph{On Definition \ref{def:thick}~\eqref{defItem:thickRestriction}:} In this case, the center is smooth and not contained in $E$, which implies that
\(
\widehat{\cA}_{J} = \widehat{\cO}_{X,\pa}[x_it ,y_jt]_{l,r}.
\)
It follows that the non-saturated restriction of the foliation to $Y:=V(\cA_J)$, see Section \ref{Sec:nonsaturated-restriction}, is formally generated by:
\[
\widehat{\cF}_{\pa\, |Y} =\spa_{\widehat{\cO}_{Y,\pa}}\{\nabla_{1|Y},\ldots,\nabla_{m|Y}  \}
\quad \text{where} \quad
\nabla_{i|Y} = \sum_{h=1}^{t} C_{ih|Y}\partial_h, \quad i=1,\ldots,m.
\]
Recall that $[C_{ih}(\pa)]$ has rank $\rho-\ell$. By Lemma \ref{lem:LocalPresentationLogSmooth}, $\cF_{|Y}$ is log-smooth at $\pa$; in particular it is saturated. This implies that the saturated restriction (see Definition \ref{Def:saturated-pullback}) $\cF_{ \boldsymbol | \boldsymbol | Y}$ equals $\cF_{|Y}$ and is, therefore, log-smooth as needed.
\end{proof}

\subsection{$\cK$-monomial foliations}\label{ssec:ResTotallyIntegrable}  We now define the final normal forms for Darboux totally integrable foliations\, and we will show that they form a thick class, see Theorem \ref{Th:aligned center new}. We start by the definition in the analytic category:

\begin{definition}[Analytic $\mathcal{K}$-monomial foliations, {see \cite[Def 3.1]{BelRACSAM}, \cite[Def 3.1]{BelRI}}]\label{def:KMonomialFoliation}

Let {$\mathcal{K} \subset \CC$ be a sub-field and consider a complex} analytic foliated logarithmic variety $(X,\cF,E)$. We say that the foliation $\cF$ is $\mathcal{K}$-monomial at a point $\pa \in X$ if it is saturated at $\pa$, see Definition \ref{Def:foliations}, and there exists Euclidean coordinates
\(
(v,w) = (v_1,\ldots,v_p,w_1,\ldots,w_{n-p})
\)
centered at $\pa$ and adapted to $E$, such that:
\[
\begin{aligned}
\cF \cdot \cO_{X,\pa} &= \spa_{\cO_{X,\pa}}(\vec{\partial}_v,\vec{\nabla}) := \spa_{\cO_{X,\pa}}\left( \partial_{v_1},\ldots, \partial_{v_p}, \nabla_{1},\ldots, \nabla_{q-p}   \right)
\end{aligned}
\]
where the derivations $\nabla_i$ are given by:
\begin{equation}\label{eq:NormalFormKmonomial}
\begin{aligned}
\nabla_j& = \sum_{k=1}^{n-p} a_{jk}w_k \partial_{w_k},  \quad j=1,\ldots,q-p.
\end{aligned}
\end{equation}
{with} $a_{jk}\in \mathcal{K}$. We say that $(v,w)$ is a monomial coordinate and $\{\vec{\partial}_v,\vec{\nabla}\}$ is a monomial presentation of $\mathcal{F}$ at $\pa$. We say that $\mathcal{F}$ is $\mathcal{K}$-monomial if it is $\mathcal{K}$-monomial at every point.
\end{definition}

In the algebraic category, we will work with a slightly weaker definition:

\begin{definition}[Formally $\mathcal{K}$-monomial foliations]\label{def:KMonomialFoliationFormal}
Consider now {an algebraic foliated logarithmic variety $(X,\cF,E)$ over $K$.} We say that the foliation $\cF$ is \emph{formally $\mathcal{K}$-monomial} at a point $\pa \in X$ if, {possibly after a field extension of the residue field of $\pa$,} the conditions of Definition \ref{def:KMonomialFoliation} are satisfied with respect to \emph{formal} coordinates $(u,w) \in \widehat{\cO}_{X,\pa}$. 
\end{definition}
We note that a residue field extension (or passing to \'etale charts) is necessary to work on stacks. Working formally is a compromise necessary for Theorem \ref{Th:aligned center new} to hold.

\begin{notation}
In order to uniformize the treatment, given a point $\pa \in X$, in this section we work with the convention that $\cO = \cO_{X,\pa}$ in the analytic category, and $\cO = \widehat{\cO}_{X,\pa}$ in the algebraic.

\emph{Also to make terminology uniform, in the main results we drop the adjective ``formally" and refer to ``$\cK$-monomial foliations" even in the algebraic case.}
\end{notation}

{It is useful to}  define:

\begin{definition}[Homogeneous with respect to the monomial presentation]\label{def:HomogenousMonomialPresentation} Let $\cF$ be a $\cK$-monomial foliation at $\pa \in X$. A function $h\in \cO$ is said to be \emph{homogeneous with respect to the monomial presentation} $\{\vec{\partial}_v,\vec{\nabla}\}$ of $\cF\cdot \cO$ if:
\[
\begin{aligned}
\partial_{v_i}(h) &\equiv 0 \quad & &i=1,\ldots,p, & \text{ and } & &
\nabla_j(h) &= \xi_{j} h, \quad \xi_{j} \in \mathcal{K}\quad & j=1,\ldots ,q-p.
\end{aligned}
\]
\end{definition}

\begin{remark}[On $\cK$-monomial foliations]\label{rk:InvariantDerivations}\label{rk:BasicKmonomial}\hfill
\begin{enumerate}

\item\label{rk:BasicKmonomialForm} It follows from the definition that a foliation $\cF$ is $\mathcal{K}$-monomial at $\pa$ if, and only if, the derivations in $\cF \cdot \cO$ can be characterized by 
\[
\left(\prod_{j=1}^{n-p} w_j\right)\sum_{j=1}^{n-p}\lambda_{ij} \frac{d w_j}{w_j}(\partial) \equiv 0, \quad i=1,\ldots,r,
\]
where the vectors $\lambda_i = (\lambda_{i1}, \ldots, \lambda_{i,n-p})\in \mathcal{K}^{n-p}$ generate the orthogonal sub-space to the linear space generated by $\vec{a}_j=(a_{j1}, \ldots,a_{j,n-p})$ from Definition \ref{def:KMonomialFoliation}.
\item\label{rk:BasicKmonomialOpen} {Restricting} to the complex-analytic category, it follows from \cite[Lemma 3.4]{BelRI} that a foliation $\cF$ is $\mathcal{K}$ monomial if, and only if, {there exist} locally defined multi-valued monomial functions:
\[
w^{\beta_i} = w_1^{\beta_{i1}}\cdots w_{n-p}^{\beta_{i,n-p}}, \quad \beta_{ij}\in \mathcal{K}, \quad i=1,\ldots, n-q
\]
such that $[\beta_{ij}]$ is a matrix of maximal rank, and
\[
\cF_{\pa} = \spa_{\cO_{X,\pa}}\{\partial \in \cD_{X,\pa}^{\log} ;\, \partial(w^{\beta_i}) \equiv 0, \, i=1,\ldots,n-q\}.
\]
Thus the property is Euclidean-open \cite[Lemma 3.5]{BelRI}.
\item  {Monomial foliations locally look like log smooth foliations, but lack the necessary divisor to make them log smooth --- though such divisors exist locally.  Example \ref{ex:LogSmooth}(1) gives a log smooth foliation which is not monomial, and Example \ref{ex:LogSmooth}(3) shows that the existence of a local divisor is a subtle question.}
\end{enumerate}
\end{remark}

Our objective is to prove Theorem \ref{Th:aligned center new}. In particular, we must show that $\cF$-aligned blow-ups preserve $\cK$-monomiality. But the notion of monomial coordinates and presentation is not necessarily coherent with $\cF$-aligned presentations, see Definition \ref{def:Faligned}. Our first goal is to show that one can chose coordinates coherently for both Definitions. We start by recalling a technical result:

\begin{lemma}[{see \cite[Cor. 3.8]{BelRACSAM}, \cite[Lemma 5.14]{BelBmon}}]\label{lem:homogenousGenerators}
Suppose that $\mathcal{F}$ is a {$\mathcal{K}$-monomial} foliation at a point $\pa\in X$. Let $\cR$ be an $\cF$-invariant Rees algebra. Given monomial coordinates $(v,w)$ and monomial presentation $\{\vec{\partial}_{v},\vec{\nabla}\}$ for $\cF$ at $\pa$, there exists $h_1,\ldots, h_m \in \cO$ and $e_1,\ldots,e_m \in \mathbb{N}$ such that
\[
\cR \cdot \cO =\cO [h_1 t^{1/e_1}, \ldots, h_m t^{1/e_m}]^{\inte}
\]
where each $h_j$ is 
homogeneous with respect to the monomial presentation, see Definition \ref{def:HomogenousMonomialPresentation}. Finally, if $\cR= \cA_J$ is a weighted $E$-adapted center, then $(h_1,\ldots,h_m)$ can be chosen so that it can be completed into a system of $E$-adapted coordinates.
\end{lemma}
\begin{proof}
The first part of the Lemma follows from the Claim below applied to a system of generators of $\cR \cdot \cO$:

\begin{claim}
\label{cl:HomMonomialPresentation} 
Suppose that $\cR$ is $\cF$-invariant. Fix $a \in \mathbb{Q}_{>0}$ and
{$f \in R_a$. There exist} 
elements $h_1,\ldots, h_m \in R_a \cdot \cO$ which are homogeneous with respect to the monomial presentation, such that $(f) \cdot \cO \subset (h_1,\ldots, h_m)$.
\end{claim}
\begin{proof}
This result has been proved in the analytic case in \cite[Cor. 3.8]{BelRACSAM}; and in the formal case with $\mathcal{K}= \mathbb{Q}$ in \cite[Lemma 5.15]{BelBmon}. Both proofs can be adapted directly to cover the formal case with a general field $\mathcal{K}$.
\end{proof}

{Next}, suppose that $\cR = \cA_J$ is a weighted $E$-adapted center. In this case, $\cR$ admits a system of generators:
\[
\cR \cdot \cO = \cO[f_1t^{1/a_1},\ldots,f_lt^{1/a_l},g_1 t^{1/c_1},\ldots,g_st^{1/c_s}]
\]
where $(f_1,\ldots,f_l,g_1, \ldots, g_s)$ can be completed into $E$-adapted coordinates of $\cO$, the $f_i$ are free variables and the $g_k$ are $E$-divisorial. First, {up to} re-ordering the variables $(w_1,\ldots,w_{n-p})$ and units, we may suppose that $g_k = w_k $,  which are homogeneous. Next, we apply the claim 
to the free variables $f_i$. We conclude from the fact that $\cR =\cA_J$ is a smooth center.
\end{proof}

We are ready to show that both presentations are compatible; {compare with \cite[Lemma 4.5]{BelRACSAM}:}

\begin{lemma}[Presentation of $\mathcal{F}$-aligned center for $\mathcal{K}$-monomial foliations]\label{lem:PresentationMonomialAligned}
Suppose that $\mathcal{F}$ is a {$\mathcal{K}$-monomial} foliation at a point $\pa\in X$. Let $\cA_J$ {be an} 
$\cF$-aligned center. There exists $\cO$-coordinates $(x,y,z) = (v,w)$ centered at $\pa$ and a presentation:
\[
\cF \cdot\cO=\spa_{\cO}(\partial_1,\ldots,\partial_{p+q})
\]
such that:
\begin{itemize}
\item The coordinate system $(v,w)$ is monomial and $(\partial_1,\ldots,\partial_{p+q}) = (\vec{\partial}_v,\vec{\nabla})$ is a monomial presentation of $\cF \cdot \cO$ at $\pa$.
\item The generators $(\partial_1,\ldots,\partial_{p+q})$ form a $\cF$-aligned $\cO$ presentation, and the center admits a presentation 
\[
{\cA_{J} = \cO[x_it^{1/a_i},y_jt^{1/b_j}, z_kt^{1/c_k}]_{l,r,s}}
\]
satisfying all properties of Definition \ref{def:Faligned} (resp.  \ref{def:FormalFaligned}).
\end{itemize}
\end{lemma}
\begin{proof}
Fix a monomial coordinate system $(v,w)$ and a monomial basis $\{\vec{\partial_{v}}, \vec{\nabla}\}$. The Rees algebra associated to the center of blow-up is given by:
\[
\mathcal{A}_J \cdot \cO = \cO[f_1 t^{1/d_1},f_2 t^{1/d_2}, \ldots, f_m  t^{1/d_m}],
\]
where $d_i \in \mathbb{N}$ and the collection $(f_1,\ldots,f_{m})$ can be completed to form a system of
{regular parameters of $\cO$ adapted to $E$.}
  By hypothesis, $\mathcal{A}_J \cdot \cO$ is $\mathcal{F}$-aligned, so we may suppose that there exists $0 \leq \ell \leq m$ such that $(f_1,\ldots,f_{\ell})$ is transverse to $\mathcal{F}$, see {Definition \ref{def:TransverseSection}}, and $\ell$ is maximal in this respect.

We argue by induction on $\ell$. First suppose that $\ell=0$, that is, $\cA_J$ is $\cF$-invariant. By Lemma \ref{lem:homogenousGenerators}, we have
\[
\cA_{J}\cdot \cO=\cO [h_1 t^{1/e_1},h_2 t^{1/e_2}, \ldots, h_m  t^{1/e_m}]
\]
where each $h_i$ is 
{homogeneous} with respect to the monomial basis, and $(h_1,\ldots,h_m)$ may be completed into $E$-adapted coordinates. Now, since the center is a smooth weighted center, {up to} re-indexing the variables $(w_1,\ldots,w_{n-p})$, we may suppose that 
\[
\tilde{w}_1 = h_1, \quad \tilde{w}_i = w_i,\, i=2,\ldots,n-p, \quad \tilde{v}=v,
\] 
is a well-defined {$\cO$-change of coordinates.} 
This implies that  
\[
\cA_{J}\cdot \cO=\cO [\tilde{w}_1 t^{1/e_1},h_2 t^{1/e_2}, \ldots, h_m  t^{1/e_m}]
\] 
and, {at the} level of vector-fields:
\[
\begin{aligned}
\partial_{v_i} &= \partial_{\tilde{v}_i} \quad & i=1,\ldots,p,\\
\nabla_i & =\sum_{j=2}^{n-p} a_{ij} \tilde{w}_j \partial_{\tilde{w}_j}   + \xi_{i1} \tilde{w}_1 \partial_{\tilde{w}_1}, \quad & i=1,\ldots,q-p,
\end{aligned}
\]
where we recall that $\nabla_i(h_1) = \xi_{i1}h_1$ with $\xi_{i1} \in \mathcal{K}$. Therefore $\{\vec\partial_{v}, \vec{\nabla}\} = \{\vec\partial_{{\tilde v}}, \vec{\nabla}\}$ 
is a monomial basis. We conclude the base case $\ell=0$ by recursively applying the above process to $h_2,\ldots,h_m$.

Next, let $\ell>0$ be fixed and suppose the result proved whenever $\ell'<\ell$. We may suppose, without loss of generality, that $f_1$ is a 
{vertical coordinate,} that is, $\cF(f_1) = \cO$. Without loss of generality, we may suppose that $\partial_{v_1}(f_1)$ is a unit. We now perform the $\cO$-change of coordinates 
\[
\tilde{v}_1 = f_1,\quad  \tilde{v}_i = v_i, \, i=2,\ldots,p, \quad \tilde{w} = w.
\]
At the level of vector-fields, we obtain:
\[
\begin{aligned}
\partial_{v_1} &= \partial_{\tilde{v}_1}, \quad & \\
\partial_{v_i} & = \partial_{\tilde{v}_i} + \partial_{v_i}(g) \partial_{\tilde{v}_1} , \quad & i=2,\ldots,q\\
\nabla_i & =  \sum_{j=1}^{n-p} a_{ij} \tilde{w}_j \partial_{\tilde{w}_j}   + \nabla_i(g)  \partial_{\tilde{v}_1} , \quad & i=1,\ldots,k-q,
\end{aligned}
\]
and we conclude that:
\[
\begin{aligned}
\spa_{\cO}\{\vec{\partial}_{v},\vec{\nabla}\} &= \spa_{\cO}\{\partial_{\tilde{v}_1}, \partial_{\tilde{v}_2}, \ldots, \partial_{\tilde{v}_p}, \nabla_1 - \nabla_1(g)\partial_{\tilde{v}_1}, \ldots, \nabla_{q-p} - \nabla_{q-p}(g)\partial_{\tilde{v}_1}\} \\
&= \spa_{\cO}\{\partial_{\tilde{v}_1}, \partial_{\tilde{v}_2}, \ldots, \partial_{\tilde{v}_p}, \ \sum_{j=1}^{n-p} a_{1j} \tilde{w}_j \partial_{\tilde{w}_j} , \ \ldots, \ \sum_{j=1}^{n-p} a_{(q-p)j} \tilde{w}_j \partial_{\tilde{w}_j} \} 
,
\end{aligned}
\]
where the generators in the right hand side form of a monomial presentation. Denote by $H = V(\tilde{v}_1)$ and {consider $\cF_{|H}$, see Section \ref{Sec:nonsaturated-restriction}}, and $\mathcal{A}_{J|H}=\cO_{|H} [f_2 t^{1/e_2}, \ldots, f_m  t^{1/e_m}]$; note that $\cF_{|H}$ is $\cK$-monomial and that $\mathcal{A}_{J|H}$ is formally $\mathcal{F}_{|H}$-aligned with $\ell'=\ell-1$. We conclude by induction and by using the lifting associated to $(\tilde{v}_1,\partial_{\tilde{v}_1})$, see Definition \ref{def:LiftingAssociatedToD}.
\end{proof}

\begin{corollary}\label{lem:PresentationMonomialAligned1} The restriction $\cF_{|V({J})}$ of $\cK$-monomial foliation to a smooth $\cF$-aligned center $\cA_J$ is $\cK$-monomial.
\end{corollary}
\begin{proof}
By Remark \ref{rk:BasicKmonomial}\eqref{rk:BasicKmonomialForm}, a $\cK$-monomial foliation $\cF$ is described by the dual sheaf of differential forms generated by:
\[
\left(\prod_{i =1}^{n-p} w_i\right) \sum_{j=1}^{n-p}\lambda_{ij} \frac{d w_j}{w_j}
\]
where $(v,w)$ are monomial coordinates. By Lemma \ref{lem:PresentationMonomialAligned}, the restriction to the $\cF$-aligned center  {preserves} this form.
\end{proof}

\begin{theorem}[Compare {\cite[Prop. 4.4]{BelRACSAM}}]\label{Th:aligned center new}
\noindent
The class of triples $(X,\cF,E)$ where $\cF$ is $\cK$-monomial is thick. Moreover, let $(X,\cF,E)$ be a smooth foliated logarithmic orbifold such that $\cF$ is $\cK$-monomial. Consider an $\cF$-aligned blowing up $\pi: (X',\cF',E') \to (X,\cF,E)$. Then the controlled and strict transforms of $\cF$ coincide and are $\cK$-monomial.
\end{theorem}
\begin{proof}
Note that condition \eqref{defItem:thickProduct} from Definition \ref{def:thick} is immediate, while condition \eqref{defItem:thickRestriction} follows from Corollary \ref{lem:PresentationMonomialAligned1}. We now check the statement about $\cF$-aligned blow-ups and, in particular, we check condition \eqref{defItem:thickFaligned}. We check that the controlled transform is $\cK$-monomial, hence saturated, therefore it equals the strict transform.

Let $\pa \in V(\cA_J)$ and consider the coordinate systems $(x,y,z)=(v,w)$ and presentation $\cF=\spa \{\partial_1,\ldots,\partial_{p+q}\}$ given by Lemma \ref{lem:PresentationMonomialAligned}. Since the result is essentially local, we may {work over $\pa$.} 
Consider the weighted blow-up with center $\cA_J$ and let {$\tau_i:W_i \to X$}
be the induced generically finite morphisms of varieties, see {Section~\ref{Sec:stack-descent}}. To simplify the notation, we will assume that $V(\cA_J)= \{\pa\}$; the general case follows from the same argument, but the notation is heavier. We consider each chart separately --- note that it is enough to consider only the $v_1$ and $w_1$-charts, all other charts having a symmetric treatment. 

\medskip
\noindent
\emph{Chart $v_1$:} Following the construction from {Section~\ref{Sec:stack-descent},} we get:
\[
\begin{aligned}
v_1 = u^{d_1}, \quad v_i = u^{d_i} \tilde{v}_i,\, i=2,\ldots,p, \quad w_j = u^{e_j}\tilde{w}_j, \, j=1,\ldots n-p,
\end{aligned}
\]
and, by applying expressions {\eqref{eq:TransformDerivationsWeighted}, page \pageref{eq:TransformDerivationsWeighted}}, we get {(denoting the morphism $\tau_{v_1}$)}:
\[
\begin{aligned}
\tau_{v_1}^{c}(\partial_{v_1})&= \frac{1}{d_1}\left(u\partial_{u} - \sum_{i=2}^p d_i \tilde{v}_i\partial_{\tilde{v}_i} - \sum_{i=1}^{n-p} e_i \tilde{w}_i \partial_{z_i}\right)  &  &\\
\tau_{v_1}^{c}(\partial_{v_i}) &=  \partial_{\tilde{v}_i}, &\quad &i=2,\ldots,p\\
\tau_{v_1}^{c}(\nabla_{j}) &= \sum_{k=1}^{n-p} a_{jk}\tilde{w}_k \partial_{\tilde{w}_k},  & \quad & j=1,\ldots,q-p,
\end{aligned}
\]
which generates $\tau_i^{c}(\cF)$ by Proposition \ref{prop:RegLiftingFAligned}. To simplify the notation, denote by $\tilde{\nabla}_j:=\tau_{v_1}^{c}(\nabla_{j})$. Now, $\tau_i^{c}(\cF) $ is generated by the vector-fields:
\(
\{\vec{\partial}_{\tilde{v}}, u\partial_u, \tilde{\nabla}_{j}) \}
\)
with $j=1,\ldots,q-p$ at every point $\pb$ which is in the pre-image of of $\pa$.\footnote{In the analytic category we may claim more, that is, the derivations are defined in a neighborhood of the exceptional divisor. But, in the algebraic category we must work formally, so the above derivations are only defined in the pre-image of $\pa$.} It follows that $\tau_i^{c}(\cF)$ is $\cK$-monomial at the origin of  the $v_1$-chart, and we claim that this is true at every point $\pb \in \tau_{v_1}^{-1}(\pa)$.

In fact, note that every $\pb \in \tau_{v_1}^{-1}(\pa)$ is written as $(u,\tilde{v},\tilde{w})= (0,\pb_{\tilde{v}},\pb_{\tilde{w}}) \in K^n$. We will prove the claim for a single translation, that is, when only one of the entries of $(0,\pb_{\tilde{v}},\pb_{\tilde{w}})$ is non-zero; the general case will follow by iterating this argument, one translation at a time. Note that translations in the $\tilde{v}$ variable trivially preserve all properties of a monomial presentation and coordinates, so we only need to deal with translations in $\tilde{w}$.

Without loss of generality, we assume that $\pb$ is obtained by a translation by $b\in K$ in the $\tilde{w}_{n-p}$ variable. Let $\tilde{x}= \tilde{w}_{n-p} - b$ and let $\tilde{w} = (\tilde{w}_1,\ldots, \tilde{w}_{n-p-1})$ denote all $\tilde{w}$ coordinates with the exception of $\tilde{w}_{n-p}$, so that $(u,\tilde{x},\tilde{v},\tilde{w})$ is centered at $\pb$. 
Consider now the expressions of the vector-fields:
\[
\begin{aligned}
\tilde{\nabla}_{j}& = \sum_{k=1}^{n-p-1} a_{jk}\tilde{w}_k \partial_{\tilde{w}_k} + a_{j,n-p}(\tilde{x}+b)\partial_{\tilde{v}}, \quad j=1,\ldots ,q-p
\end{aligned}
\]
If $a_{j,n-p} =0$ for every $j$, then $\{\vec{\partial}_v,\vec{\nabla}\}$ is a monomial basis and we are done. Otherwise, we may suppose without loss of generality that $a_{1,n-p} \neq 0$. Consider the change of coordinate:
\footnote{Note that this is a $\widehat{\cO}_{X,\pb}$-change of coordinates in general. But, it is actually stronger in two situations: in the analytic category, the change is Euclidean; in the algebraic category, the change is \'{e}tale provided that $\mathcal{K} = \mathbb{Q}$.} 
\[
w_j = \tilde{w}_j \left(\tilde{x} + b\right)^{ - a_{1,j}/a_{1,n-p}}, \quad j=1,\ldots, n-p-1,
\]
which yields:
\[
\begin{aligned}
\tilde{\nabla}_{1} &= a_{j,n-p}(\tilde{x}+b)\partial_{\tilde{x}}\\
\tilde{\nabla}_{j} &= \sum_{k=1}^{n-p-1} \left(a_{jk} - a_{j,n-p} \frac{a_{1k}}{a_{1,n-p}} \right)w_k \partial_{w_k} + a_{j,n-p}(\tilde{x}+b)\partial_{\tilde{x}}, \quad j=2,\ldots ,q-p.
\end{aligned}
\]
{Replacing $\tilde{\nabla}_j$ by a} {linear combination of $\tilde{\nabla}_j$ and $\tilde{\nabla}_1$ we}
obtain a {$\cK$-monomial} presentation, proving the claim.

\medskip
\noindent
\emph{Chart $w_1$:} Following the construction from {Section~\ref{Sec:stack-descent}}, we get:
\[
\begin{aligned}
v_i = u^{d_i} \tilde{v}_i,\, i=1,\ldots,p, \quad w_1 = u^{e_1}, \quad
w_j = u^{e_j}\tilde{w}_j, \, j=2,\ldots n-p,
\end{aligned}
\]
and, {again} applying expressions \eqref{eq:TransformDerivationsWeighted}, we get:
\[
\begin{aligned}
\tau_{w_1}^{c}(\partial_{v_i}) &= 
\partial_{\tilde{v}_i}, \quad i=1,\ldots,p\\
\tau_{w_1}^{c}(\nabla_{j}) &= \frac{1}{e_1}\left(a_{j1} u\partial_u  + \sum_{k=2}^{n-p} \left(e_1 a_{jk}-a_{j1} e_k \right)\tilde{w}_k \partial_{\tilde{w}_k}  - \sum_{k=1}^{p} a_{j1}d_k\tilde{v}_k\partial_{\tilde{v}_k} \right),
\end{aligned}
\]
for $j=1,\ldots,q-p$, which generates $\tau_{w_1}^{c}(\cF)$ by Proposition \ref{prop:RegLiftingFAligned}. We easily check that $\tau_{w_1}^{c}(\cF)$ is $\cK$-monomial at the origin, and we may argue as in the $v_1$-chart to conclude that it is $\cK$-monomial everywhere in the pre-image of $\pa$.
\end{proof}

\subsection{{Smooth} cobordant resolution and log-smooth weighted resolution of $\mathcal{K}$-monomial foliations}\label{ssec:LogSmoothnessMonomial}
A priori, the notion of log-smoothness and $\cK$-monomiality are independent. On the one hand, the saddle-node foliation in Example \ref{ex:LogSmooth}(1) is log-smooth, but not monomial. At the other hand, monomial foliations include singular foliations even when the logarithmic structure is empty. Nevertheless, we can prove that every $\cK$-monomial foliation can be transformed into a log-smooth $\cK$-monomial foliation:

\begin{theorem} \label{thm:LogSmoothnessMonomial}
Let $(X,\cF,E)$ be a smooth logarithmic variety and suppose that $\cF$ is $\cK$-monomial. There exists a sequence of weighted (or cobordant) blow-ups
\[
(X,\cF,E)\stackrel{\sigma_0}\leftarrow (X_1,\cF_1,E_1)\stackrel{\sigma_{1}}\leftarrow\cdots\stackrel{\sigma_{k-1}}\leftarrow (X_k,\cF_k,E_k)=(X',\cF',E'),
\]
where $\cF_{i+1}$ is the strict transform of $\cF_i$, such that:
\begin{itemize}
\item[(1'')] The foliation $\cF'$ is $\cK$-monomial and log-smooth. In the case of cobordant blow-ups, moreover, the induced foliation $\cF''$ is smooth.
\item[(2'')] The center ${J}_i$ of the blow-up $\sigma_i$ is $\cF_i$-invariant.
\item[(4'')] Property \eqref{It:functorial} from Theorem \ref{thm:PrincipalizationFoliated} hold true. 
\end{itemize}
 \end{theorem}
\begin{proof}
We prove the cobordant statement. If $E=\emptyset$, then the weighted statement follows from the cobordant case, Theorem \ref{Th:aligned center new}, and the fact that log-smoothness after a cobordant sequence implies log-smoothness over the orbifold. When $E \neq \emptyset$, some extra care is needed (for instance, if $\cF$ is already log-smooth, but not smooth, a direct use of the cobordant statement would lead to non-canonical blow-ups): this can be easily done by modifying the considered invariant, as indicated below.

To prove the cobordant statement, consider an auxiliary invariant which we call the \emph{\Sm -rank} of a foliation $\mathcal{F}$ at a point $\pa$, denoted by $\nsrank_{\pa}(\mathcal{F})$, which is defined as the dimension of the image of the mapping  $\mathcal{F} \otimes \cO_{X,\pa}/\mathfrak{m}_{X,\pa}  \to  \cD_{X} \otimes \cO_{X,\pa}/\mathfrak{m}_{X,\pa}$.\footnote{To obtain the weighted statement, consider instead $\nsrank^{\log}_{\pa}(\mathcal{F})$, that is, the dimension of the image of the mapping  $\mathcal{F} \otimes \cO_{X,\pa}/\mathfrak{m}_{X,\pa}  \to  \cD_{X}^{\log} \otimes \cO_{X,\pa}/\mathfrak{m}_{X,\pa}$.}

We start by noting three key properties:

\begin{itemize}
\item By Nakayama's Lemma, the \Sm -rank is lower semicontinuous;
\item The invariant $\nsrank(\mathcal{F})$ may take only a finite number of values, as it is an integer between $0$ and $\rank(\cF)$; 
\item If $\nsrank_{\pa}(\cF) = \rank_{\pa}(\cF)$ and $\cF$ is saturated, then $\cF$ is smooth.
\end{itemize}

We note that $\nsrank_{\pa}(\cF)$ carries the same information as the \emph{dimensional type} $\tau(\cF,\pa)$ of \cite[Page 919]{Cano}, specifically $\nsrank_{\pa}(\cF)+\tau(\cF,\pa) = \dim(X)$.

As $\cK$-monomial foliations are saturated foliations, in order to prove the Theorem it is enough to define a cobordant blow-up that increases the value of $\nsrank(\mathcal{F})$.  Therefore, let $\rho = \min\{ \nsrank_{\pa}(\cF);\, \pa\in X \}$, denote by $J$ the reduced ideal whose support is the locus of points where the \Sm -rank is equal to $\rho$. Fix a point $\pa \in X$ such that $\nsrank_{\pa}(\mathcal{F}) = \rho$. We check that $J$ is a smooth center. To this end, consider a monomial presentation $(u,w)$ and $\{\vec{\partial}_v,\vec{\nabla}\}$ of $\cF$, see Definition \ref{def:KMonomialFoliation}. As the \Sm -rank can be checked formally, we conclude that $\nsrank_{\pa}(\cF) = p$ (that is, the number of $v$ coordinates). Moreover, we claim that:
\[
J \cdot \widehat{\cO}_{X,\pa}= \left(w_k ;\, \exists j \in \{1,\ldots,q-p\} \text{ such that } a_{jk}\neq 0 \right),
\]
where the constants $a_{jk}$ are given in equation \eqref{eq:NormalFormKmonomial} in page \pageref{eq:NormalFormKmonomial}. In fact, the claim follows from the local description of the derivatives $\nabla_j = \sum_{k=1}^{n-p} a_{jk}w_k \partial_{w_k}$ combined with the fact that $J$ can also be defined as the support of the mapping:
\[
\bigwedge^{\rho+1}\left( \mathcal{F} \otimes \cO_{X,\pa}/\mathfrak{m}_{X,\pa}\right) \to \bigwedge^{\rho+1}\left( \cD_{X} \otimes \cO_{X,\pa}/\mathfrak{m}_{X,\pa}\right).
\]
It follows that $J$ defines a smooth center. Note, moreover, that $J$ is $\cF$-invariant, as it is formally $\cF$-invariant. To prove the Theorem, it now remains to check that the \Sm -rank increases after blowing-up $J$. Indeed, it is sufficient to check this locally (and formally) at the point $\pa$. Consider the associated weighted blow-up cobordism $\sigma : B \to X$. Without loss of generality, we may suppose that $J = (w_1,\ldots,w_l)$, where $l\leq n-p$, so that $B =\Spec_X\mathcal{O}_X[s,w_1', \ldots, w_l']/(w_1-tw_1',\ldots,w_r-tw_l')$. As this is a $\cF$-invariant blow-up and $\cF$ is $\cK$-monomial, the controlled and strict transforms coincide by Theorem \ref{Th:aligned center new}. Moreover, by Proposition \ref{prop:TransformFoliationAligned}, $\sigma^{c}(\cF)$ is generated by $\sigma^{c}(\partial_{v_i}) = \partial_{v_i}$ and:
\[
 \sigma^c(\nabla_j) = \sum_{k=1}^{n-p} a_{jk} w'_k \partial_{w'_k}.
    \]
    Thus on every point of $B_+ = B \setminus V(w'_1, \ldots, w'_l)$, at least one of the derivations $\sigma^c(\nabla_j)$ (maybe different for each point) is smooth. We conclude that the \Sm -rank is at least $p+1 = \rho+1$ everywhere over $B_{+}$, as needed.
\end{proof}

\begin{corollary}\label{cor:LogSmoothnessSmooth}
Let $(X,\cF,E)$ be a smooth logarithmic variety and suppose that $\cF$ is log-smooth. There exists a sequence of cobordant blow-ups
\[
(X,\cF,E)\stackrel{\sigma_0}\leftarrow (X_1,\cF_1,E_1)\stackrel{\sigma_{1}}\leftarrow\cdots\stackrel{\sigma_{k-1}}\leftarrow (X_k,\cF_k,E_k)=(X',\cF',E'),
\]
where $\cF_{i+1}$ is the strict transform of $\cF_i$, such that:
\begin{itemize}
\item[(1'')] The foliation $\cF'$ is smooth.
\item[(2'')] The center ${J}_i$ of the blow-up $\sigma_i$ is $\cF_i$-invariant.
\item[(4'')] Property \eqref{It:functorial} from Theorem \ref{thm:PrincipalizationFoliated} hold true. 
\end{itemize}
\end{corollary}
\begin{proof}
Since $\cF$ is a sub-bundle of $\cD_X^{\log}$, it is enough to prove the result for $\cF=\cD_X^{\log}$. As $\cD_X^{\log}$ is a $\mathbb{Q}$-monomial foliation, the statement follows from Theorem \ref{thm:LogSmoothnessMonomial}.
\end{proof}

\begin{remark}
Corollary \ref{cor:LogSmoothnessSmooth} can be proved in a more direct way via combinatorial methods. In fact, the sequence of cobordant blow-ups corresponding to the barycentric sub-division of the fan provides such a resolution.
\end{remark}

\subsection{Presentation of Darboux first integrals}
\label{rk:BasicKmonomialDarboux}
Definition \ref{def:DarbouxTI} described Darboux totally integrable foliations as pullbacks of monomial foliations. Here we show that the definition agrees with a local description commonly found in the literature, see e.g. \cite[Page 968]{BelRI} {and references therein}. 
We say that a foliation   $\cF$  on a smooth variety $X$  admits global meromorphic (respectively rational) $\mathcal{K}$-\emph{Darboux first integrals}  if there exist  meromorphic ({respectively,} rational) functions $f_{ij} \in K(X)$ and constants $\beta_{ij} \in \mathcal{K}^{\ast}$ for $i=1,\ldots,r$ and $j=1,\ldots,s$ such that a derivation $\partial \in \cD_{X,\pa}^{\log}$ belongs to $\cF_{\pa}$ if and only if 
\[
\sum_{j=1}^{n_i}\beta_{ij} \frac{df_{ij}}{f_{ij}}(\partial) \equiv 0, \quad i=1,\ldots,r
.\] 
{We may, and do,} assume here that $\sum_{j=1}^{n_i}\beta_{ij}\ne 0$. {We can always reduce to this case by repeatedly replacing} the form $\sum_{j=1}^{n_i}\beta_{ij} \frac{df_{ij}}{f_{ij}}$ by $\sum_{j=2}^{n_i}\beta_{ij} \frac{df'_{ij}}{f'_{ij}}$ with $f'_{ij}:=\frac{f_{ij}}{f_{i1}}$. 
Consider a {meromorphic (respectively, rational) map} $$\phi= \prod^r_{i=0}[1:f_{i1}:\ldots:f_{in_i}]:  X\to \prod^r_{i=0} \PP^{n_i}=\prod^r_{i=0}\Proj(K[x_{i0}:\ldots:x_{in_i}])$$ on $\prod^r_{i=0} \PP^{n_i}$ such that $\phi^*(x_{ij}/x_{i0})=f_{ij}.$ 

Define regular logarithmic forms \[
\omega_i:=\sum_{j=1}^{n_i}\beta_{ij} \frac{dx_{ij}}{x_{ij}}+\beta_{i0}\frac{dx_{i0}}{x_{i0}}, \quad i=1,\ldots,r,\]
where $\beta_{i0}:=-\sum_{j=1}^s\beta_{ij}$. The forms $\omega_i$ on $\prod \PP^{n_i}_{x_{ij_i}}$ determine the {$\cK$-monomial} foliation 
 $${\cG}:=\{\partial\mid \omega_i(\partial)=0\}.$$ Indeed on each {affine} space $\prod \PP^{n_i}_{x_{ij_i}}:=\prod^r_{i=0} (\PP^{n_i}\setminus V({x_{ij_i}}))$ the form $\omega_i$ can be represented as $$\omega_i:=\sum_{j=0,j\neq j_i}^{n_i}(\beta_{ij}-\beta_{ij_i})\frac{dx_{ij}}{x_{ij}}=\sum_{j=0,j\neq j_i}^{n_i}(\beta_{ij}-\beta_{ij_i})\frac{dx_{ij}}{x_{ij}}=\sum_{j=0,j\neq j_i}^{n_i}\beta_{ij}\frac{du_{ij}}{u_{ij}},$$  where $u_{ij}:=\frac{x_{ij}}{x_{ij_i}}$ are regular functions on $\prod \PP^{n_i}_{x_{ij_i}}$.

{The saturated pullback $\cF:=\satpull{\phi}(\cG)$} of $\cG$ can be represented as the restriction {$\satpull{\pi}(\cG)_{{\boldsymbol | \boldsymbol |}{\Gamma_\phi}}$ of the pullback $\satpull{\pi}(\cG)$} of $\cG$ on $X\times \prod \PP^{n_i}_{x_{ij_i}}$ to the (closed) graph $\Gamma_\phi\subset X\times \prod \PP^{n_i}_{x_{ij_i}}$ of $\phi: X\to \prod^r_{i=0} \PP^{n_i}$. Moreover the morphism $\Gamma_\phi\to X$ is proper being a composition of a closed embedding followed by the projection along a projective variety.

\section{Main results}\label{sec:Final}
  
\subsection{Functorial principalization over foliated logarithmic manifolds}\label{ssec:FunctorialPrincipalization}

We are ready to prove the main result of the paper.

\begin{proof}[Proof of Theorem \ref{thm:PrincipalizationFoliated}]
{We construct a sequence as in Parts \eqref{It:principalization} and \eqref{It:is-aligned} of the theorem, satisfying \eqref{It:functorial}.}
By taking $\cR = \cO_X[t\cI]$, we reduce the Theorem to Rees algebras instead of ideal sheaves. We define a sequence {of weighted or cobordant blow-ups}:
\[
\sigma_i : (X_{i+1},\cF_{i+1},E_{i+1}) \to (X_{i},\cF_i,E_i),
\]
where $\cF_{i+1} = \sigma_i^c(\cF_i)$ and $\cR_{i+1} = \sigma^c(\cR_i)$ {(respectively, $\cF_{i+1} = \sigma_i^s(\cF_i)$ and $\cR_{i+1} = \sigma^s(\cR_i)$)}, by blowing up the $\inv_{\cF_i}$-maximal $(\cR_i,\cF_i)$-admissible center $\cA_{J_i}$. This center is well-defined by Proposition \ref{prop:FCanonicalInvBasicProperties}. Functoriality follows from \ref{claim:MainInduction}(3). 
{The invariant is upper semicontinuous on $X_i$; in the analytic case of weighted blow-ups, $X_i$ is proper over the relatively compact $X_0\subset X$, and therefore the invariant $\inv$ attains a maximum on $X_i$ along the closed  $\cF_i$-aligned center $J_i$. The same holds true on the cobordant blow-ups by functoriality.
 After each} blow-up the invariant {of the controlled transform} drops by Theorem \ref{prop:DecreaseInvariant}; {the same holds for the strict transform by monotonicity of the invariant, see Propositions \ref{prop:FCanonicalInvBasicProperties}(1,2)}.

{Since the invariant takes values in a well-ordered set, see \ref{claim:MainInduction}\eqref{It:induction-wo-usc}, the sequence is finite, in other words, after finitely many steps the controlled transform of $\cR$ is trivial. This gives Parts {\eqref{It:principalization} and \eqref{It:is-aligned}} of the theorem, satisfying functoriality, Part {\refeq{It:functorial}}.}

{We now consider the other statements.}
Property {\eqref{It:non-singular}} follows from Theorem \ref{Th:aligned center}. {Next, i}f $\cI$ is $\cF$-invariant, then Lemma \ref{lem:FcanonicalInvForInvariantRees} and Proposition \ref{prop:FInvariantBlowup}(2) guarantees that the sequence of blow-ups is $\cF_i$-invariant, proving Part \eqref{It:is-F-invariant}.\footnote{This also follows from $\delta$-invariance, applied to all local sections of $\cF$.} {To prove Part \eqref{It:non-compact}, note that t}he procedure is finite in the algebraic case and over a relatively compact $X_0 \subset X$ in the analytic cases. For a  noncompact coherent analytic space $X$, consider any collection $X_0^{i}$ of relatively compact open subsets such that $\bigcup X_0^{i} = X$ with corresponding modification $X'^{i} \to X_0^{i}$. By functoriality $X'^{i} \times_{X_0^{i}}(X_0^{i} \cap X_0^{j}) = X'^{j} \times_{X_0^{i}}(X_0^{i} \cap X_0^{j})$. These glue together to form  a modification $X' \to X$ giving a principalization of $\cI$. 
\end{proof}

By combining Theorem \ref{thm:PrincipalizationFoliated} with a classical idea in resolution of singularities, we deduce just as in \cite{ATW-weighted,Wlodarczyk-cobordant}:

\begin{proof}[Proof of Theorem \ref{thm:EmbeddedDesingularization}]
We apply Theorem \ref{thm:PrincipalizationFoliated} to the ideal $\cI^Y$ using strict transforms, 
but stop the principalization of $\cI^Y$ as soon as  the invariant along the center is of the form
\[
\inv_{\cF_i,\pa}(\cI^{Y_i})=(1,\ldots,1, \infty+1,\ldots, \infty+1, \infty+\infty + 1,\ldots, \infty+\infty +1).
\]
At this point, the center is supported on an open and closed component $Y^{(j)}_i$ of $Y_i$. Note that this component is isolated from the rest of $Y_i$ since, otherwise, there would be a point $\pa$ where the invariant $\inv_{\cF_i,\pa}(\cI^{Y_i})$ would have an entry $>1$. Therefore we do not blow-up $Y_i^{(j)}$, and we continue the principalization process over the open set $X_i\setminus Y_i^{(j)}$ --- or a relatively compact open in the analytic case. The process stops by compactness, at which point (1') is satisfied as each component is the support of a center, and the other properties follow as they hold in Theorem  \ref{thm:PrincipalizationFoliated}.
\end{proof}

\subsection{Resolution of transverse sections: Proof of Theorem \ref{thm:ResTransverseSections}}\label{ssec:Proofthm:ResTransverseSections}

We start by the following Lemma, which shows that the invariant $\inv_{\cF}$, see Definition \ref{def:FCanonicalInvariantGeneral}, allows us to characterize transverse sections, see Definition \ref{def:TransverseSection}:

\begin{lemma}[Invariant characterization of transverse sections]\label{lem:InvariantTransverseSection} A subvariety $Y \subset X$ of codimension $p$ is transverse to $\cF$ at a point $\pa$ if, and only if:
\[
\inv_{\cF,\pa}(\cO_X[t\cI^{Y}]) = (1,\ldots, 1),
\]
where invariant has $p$ entries. If additionally $\cF$ is of rank $p$ then $Y$ is a transverse section of $\cF$ at $\pa$.
\end{lemma}
\begin{proof}
First, suppose that $Y$ is transverse to $\cF$ at $\pa$. By Theorem \ref{thm:splittingFoliation},
there exists a partial system of nested-regular coordinates $(x_1,\ldots,x_p)$ such that:
\[
\begin{aligned}
\cF \cdot \cO_{X,(x)} &\  \supseteq \  \spa_{\cO_{X,(x)}}(\partial_{x_1},\ldots, \partial_{x_p}), \quad \text{ and } \quad
\cI^{Y} \cdot \cO_{X,(x)} &=(x_1,\ldots,x_p).
\end{aligned}
\]
In particular, we conclude that $\widehat{\cO}_{X,\pa}[t\cI^Y]$ is a formally $\cF$-aligned center. It follows from Definition \ref{def:FCanonicalInvariantCenters}, Lemma \ref{lem:BasicPropertyFInvariantCenter} and the inductive claim \ref{claim:MainInduction} (1), that $\inv_{\cF,\pa}(\cI^{Y}) = (1,\ldots, 1)$, {with} $p$ entries.

Conversely, let $\cR = \cR[t\cI^{Y}]$ and assume that $\inv_{\cF,\pa}(\cR[t\cI^{Y}]) = (1,\ldots, 1)$, with $p$ entries. Note that, from the faithful flatness of $\widehat{\cO}_{X,\pa}$ over $\cO_{X,\pa}$, it is enough to prove the result formally. We now prove the Lemma by induction on the {dimension of the} variety. 
{We have $\ord_{\cF,\pa}(\cI^Y) = 1$, so let $x_1\in \cI^Y$} be an $\cF$-maximal contact element with respect to $\cR$, see Definition \ref{def:FMaximalContactRees}, and let $\partial_1 \in \cF_{\pa}$ be a derivation such that $\partial_1(x_1)$ is a unit. Set $H = V(x_1)$. By Theorem \ref{thm:splittingFoliation}, \( \cF \cdot \widehat{\cO}_{X,\pa} = \Span_{\widehat{\cO}_{X,\pa)}} (\partial_{x_1}, \cF_{|H})
\), where $\cF_{|H}$ is a foliation over $\widehat{\cO}_{H,\pa}$. {Since the $\cF$-order is 1 we have $\cC(\cR,\cF) =\cR$.}  It now follows from Lemma \ref{lem:FcanonicalInvForFiniteRees} that
\[
\inv_{\cF,\pa}(\cR)=(\ord_{\cF,\pa}(\cR),\inv_{\cF_{H},\pa}(\cC(\cR,\cF)_{|H})) = (1,\inv_{\cF_{H},\pa}(\cO_H(t\cI^{Y\cap H}))).
\]
 We conclude using the inductive assumption applied to $Y\cap H$.
\end{proof}

We may now turn to the proof of the theorem:

\begin{proof}[Proof of Theorem \ref{thm:ResTransverseSections}]
We follow the algorithm given by Theorem \ref{thm:EmbeddedDesingularization} up until the invariant is equal to $(1,\ldots,1)$, where the vector has $\codim(Y)(\leq \mbox{rank}(\cF)$) entries. By Lemma \ref{lem:InvariantTransverseSection}, we have not {blown up} points in $Y^{tr}$. Moreover, by Lemma \ref{lem:InvariantTransverseSection} and Proposition \ref{prop:FCanonicalInvBasicProperties}, the maximal locus of the invariant is transverse to $\cF$ and, therefore, a union $Z$ of connected components of the strict transform of $Y$. Applying the theorem  to the complement of $Z$   --- or a relaticely compact open in the analytic case --- the result follows by compactness.
\end{proof}

A complementary situation where $\cF$ is transverse to $Y$ at a point, suggested to us by J. V. Pereira, is the following: 		

\begin{proposition}\label{Prop:sub-transverse} Assume $Y$ irreducible, and \( Y \subset X \) and \( \mathcal{F} \) are smooth at  some point \( x \in Y \). Assume further that  there exists an $E$-adapted partial coordinate system \( x_1, \dots, x_\ell \) at \( x \) such that 
\[
\hat{\mathcal{F}}_x = \text{span}(\partial_{x_1}, \dots, \partial_{x_k}) \quad \text{and} \quad \hat{I}_{Y,x} = (x_1, \dots, x_\ell) \qquad \text{with}\quad k \leq \ell. 
\]
 Then, after resolution, the property is satisfied for all points of the strict transform \( \mathcal{F}' , {Y'} \) on \( X' \).
\end{proposition}

\begin{proof}
Note that $Y$ is smooth of codimension $\ell$ and in normal crossings with $E$. The above hypothesis is satisfied if and only if 
\[
\text{inv}_x(\mathcal{I}_Y,\mathcal{F}) = (\underbrace{1, \dots, 1}_{k}, \underbrace{\infty+1 , \dots, \infty+\infty+1}_{\ell - k}).
\]\end{proof}

\subsection{Thick resolution of singularities}
\label{ssec:Proofthm:ResTotallyIntegrable}
\begin{proof}[Proof of Theorem \ref{thm:ResThickClasses}]\hfill

{\sc Step 1: Resolution of indeterminacies.}  Note that, {up to} resolution of indeterminacy and usual resolution of singularities, we can suppose that $\varphi : X \to B$ is {a morphism} and $X$ {and $B$ are} smooth.

From here on we consider in here the algebraic case; the analytic case follows from the same argument applied to relatively compact open sets $X_0 \subset X$ and the functoriality of the construction.

We also consider here the result for the weighted blow-ups, the arguments for cobordat blow-ups being identical.

{\sc Step 2: Embedding.} 
We  follow the idea of the proof of \cite[Theorem 1.2.12]{ATW-relative} combined with the methods from Section~\ref{Sec:nice-foliations}. Consider the product $Y = X \times B$ with projection $p_B:Y \to B$, the graph $Z = \Gamma(\varphi) \subset Y$ and the SNC divisor $F = E \times B \subset Y$; denote by $\iota= id\times \varphi : X \to Y$ the graph embedding of $X$ with image $Z$, a closed analytic subset. 
Consider the {foliation $\cH$} given by the saturated pullback $\cH=\satpull{(p_B)}(\cG)=(TX \times {\cG} )\cap \cD_{Y}^{\log}$, where we recall that $(B,\cG,E)$ belongs to a thick class $\cC$.
Note that $(Y,\cH,F) \in \cC$ by Definition \ref{def:thick}\eqref{defItem:thickProduct}.
Since  $\varphi= p_B \circ \iota$, we have $\satpull{\phi}(\cF) = i^*(\cH_{{\boldsymbol | \boldsymbol |}Z})$, see Definition \ref{Def:saturated-pullback}\eqref{DefItem:saturated-pullback-Restriction}.

{\sc Step 3: Foliated embedded resolution.} 
Next, we consider the resolution process given by Theorem \ref{thm:EmbeddedDesingularization} applied to the embedding $Z \subset Y$. In other words, there exists a locally finite weighted resolution by aligned centers
\[
\sigma': (Y', {\cH}',F') \to (Y,{\cH},F)
\]
where $\cH'$ is the strict transforms of ${\cH}$. By Definition \ref{def:thick}\eqref{defItem:thickFaligned}, the triple $(Y', {\cH}',F') \in \cC$. Moreover the strict transform $Z'\subset Y'$ of $Z$ is smooth, has normal crossings with the divisor $F'$, and forms an aligned center for $\cH'$. As $Z'$ has normal crossings with $F'$, it is not contained in $F'$.

It follows from Definition \ref{def:thick}\eqref{defItem:thickRestriction} that the restriction $(Z', {\cH}'_{{\boldsymbol | \boldsymbol |}Z'}, F'_{| Z'}) \in \cC$.

{\sc Step 4: Identifying the foliations.} 
Write $\pi':Z' \to X$ for the induced morphism. Since $Z\subset Y$ is unaffected at generic points
we conclude that $\cF':= \satpull{(\pi')}(\cF) = {\cH}'_{{\boldsymbol | \boldsymbol |}Z'}$. It follows that $(Z',\cF',F'_{| Z'})$ belongs to class $\cC$, as needed.
\end{proof}

\appendix

\section{Transforms of singular foliations} \label{App:transforms}
\subsection{Saturated foliations via differential forms}
\subsubsection{Pfaffian presentation} 
Denote by $\Omega_X^{\log}=\Omega_X(\log E)$ the sheaf of logarithmic forms. It is  generated locally by the forms $dx_i$ associated with  free coordinates,  along with $\frac{dx_j}{x_j}$ associated with divisorial coordinates. It is the dual of  $\cD^{\log}_{X}$, and is compatible with the notions introduced in \cite{Kato-log}.

{Fixing a point} $\pa\in X$, recall that a $1$-form {$\omega \in \Omega^{\log}_{X,\pa}$} can be identified with a homormophism $\cD^{\log}_{X}\to \mathcal{O}_{\pa}$, which we may further restrict to a homeomorphism $\mathcal{F}_{\pa} \to \mathcal{O}_{\pa}$. There exists, therefore, a canonical map $$\Omega^{\log}_{X} \to \mbox{Hom}(\mathcal{F}, \mathcal{O}_{X}).$$ 

We denote by
$N_\cF^*=\mathcal{F}^{\perp}$
the kernel of this map, which is a coherent sub-sheaf of $\Omega^{\log}_{X}$.
This gives the exact sequence 
$$0\to \cF^\perp\to \Omega^{\log}_{X}\to \cF^*.$$
 The fact that $\mathcal{F}$ is closed under Lie bracket operations, moreover, implies {that $\mathcal{F}^{\perp}$ satisfies} the following property, for every $\pa \in X$: 
\begin{equation} \label{Eq:pfaffian involutive}
d\omega \wedge \omega_1 \wedge \cdots \wedge \omega_d \equiv  0, \quad \forall \omega,\, \omega_1,\ldots,\omega_d \in {\mathcal{F}^{\perp}_{\pa}} \quad \text{where $d  =n-\mbox{rank}(\mathcal{F})$.}
\end{equation}

Similarly for any coherent subsheaf $N^*\subset \Omega_X^{\log}$ satisfying Equation \eqref{Eq:pfaffian involutive} we denote $\cF_N={N^*}^\perp$. It is the kernel of the map $$D_X^{\log}=Hom(\Omega_X^{\log},\cO_X)\to (N^*)^*=Hom(N^*,\cO_X)$$ so that we have the exact sequence
$$0\to \cF_N\to \cD^{\log}_X\to (N^*)^*.$$
The assumption that $N^*$ satisfies Equation \eqref{Eq:pfaffian involutive} implies that $\cF_N$ is a foliation.

We denote by $\cF^{sat}:=\mathcal{F}^{\perp \perp}$ the {\it saturation of} $\cF$, again a coherent sub-sheaf of $\cD^{\log}_X$ which can be shown to be closed under the Lie-bracket operation and saturated.

If $\cF$ is {\Sm} at $\pa$ then it is locally a direct summand of $\cD_{X,\pa}^{\log}=\cF_\pa\oplus (\cF^\perp_{\pa})^*$, and thus $\cF=(\cF^\perp)^\perp$ is saturated. Since  any foliation $\cF$ is saturated over its {\Sm} locus $U$, we have
\begin{lemma}\label{Lem:saturated-U} The saturation of a foliation $\cF$ is determined by its restriction to the  {\Sm} locus $U \subset (X,\cF)^\nsing$, as follows:
 $$\cF^\sat=\{\partial\in \cD^{\log}_X\mid \partial_{|U}\in \cF_{|U} \}\subseteq \cD^{\log}_X.$$
\end{lemma}

\subsubsection{Saturated pullback of foliations}\label{Sec:pullback-foliation} 
Pullbacks of foliations in the algebraic and analytic literature are typically considered in the saturated case, as we recall now, taking into consideration the logarithmic structure. The literature typically restricts to dominant morphisms, an assumption we do not require.  Saturated pullbacks will be used in Section \ref{ssec:Proofthm:ResTotallyIntegrable} and Lemma \ref{lem:PresentationMonomialAligned1}.

For any morphism $\phi: (X,E)\to (Y,F)$ consider the natural pullback of forms 
    $$d\phi^*: \ \phi^{*}(\Omega^{\log}_Y)  \quad \to \quad \Omega^{\log}_X.$$

\begin{definition}[Saturated pullback of foliation]\label{Def:saturated-pullback}  

Given a saturated foliation $\cG$ on $Y$, we assume that $\phi$ satisfies the following: Let $V \subset  Y$ be a Zariski-open set where $\cG$ is regular and $U\subset \phi^{-1} V$ be a Zariski-open set where $\phi$ is  regular. Assume $U$ is dense in $X$.

\begin{enumerate}
\item 
Define the {\it logarithmic saturated pullback of $\cG$ }
under $\phi$ as above to be the sheaf of the logarithmic derivations 
$$\satpull{\phi}(\cG):=d\phi^*(\cG^\perp)^\perp \subset \cD^{\log}_X.$$
We will suppress the adjective ``logarithmic" from here on.

\item\label{DefItem:saturated-pullback-Restriction} If $i: (X,E) \to (Y,F)$ is a  closed embedding of a smooth subvariety $X$  into a smooth variety $Y$ then the saturated pullback ${\satpull{i}(\cG)}$ will be called the {\it saturated restriction} and denoted by
$\cG_{{\boldsymbol | \boldsymbol |}X}$.

\item More generally, let  $\phi: (X,E)\dashrightarrow (Y,F)$ be a rational or meromorphic map, and $\cG$ a rational or  meromorphic foliation. Let $V \subset  Y$ be a Zariski-open set where $\cG$ is regular and $U\subset \phi^{-1} V$ be a Zariski-open set where $\phi$ is  regular.
Assume $U$ is dense in $X$.
 Then we define $$\satpull{\phi}(\cG)\quad :=\quad \{\ \partial\in \cD^{\log}_X\ \ \mid\ \  \partial_{|U}\in \satpull{(\phi_{|U})}(\cG)\ \}.$$

\end{enumerate}
	\end{definition}

We note that part (3) is well defined by Lemma \ref{Lem:saturated-U}.

\begin{lemma}\label{lem:SatPullBack} The saturated pullback $\satpull{\phi}(\cG)$ of a  rational or meromorphic foliation is a foliation.
\end{lemma}
\begin{proof}
First, assume that $\phi$ is a regular map and $\cG$ is a regular foliation. We do retain the assumption that the preimage $U$ of the {\Sm} locus $V$ of $\cG$ is dense. By construction, the sheaf $\satpull{\phi}(\cG)$ is coherent.
We  use the De Rham--Saito theorem \cite[Theorem 1]{Saito-deRham}:
Since $\cG^\perp$ satisfies Equation \eqref{Eq:pfaffian involutive} on page \pageref{Eq:pfaffian involutive}, the restriction $\cG_{|V}^\perp$ has the stronger property that every local section $\omega$ can be written as $d\omega = \sum \eta_i \wedge \omega_i$, for some sections $\eta_i$ of $\Omega_V^{\log}$ and $\omega_i$ of $\cG^\perp_{|V}$. Pulling back, 
 the same holds for $d\phi^*(\cG^\perp)_{|U}$, which implies Equation \eqref{Eq:pfaffian involutive} in general. It follows that $\satpull{\phi}(\cG) = (d\phi^*(\cG^\perp))^\perp$ is a foliation. 

In the general case, we first replace $\cG$ by the unique saturated holomorphic foliation which coincides with $\cG$ along $V$. Eliminating indeterminacies, there is a proper bimeromorphic $\pi:X' \to X$ such that $\phi_1 = \phi\circ \pi: X_1 \to Y$ is holomorphic. So $\satpull{(\phi_1)}(\cG)$ is a foliation $\cG_1$ on $X_1$. The saturation of its pushforward in $X$ is coherent, and agrees with $\satpull{\phi}(\cG)$, as they agree on a dense open $U$. 
 It is involutive on the dense open set $U$, hence it is involutive.
\end{proof}

\subsection{Pullbacks of distributions in the nonsaturated case}\label{sec:PullBackNonSaturated}

\subsubsection{Lift of derivations}\label{ssec:LiftDerivation}
For a morphism $\phi: (X,E)\to (Y,F)$ we again
consider the natural pull back of forms 
    $d\phi^*: \ \phi^{*}(\Omega^{\log}_Y)  \to \Omega^{\log}_X$ and its dual $d\phi: \cD_X^{\log} \to \phi^*(\cD_Y^{\log})$. 

Consider an element $\partial \in \cD^{\log}_Y(U)$ over some open $U \subset Y$. It gives rise to a section 
\[
\phi^*(\partial) \in \phi^*(\cD^{\log}_Y)(\phi^{-1}(U)).
\]
We say that $\partial$ \emph{lifts to $X$} if $\phi^*(\partial)$ is in  the image of $d\phi: \cD_X^{\log} \to \phi^*(\cD_Y^{\log})$ over $\phi^{-1}(U)$.

    \subsubsection{Nonsaturated  distribution pullback}  
    Let $\cG\subset \cD_Y^{\log}$ be a distribution. Pulling back, we have a morphism of coherent sheaves $j:\phi^*(\cG) \to \phi^*(\cD_Y^{\log})$ with image the sheaf 
    $j(\phi^*\cG)$.

    \begin{definition}[Nonsaturated  distribution pullback]\label{Def:fol-pullback} \hfill
    \begin{enumerate}
    \item For any morphism $\phi: (X,E)\to (Y,F)$ and any distribution $\cG\subseteq \cD^{\log}_Y$ on $Y$ we define the \emph{distribution pullback} $\dispull{\phi}(\cG) = j(\phi^*(\cG)) \times_{\phi^*(\cD_Y^{\log})} \cD_X^{\log} = (d\phi)^{-1} (j(\phi^*(\cG))) $ as in the following Cartesian diagram:
    $$
    \xymatrix{ & \phi^*\cG \ar[d] \\ 
    \dispull{\phi}(\cG) \ar[r] \ar@{^(->}[d] & j(\phi^*(\cG)) \ar@{^(->}[d] \\ \cD_X^{\log} \ar[r]^{d\phi} & \phi^*(\cD_Y^{\log}).
 }$$
 \item Let $i: (X,i^{-1}(F))\to (Y,F)$ be a closed embedding of a smooth subvariety $X\not\subset F$ into the variety $Y$ and let 
$\cG$ be a distribution $\cG\subseteq \cD^{\log}_Y$ on $Y$. Then the distribution pullback $\dispull{i}(\cG)$ will be called the {\it nonsaturated restriction} and denoted by
$\cG_{|X}$. 
\item Suppose, furthermore, that both $\cG$ and $\dispull{\phi}(\cG)$ are foliations. Then we may say that $\dispull{\phi}(\cG)$ is the \emph{foliation pullback}, and denote it by $\folpull{\phi}(\cG)$.
 \end{enumerate}
  \end{definition}

Note that $\dispull{\phi}(\cG)$ contains the relative sheaf  of derivations $\cD^{\log}_{X/Y}$ as we have  the exact sequence
  $$
  0\to \cD^{\log}_{X/Y}\to \cD^{\log}_X \to \phi^*(\cD_Y^{\log}).
  $$

\begin{remark}\label{Rem:non-saturated-restriction} The nonsaturated restriction plays {a role} in the construction of the center and in the algorithm, specifically the restriction to a maximal contact hypersurface  (see Section \ref{sec:Fol:LocalSplittingFoliations}). This notion does come with its challenges:  it is not clear in what generality the distribution pullback of a foliation is a foliation. It is not clear in what generality it is functorial. And it is in general not the case that sections of $\dispull{\phi}(\cG)$ are generated by lifts of sections in $\cG$, as $\phi^*(\cG)$ consists of $\cO_X$-linear combinations of such sections. We do verify these properties in the key case of smooth morphisms (Lemma \ref{Lem:smooth-pullback}). One can also verify this for transversal restriction (Lemma \ref{Lem:restriction-and-ideal} and an inductive application of Theorem \ref{thm:splittingFoliation}).
\end{remark}

\begin{lemma}[Criterion for $\dispull{\phi}(\cG)$ being a foliation]\label{lem:NonSaturatedPullbackIsFoliation} 
Let $\phi: (X,E)\to (Y,F)$ be a morphism and $\cG\subseteq \cD^{\log}_Y$ be a foliation. Suppose that:
\[
\dispull{\phi}(\cG) = \spa_{\cO_X}\{ \partial' \ \mid \ \ \text{$\partial'$ is a lift of a derivation from $\cG$} \}.
\]
then the distribution $\dispull{\phi}(\cG)$ is a foliation.
\end{lemma}
\begin{proof}
In fact, given two lifted derivations $\partial_1$ and $\partial_2$, we verify directly
\[
[\partial_1,\partial_2](\phi^{\ast}_{\pa}(f)) = \phi^{\ast}_{\pa}([\partial_1',\partial_2'](f)), \quad \forall f\in \cO_{Y,\phi(\pa)}.
\]
implying that $[\partial_1,\partial_2]$ is a lifted derivation {as well.}
\end{proof}

\subsubsection{Nonsaturated restriction}\label{Sec:nonsaturated-restriction}
Let $\phi$ be a closed embedding with ideal $\cI_X$. Then we have an exact sequence of $\cO_Y$-modules
$$\xymatrix{0 \ar[r] & \cD_X^{\log} \ar[r]^{d\phi} & \phi^*(\cD_Y^{\log}) \ar[r] & \cHom(\cI_X, \cO_X)}.$$ Also the arrow $\cD^{\log}_Y \to \phi^*(\cD^{\log}_Y)$ is a surjective morphism of $\cO_Y$-modules. Thus we get the following well-known formula:
\begin{lemma}\label{Lem:restriction-and-ideal} If $i: (X,i^{-1}(F))\to (Y,F)$ is a closed embedding then $$\cG_{|X}=\spa_{\cO_X}\{\,\partial_{|X}\ \mid \ \ \partial \in \cG, \quad\partial(\cI_X)\subseteq \cI_X\}.$$
If $\cG$ is a foliation, then $\cG_{|X}$ is a foliation.
 \end{lemma}
 The second statement follows from the formula and Lemma \ref{Lem:restriction-and-ideal}.
 
 This formula is used in the construction of the center in Section \ref{sec:Fol:LocalSplittingFoliations}.\footnote{Functoriality is not true without transversality: $\cF = (\partial_z + x\partial_x)$ and $Y = V(x+y)$, $X = V(x,y)$.}

\subsubsection{Nonsaturated distribution pullback {and descent} under a smooth morphism}\label{Sec:nonsat-smooth} 
\begin{lemma}\label{Lem:smooth-pullback} The {distribution pullback} of a {distribution} $\cG$   over $Y$ under the smooth morphism $\phi:(X,\phi^{-1}(E))\to (Y,E)$ is the subsheaf of $\cD_X^{\log}$ {given by}
\[
\dispull{\phi}(\cG) = \spa_{\cO_{X}}\{ \partial' \in \cD_{X}^{\log} \mid \text{$\partial'$ is a lift of a derivation from $\cG$} \}.
\]
In particular, if $\cG$ is a foliation, then $\folpull{\phi}{(\cG)}=\dispull{\phi}{(\cG)}$ is a foliation.

If $\phi= \phi_1\circ\phi_2$ is a composition of smooth morphisms, then $\dispull{\phi}(\cG) = \dispull{(\phi_1)}(\dispull{(\phi_2)}(\cG))$.
\end{lemma}
\begin{proof} 
We may work locally, therefore we may assume given elements $u_1,\ldots,u_k\in \cO_{X}$ inducing an \'etale morphism $X \to Y \times \AA^k$. This provides a splitting $\psi: \phi^*(\cD^{\log}_Y) \to \cD_X$ of 
$$0\to \cD^{\log}_{X/Y}\to \cD^{\log}_{X}\to \phi^*(\cD^{\log}_Y)\to 0,$$ so  we may write  $\cD^{\log}_{X} = \cD^{\log}_{X/Y} \oplus  \phi^*(\cD^{\log}_Y)$. It is determined by $\psi(\partial) (u_i) = 0$ for all $\partial \in  \cD^{\log}_Y$.
In particular, every section of $\cD^{\log}_Y$ lifts. Applying this to the restricted exact sequence 
$$0\to \cD^{\log}_{X/Y}\to \dispull{\phi}(\cG) \to \phi^*(\cG)\to 0,$$ we see that $\dispull{\phi}(\cG) =  \cD^{\log}_{X/Y} \oplus \phi^*(\cG)$, where $\phi^*(\cG)$ is generated by liftings of derivations on $Y$ of the form $\psi(\partial)$ and   $\cD^{\log}_{X/Y}$, the lifts of 0.

The composition statement follows by locally splitting $\cD^{\log}_{X/Y}$ as $\cD^{\log}_{\phi_1} \oplus \cD^{\log}_{\phi_2}$.
\end{proof}

In the opposite direction, we have the following: 
\begin{lemma}\label{Lem:smooth-descent} Assume $\phi$ is smooth and surjective, and write $X_2 = X\times_YX$ with diagram  
 $$\xymatrix{ X_2 \ar[r]^{\pi_2}\ar[d]_{\pi_1}\ar[dr]|-{\phi_2} & X \ar[d]^\phi \\ X \ar[r]_{\phi} & Y.}$$
 Let $\cG_X\subset \cD^{\log}_X$ be a foliation \emph{containing $\cD^{\log}_{X/Y}$}. Assume $\folpull{{\pi_1}}(\cG_X) = \folpull{{\pi_2}}(\cG_X) \subset \cD^{\log}_{X_2}$. Then there is a unique $\cG\subset \cD^{\log}_Y$ such that $\folpull{\phi}(\cG) = \cG_X$.
 \end{lemma}
 \begin{proof} We obtain an equality $$\pi^*(\cG_X/\cD^{\log}_{X/Y}) = \pi_2^*(\cG_X/\cD^{\log}_{X/Y}) \quad \subset \quad \cD^{\log}_{X_2}/ \cD^{\log}_{X_2/Y}= \phi_2^*(\cD^{\log}_Y),$$ which, by faithfully flat descent, gives a unique subsheaf $\cG \subset \cD^{\log}_Y$. This sheaf is automatically a foliation: using the splitting $\psi$ in the proof of the previous lemma, one checks that for $\partial_i \in \cG$, we have $[\psi(\partial_1),\psi(\partial_2)] = \psi([\partial_1,\partial_2])$. 
 \end{proof}

\subsection{Transforms of {generically \'{e}tale} morphisms} \label{Sec:birational-transform} 
We now consider {\it  generically \'etale}  morphisms $\phi: (X,E) \to (Y,F)$ between smooth logarithmic varieties satisfying the extra hypothesis that:
\begin{equation}\label{Eq:divisor-D} \text{$\phi$ is \'{e}tale everywhere outside of a SNC divisor $D \subset E$.}
\end{equation}
Generically \'etale morphisms {satisfying this property} arise as charts for stack-theoretic blow-ups, where $D$ arises as exceptional divisors, or, more generally. birational morphisms. Our goal is to  specify how to transform foliations in this key setting and, in particular, to define a non-saturated version of the \emph{strict} transform.

Consider again the morphism $d\phi: \cD_X^{\log} \to \phi^*(\cD_Y^{\log})$. It is an isomorphism away from the divisor $D$, so we have an inclusion $\phi^*(\cD_Y^{\log}) \subset  \cD_X^{\log}(*D)$, where $\cD_X^{\log}(*D)$ is the quasi-coherent sheaf of meromorphic sections of $\cD_X^{\log}$ with poles of arbitrary order restricted to $D$.

\begin{definition}[Total transform of a foliation]\label{def:TotalTransformFoliation}
The {\it total} transform of a foliation $\cG \subset \cD^{\log}_Y$ by $\phi: (X,E) \to (Y,F)$ is the $\cO_X$-sub-sheaf $\phi^*(\cG)\subset \cD_X^{\log}(*D)$.
\end{definition}

The total transform is meromorphic, and rarely involutive --- only $\phi^*(\cG)(*D)$ is guaranteed to be involutive.  In analogy with the transforms of ideals, we will work with two  transforms, the \emph{controlled} and \emph{strict} transform, which are holomorphic and involutive. They are defined in somewhat different terms.

We start by considering how derivations may be transformed.
A derivation $\partial \in \cD^{\log}_{{Y}}$ on $Y$ determines a unique meromorphic derivation  $\phi^*(\partial) \in \cD^{\log}_{{X}}(*D)$, which is again called the  {\it total transform}. There is a minimal effective divisor $r_1D_1+ \cdots + r_k D_k$ such that  
$$
\phi^*(\partial)\,\in \, \cD^{\log}_{{X}}\left(\sum r_iD_i\right) \subset  \cD^{\log}_{{X}}(*D).
$$ 
Choosing local defining sections $y_i \in \cO_X(-D_i)$ there is a minimal power  $r=(r_1,\ldots,r_k) \in  \ZZ_{\geq 0}^k$ (respectively $m=(m_1,\ldots,m_k) \in \ZZ^k$) such that 
$$
\phi^c(\partial):=y^r_{D}\cdot \partial=y_1^{r_1}\cdot\ldots\cdot y_k^{r_k} \partial
$$ 
respectively 
$$
\phi^s(\partial):=y^m_{D}\cdot \partial=y_1^{m_1}\cdot\ldots\cdot y_s^{m_k} \partial 
$$ 
determines a regular derivation in   $\cD_X^{\log}$. We call $\phi^c(\partial)$ the \emph{controlled transform} of $\partial$, and  $\phi^s(\partial)$ {\it strict transform} of $\partial$.
Note that, in general the  controlled and strict transforms of the derivation are defined up to a unit. 

Strictly speaking, the strict transform of a derivation depends on the choice of the divisor $D$, since $\partial$ itself may vanish along a divisor. In all applications $D$ will be uniquely determined as the exceptional locus of $\sigma$.

\begin{definition}[Controlled  transform] \label{def:TransformFoliation} The {\it controlled} transform $\phi^{c}(\cG)$  
of a foliation $\cG \subset \cD^{\log}_Y$ by $\phi: (X,E) \to (Y,F)$ is the $\cO_X$-subsheaf of $\cD_X^{\log}$ whose stalks are  generated by the controlled  
transforms
of  $\partial\in \cD_Y^{\log}$.
\end{definition}

\begin{lemma}\label{lem:ControlledStrictFoliaitons}
The controlled transform $\phi^{c}(\cG)$ is a  foliation.
\end{lemma}
\begin{proof}
We show that $\phi^{c}(\cG)$ is closed by Lie brackets.
To this end, let $\partial_1,\partial_2 \in \cG$ and consider their controlled transforms:
\[
\phi^{c}(\partial_1) = y^{r_1}_D \cdot \phi^{\ast}(\partial_1), \quad  \phi^{c}(\partial_2) = y^{r_2}_D \cdot \phi^{\ast}(\partial_2),
\]
which are well-defined logarithmic derivations over $\cD_X^{\log}$. In particular:
\[
\begin{aligned}
\cD_X^{\log} \, \ni \, [\phi^{c}(\partial_1),\phi^{c}(\partial_2) ] &=\\ y^{r_1+r_2}[\phi^{\ast}(\partial_1),&\phi^{\ast}(\partial_2)] + \left\{\left[  \phi^{c}(\partial_1)(y_D^{r_2}) \right]\cdot \phi^{\ast}(\partial_2) - \left[  \phi^{c}(\partial_2)(y_D^{r_1}) \right]\cdot \phi^{\ast}(\partial_1) \right\}.
\end{aligned}
\]
The result now follows from two remarks. First, as $\phi^{c}(\partial_1),\,\phi^{c}(\partial_2) \in \cD_X^{\log}$, we conclude that $\phi^{c}(\partial_1)(y_D^{r_2}) \in (y_D^{r_2})$ and $\phi^{c}(\partial_2)(y_D^{r_1}) \in (y_D^{r_1})$ , implying that the term in brackets belongs to $\sigma^{c}(\cG)$. Second, and as a consequence of the first remark, we know that:
\[
\cD_X^{\log} \, \ni  y^{r_1+r_2}[\phi^{\ast}(\partial_1),\phi^{\ast}(\partial_2)] = y^{r_1+r_2}\phi^{\ast}\left([\partial_1,\partial_2]\right),
\]
and we conclude that this term must be a multiple of $\phi^{c}\left([\partial_1,\partial_2]\right) \in \sigma^{c}(\cG)$, as needed. 
\end{proof}

In contrast to the controlled transform, the \emph{strict transform} is defined sheaf-theoretically:
  
\begin{definition}[Strict transform] \label{def:StrictTransformFoliation} 
The {\it strict} transform $\phi^s(\cG)$ of a  foliation $\cG \subset \cD^{\log}_Y$ by $\phi: (X,E) \to (Y,F)$ having exceptional divisor $D$ is the $\cO_X$-subsheaf of $\cD_X^{\log}$ given by
\(
\left[ \Phi^{\ast}(\cG) \otimes \cO(\ast D)\right] \cap \cD^{\log}_X.
\)
\end{definition}

\begin{lemma}
The strict transform $\phi^{s}(\cG)$ is a foliation.
\end{lemma}
\begin{proof}
It is enough to show that $\phi^{s}(\cG)$ is closed by Lie brackets. Let $\partial_1,\ldots, \partial_r$ be local generators of $\cG$ and consider $\nabla_1,\, \nabla_2 \in  \phi^{\ast}(\cG) \otimes \cO(\ast D)$, that is
\(
\nabla_i = \sum_{j=1}^r  f_{ij} \phi^{\ast}(\partial_j), \quad i=1,2,
\)
where $f_{ij} \in \cO(\ast D)$. Recall that $\phi^{\ast}(\cG) \in \cD_X^{\log}(\ast D)$, so that $\phi^{\ast}(\partial_k) (f_{ij}) = f_{ijk} \in \cO_X(\ast D)$, allowing us to conclude that:
\[
\begin{aligned}
[\nabla_1,\nabla_2 ] = \sum_{jk} f_{1j}f_{2k} \phi^{\ast}[\partial_j,\partial_k] +   f_{1j} f_{2kj} \phi^{\ast}(\partial_k) - f_{2k} f_{1jk}\phi^{\ast}(\partial_j),
\end{aligned}
\] 
belongs to $\phi^{\ast}(\cG) \otimes \cO(\ast D)$, as $\cG$ is closed by Lie brackets. We finish the proof by recalling that $\cD_X^{\log}$ is closed by Lie bracket, and so is the intersection \(
\left[ \Phi^{\ast}(\cG) \otimes \cO(\ast D)\right] \cap \cD^{\log}_X.
\)
\end{proof}

\begin{remark}[On the controlled transform, strict transform, and pullback]\hfill\label{Remarks:transforms} 
\begin{enumerate}	

\item\label{rk:ControlledStrictTransformSmooth} If $\phi:(X,E) \to (Y,F)$ is an \'etale morphism, then $\dispull{\phi}(\cG)= \phi^{\ast}(\cG) =\phi^c(\cG)$.  Moreover if $\cG$ is saturated, then $\phi^{\ast}(\cG) =\phi^c(\cG)=\phi^s(\cG)=\satpull{\phi}(\cG)=\dispull{\phi}(\cG)$.

\item The key generically \'{e}tale morphisms used in this paper are called \emph{$\cF$-aligned blow-ups}, where $\cF$ is a foliation, see Definition \ref{def:Faligned}. For $\cF$-aligned blow-ups, we show that $\phi^{c}(\cF) = \folpull{\phi}(\cF)$ in Proposition \ref{prop:TransformFoliationAligned}, therefore providing compatibility between the two notions of pull-backs.

\item  The strict transform of a foliation does not behave well with respect to controlled transforms of ideal sheaves (see Definition \ref{def:ControlledTransformReesAlgebra}) as we illustrate in example \ref{ex:ControlledVsStrict} below. This fact was noticed in \cite{Belthesis,BelJA} and justifies the introduction of the controlled transform. 
In particular, the notion of controlled transforms $\phi^{c}(\cF)$ by (a non-weighted and analytic) $\cF$-aligned blow-up coincides with the transform used by Belotto in {\cite[page 44]{Belthesis}}, where $\folpull{\phi}(\cF)$ was defined for blow-ups under the name {\it strict analytic transform}.  
\item The pullback ${\folpull{\phi}}(\cF)$ of a subsheaf $\cF\subseteq\cD_{X/B}$ was considered in \cite[\S5.3]{ATW-relative}, though in that paper this {notion was not}  essential. 
\item We note that other work in birational geometry uses the saturated pullback of a saturated foliation, sometime naming it ``the strict transform". 
{Note  that if $\cF$ is saturated then  $\phi^s(\cF)= \satpull{\phi}(\cF)$, so the notions are compatible.} 
\item The total, controlled, and strict transforms are tools for working with birational transformations, and are not meant to be functorial, in analogy with the transforms of ideals, and in contrast with the saturated pullback $\satpull{\phi}(\cF)$.

\end{enumerate}
\end{remark}

\subsection{Pullback and descent of relative foliations}\label{Sec:relative-pullback}\label{Sec:relative} We consider the relative case only in the very restrictive situations we need. 

\subsubsection*{Relative foliations} Consider a smooth and logarithmically smooth morphism $(Y ,F) \to (Z,H)$ of varieties with normal crossings divisors. A foliation $\cG \subset \cD^{\log}_{Y/Z} \subset \cD^{\log}_{Y}$, in other words, a foliation which annihilates functions coming from $Z$,
 is said to be a \emph{relative foliation}. 
 
\subsubsection*{Pullback} If $(W,G) \to (Z,H)$ is another smooth and logarithmically smooth morphism of varieties with normal crossings divisors, we may form the cartesian diagram
 $$\xymatrix{ X \ar[r]^\phi\ar[d] & Y \ar[d] \\ W \ar[r] & Z.}$$ 
 On the fibered product $X$ we obtain the subsheaf $$\cG_X:= \phi^*(\cG)\quad  \subset \quad  \cD^{\log}_{X/W} \simeq \phi^*(\cD^{\log}_{Y/Z}).$$ This is automatically a relative foliation. It is well-defined on individual derivations: a section $\partial$ of $\cG$ provides a section $\phi^*(\partial)$ of $\phi^*\cG$. The process is functorial for compositions of smooth and logarithmically smooth morphisms $W_1 \to W_2 \to Z$.
 
\subsubsection*{Descent}  This situation allows for particularly convenient descent arguments, precisely because  $\cD^{\log}_{X/W} = \phi^*(\cD^{\log}_{Y/Z})$:  let $X_2 = X\times_YX$ and $W_2 = W\times_ZW$, write $\pi_i: X_2 \to X$ for the two projections, and $\phi_2= \phi\circ \pi_2: X_2 \to Y$ for the induced morphism: 
 $$\xymatrix{ X_2 \ar[r]^{\pi_2}\ar[d]_{\pi_1}\ar[dr]|-{\phi_2} & X \ar[d]^\phi \\ X \ar[r]_{\phi} & Y.}$$
\begin{lemma} Suppose given a relative foliation $\cG_X \subset \cD^{\log}_{X/W}$ with an equality $$\pi_1^*\cG_X = \pi_2^*\cG_X\quad \subset\quad  \cD^{\log}_{X_2/W_2}=\phi_2^*(\cD^{\log}_{Y/Z}),$$ and assume $\phi$ is surjective, Then the unique sheaf $\cG \subset \cD^{\log}_{Y/Z}$ with $\cG_X= \phi^*(\cG)$ provided by flat descent is automatically a relative foliation.
\end{lemma}

This is a consequence of Lemma \ref{Lem:smooth-descent} (though a direct proof would be  shorter).

\subsection{Split transforms}
\label{Sec:split-pullback-foliation}
The definitions of controlled and strict transforms can be combined with the notion of relative foliations.

Consider a commutative diagram  
 $$\xymatrix@R=5mm{ X' \ar_{\phi_0}[d] \ar[rd]^{\phi'} \\  X \ar[r]_\phi\ar[d] & Y \ar[d] \\  W \ar[r] & Z.}$$
where the bottom square is as before, and $X' \to X$ is bimeromorphic.  Given a relative foliation $\cG \subset \cD^{\log}_{Y/Z}$ we have considered above the relative foliation $\phi^*\cG \subset \cD^{\log}_{X/W}$. 

By definition we have $\phi'^*(\cG) = \phi_0^*(\phi^*(\cG))$. We now \emph{define} $\phi'^c(\cG) := \phi_0^c(\phi^*(\cG))$  and $\phi'^s(\cG) := \phi_0^s(\phi^*(\cG))$, extending the controlled and strict transforms to this setting --- this will not cause confusion since when $\phi$ is \'etale this coincides with the previously defined transform. Again this is well-defined on individual derivations: $\phi'^*(\partial) = \phi_0^*(\phi^*(\partial))$, $\phi'^c(\partial) := \phi_0^c(\phi^*(\partial))$ up to units,   and $\phi'^s(\partial) := \phi_0^s(\phi^*(\partial))$ up to units. 

As $\cD^{\log}_{X/Z} = \cD^{\log}_{X/W} \oplus \cD^{\log}_{X/Y}$, we have that $\folpull{\phi}(\cG) = \phi^*(\cG)\oplus  \cD^{\log}_{X/Y}$ --- the term $\cD^{\log}_{X/Y}$ is split off in $\phi^*(\cG)$ and does not affect $\phi'^*(\cG), \phi'^c(\cG) $ and $\phi'^s(\cG)$.  
\begin{definition} \label{Def:split-transforms} We  name $\phi'^*(\cG), \phi'^c(\cG) $ and $\phi'^s(\cG)$ the \emph{total, controlled, and strict} transforms of $\cG$ under $\phi'$,  respectively, in the split situation. 
\end{definition}

This generalizes the construction of Section \ref{Sec:birational-transform}.
The construction in Sections \ref {Sec:relative-pullback} and \ref{Sec:split-pullback-foliation} is used in 
 Sections \ref{ssec:CobordantBU} and  \ref{ssec:TransformFoliations}.

\subsection{Etale charts of weighted blow-ups and transforms of derivations}

\subsubsection{\'Etale charts}\label{ssec:StackDescent}\label{Sec:stack-descent}
Recall that a smooth algebraic or analytic \emph{orbifold} is locally (with respect to the \'{e}tale and Euclidean topology respectively) given by a stack quotient $[X/G]$, where $X$ is a smooth variety or an analytic space and $G$ is a group scheme or respectively a Lie group acting on $X$, where the stabilizers $G_x \subset G$, for $x\in X$, are finite groups which are generically trivial. Arguably, the ``typical" example of an orbifold is the stack quotient $[\mathbb{A}^n/G]$, where $G$ is a finite group acting on $\mathbb{A}^n$ {linearly}.

When passing from cobordant blow-ups to the stack quotient $[B_+/\GG_m]$ we obtain the definition of weighted blow-ups of orbifolds or stacks, see Definition {\ref{CobordantBlow-up}(3)}. In {this section,} we provide different presentations of this quotient in order to have concrete expressions, {used in Theorem \ref{Th:aligned center new}.}

We have described $B_+$ locally over $X$ as $B \setminus V(x'_1,\ldots,x'_k)$, where $$B\ =\ \Spec_X \cA_J \ \ = \ \ \Spec_X \cO_X[s,x_1',\ldots, x_k'] / (x_1 - s^{w_1} x_1',\ldots ,x_k - s^{w_k} x_k').$$
  The scheme $B$ is covered by the open charts $B_{i}:= B \setminus V(x'_i)\ \  =\ \  \Spec_X A_j[x_i'^{-1}].$ 
  
  Setting $x_i'=1$ write $W_i  = \Spec \cA_J/(x_i'-1)\subset B_{i}$ with induced morphism $\tau_i: W_i \to X$.  The subgroup $\bmu_{w_i} \subset \GG_m$ stabilizes the equation $x_i - 1=0$ and thus $W_i$. The embedding $W_i \subset  B_{i}$ is equivariant for $\bmu_{w_i} \subset \GG_m$, inducing a morphism of stacks $[W_i/\bmu_{w_i}] \to  [B_{i}/\GG_m]$.
  
  \begin{lemma}\cite[Lemma 1.3.1]{Quek-Rydh} 
  The morphism $[W_i/\bmu_{w_i}] \to  [B_{i}/\GG_m]$ is an isomorphism.
  \end{lemma}
  
  This implies that, to understand any structure on $Bl_J(X) = [B_+/\GG_m]$, it is enough to understand it on $[W_i/\bmu_{w_i}]$, in other words, as a $\bmu_{w_i}$-equivariant structure on the \'etale chart $W_i$. 

Rather than repeat the simple proof of this lemma in \cite{Quek-Rydh}, let us mention why it is true: it is bijective on points and isotropy groups. This is sufficient for normal stacks by a result of Asgarli--Inchiostro \cite[Theorem A.5]{Asgarli-Inchiostro}, generalizing the  classical result on schemes \cite[Proposition 7.3]{Hartshorne}.

Indeed, setting $x_i'=1$ selects a unique $\bmu_r$ orbit in every $\GG_m$ orbit in $B_+$, so the map is bijective on points. Also, the stabilizer of any point on the orbit is contained in $\bmu_r$ and is invariant under $\GG_m$, hence the map is bijective on stabilizers.
  
We note that this discussion generalizes to varieties with torus actions, for instance the result of a composition of cobordant blow-ups. In all these cases, there exists a well-defined coarse moduli space, which is automatically a variety with quotient singularities, sometimes called a V-manifold. An open chart of the moduli space of $Bl_J(X)$ is simply the schematic quotient $W_i/\bmu_{w_i}$.

We now provide explicit equations for these objects. We describe the case $i=1$, the other cases are obtained by permutation.

Setting $x_1'=1 $ in the equation for $B_+$, we rename the restricted variables ${\bar s}, {\bar x_2}, \ldots, {\bar x_k}$\footnote{to allow comparison on $\tilde B_1$ below.} and obtain 
\begin{equation}
W_1 \ \ = \ \ \Spec_X \cO_X[{\bar s},{\bar x_2},\ldots, {\bar x_k}] / (x_1 - {\bar s}^{w_1},\ldots ,x_k - {\bar s}^{w_k} {\bar x_k}),
\end{equation}
which is indeed a smooth variety. 

The action of $\zeta \in \bmu_{w_1}$ is given by $$ \zeta\cdot({\bar s},{\bar x_2},\ldots,{\bar x_k}) \quad = \quad (\zeta^{-1} {\bar s} , \zeta^{w_2} {\bar x_2},\ldots, \zeta^{w_k} {\bar x_k}).$$

Next, the exceptional divisor of the stack theoretic blow-up $[B_{+}/\GG_m]\to X$ in the $x_1'$-chart $[B_1/\GG_m]$ is determined by the exceptional divisor $V(s)$ of the cobordant blow-up $B_+\to X$. Its restriction to $W_1$ still describes the exceptional divisor of the chart $[W_1 / \bmu_{w_1}]$.

 Recall that the structure sheaf of the geometric quotient of $W_1$ by $\bmu_{w_1}$ is given by the sheaf of $\GG_m$-invariant functions
\[
\cO_{W_1/\bmu_1} \simeq  (\cO_{W_1})^{\bmu_{w_1}}\simeq (\cO_{B_{1}})^{\GG_m}.
\]

We describe the inverse morphism $B_1 \to [W_1 / \bmu_{w_1}] \subset X'$
as follows. Write $\tilde B_1 = \Spec_{B_1} \cO_{B_1}[z]/(x_1' - z^{w_1})$,  which is a principal bundle $\bmu_{w_1}$-bundle through  $z \mapsto \zeta^{-1} z$ for $\zeta \in \bmu_{w_1}$. We map $\tilde B_1 \to W_1$ via ${\bar s} = sz$ and ${\bar x_i} = x_i' z^{-w_i}$, noting that $s^{w_i} x_i' = x_i = {\bar s}^{w_i} {\bar x_i}$ for $i=2,\ldots,k$.   One checks that the morphism is $\bmu_{w_1}$-equivariant, giving a cartesian diagram sitting atop Diagram \eqref{Eq:bigdiagram}, page \pageref{Eq:bigdiagram}:
$$\xymatrix{  \tilde B_1 \ar[r]^{\tilde q} \ar[d]_{\alpha} & W_1 \ar[d] \\  B \ar[r]^q & X'.}$$

The scheme $\tilde B_1$ also serves as a $\GG_m$-bundle over $W_1$ through $z \mapsto \xi z$ for $\xi \in \GG_m$, and $\tilde B_1 \to B_1$ is $\GG_m$-equivariant, inducing the map $W_1 \to [B_+/\GG_m] = X'$ with which we started.

\subsubsection{Transform of derivations in \'etale charts}\label{Sec:etale-charts} We now describe how to transform derivations after blowing up. Let  $\tau_1: W_1\to X$ be the $x_1'$-chart of the weighted blow-up computed above. Since $\tau_1$ is a generically finite morphism, we may define the total, controlled and strict transform of a foliation $\cF \subset \cD_X^{\log}$ following Definitions \ref{def:TotalTransformFoliation} and \ref{def:TransformFoliation}. 
In particular:
\begin{lemma}
\[
\begin{aligned}
\tau_1^{\ast}\left(\partial_{x_1}\right)&=\frac{1}{w_1 {\bar s}^{w_1}}\left( {\bar s}\partial_{{\bar s}} - \sum_{i=2}^k w_i \,{\bar x_i} \partial_{{\bar x_i}}\right) &&\\
\tau_1^{\ast}\left(\partial_{x_i}\right)&= \frac{1}{{\bar s}^{w_i}}\partial_{{\bar x_i}},
&\quad &i=2,\ldots,k\\
\tau_1^{\ast}\left(\partial_{x_i}\right)&= \partial_{x_i} , & \quad & i=k+1,\ldots,n,
\end{aligned}
\]
and 
\begin{equation}\label{eq:TransformDerivationsWeighted}
\begin{aligned}
\tau_1^c\left(\partial_{x_1}\right)=\tau_1^s\left(\partial_{x_1}\right)&=\frac{1}{w_1}\left( {\bar s}\partial_{{\bar s}} - \sum_{i=2}^k w_i \,{\bar x_i} \partial_{{\bar x_i}}\right) &&\\
\tau_1^c\left(\partial_{x_i}\right)=\tau_1^s\left(\partial_{x_i}\right)&=\partial_{{\bar x_i}},
&\quad &i=2,\ldots,n.
\end{aligned}
\end{equation}
\end{lemma}

\begin{proof} Equation \eqref{eq:TransformDerivationsWeighted} follows by cancelling denominators, so it suffices to prove the description of the total transforms. For this it suffices to check the action of the derivations $\partial_{i}$ on the right of the total transforms on $x_1,\ldots,x_n$. The key computations are the following:

First,  
\begin{align*} \partial_1({\bar s}^{w_1}) &= \frac{1}{w_1 {\bar s}^{w_1}}\left( {\bar s}\partial_{{\bar s}}({\bar s}^{w_1}) -\sum_{i=2}^k 0 \right) = 1 \\ 
\partial_1({\bar x_i}{\bar s}^{w_i}) &= \frac{1}{w_1 {\bar s}^{w_1}}\left( w_i {\bar x_i}{\bar s}^{w_i} -  w_i {\bar x_i}{\bar s}^{w_i}\right) = 0 \quad  i=2,\ldots,k.
\end{align*}
 Similarly, $$\partial_i({\bar x_i}{\bar s}^{w_i}) =1, \quad \partial_i({\bar s}^w_1) = \partial_i({\bar x_j}{\bar s}^{w_j}) = 0, \quad i,j=2,\ldots, k, \ \ i\neq j.$$ 
\end{proof}

These transforms on $W_1$ are $\bmu_{w_1}$-equivariant, giving rise to transforms on $X'$. We check that these are compatible with those defined through the split transforms on $B_1$  by comparing them on the \'etale cover $\tilde B_1$. By Lemma \ref{lem:LocalExpressionTransformDerivation} it suffices to compare $$x_1'\partial_{x_1'} = (1/w_1)z\partial_z \quad \text{and}\quad  x_i'\partial_{x_i'},\ i=2,\ldots,k$$ with the expressions resulting from {Equation \eqref{eq:TransformDerivationsWeighted}} above. We use the change-of-variables ${\bar s} = sz, {\bar x_i} = x_i'z^{-w_i}$ in the definition of the morphism.

We check that $\tilde q^*\tau^c(\partial_{x_i}) = z^{w_i} \partial_{x_i'}$. When $i=1$ we find that 
\begin{align*}
\tilde q^*\tau^c(\partial_{x_1}) (\bar s) &= sz/w_1&\quad& \partial_{x_1'}(\bar s) &=& \partial_{x_1'}(sz) &=& (sz) / (w_1 z^{w_1})  \\ 
  \tilde q^*\tau^c(\partial_{x_1}) (\bar x_j) &= -(w_jx_j')/(w_1z^{w_j}) &\quad & \partial_{x_1'}(\bar x_j) &=& \partial_{x_1'}(x_j'z^{-w_j}) &=& -(w_j x_j') / (w_1 z^{w_j} z^{w_1}),
 \end{align*} 
 and for $i=2,\ldots,k$ we find 
 \begin{align*}
\tilde q^*\tau^c(\partial_{x_i}) (\bar s) &= 0&\quad& \partial_{x_i'}(sz)  =0 && \\ 
  \tilde q^*\tau^c(\partial_{x_i}) (\bar x_j) &= \delta_{ij} &\quad & \partial_{x_1'}(x_j'z^{-w_j}) = \delta_{ij}z^{-w_j},&&
 \end{align*} 
 as needed.

\bibliographystyle{amsalpha}
\bibliography{principalization}

\end{document}